\documentclass[11pt,a4paper]{article}

\usepackage{amsmath} % assumes amsmath package installed
\usepackage{amssymb}  % assumes amsmath package installed
\usepackage{graphics}
\usepackage{graphicx}
\usepackage{epsfig}
\usepackage[ruled,vlined]{algorithm2e}
\usepackage{subfig}
\usepackage{amsfonts}
\usepackage{mathrsfs}
\usepackage{bbm}
\usepackage{color}
\usepackage{makeidx}
\usepackage{color}
\usepackage{amsfonts}%
\usepackage[]{algorithm2e}
\usepackage{amsthm}

\def\opi{{\mathrm{i}\mkern1mu}}
\def\veps{\varepsilon}
\def\eps{\epsilon}

\newcommand{\bC}{\mathbb{C}}

\newcommand{\real}{\operatorname*{Re}}
\newcommand{\imag}{\operatorname*{Im}}

\newtheorem{theorem}{Theorem}

\newtheorem{prop}{Proposition}
\newtheorem{assumpt}{Assumption}
\newtheorem{rmk}{Remark}
\makeindex

\title{Numerical inverse Laplace transform for convection-diffusion equations}
\author{Nicola Guglielmi\footnotemark[1] \and Mar\'ia L\'opez-Fern\'andez\footnotemark[2] \and Giancarlo Nino\footnotemark[3]}

\begin{document}

\maketitle
\renewcommand{\thefootnote}{\fnsymbol{footnote}}
\footnotetext[1]{Gran Sasso Science Institute,
                 via Crispi 7,
                 L'Aquila, Italy. Email: {\tt nicola.guglielmi@gssi.it}}
\footnotetext[2]{Dipartimento di Matematica,
Universit\`a di Roma La Sapienza,
Piazzale Aldo Moro 5,
I-00185  Roma,  Italy. Email: {\tt lopez@mat.uniroma1.it}}
\footnotetext[3]{Dipartimento di Ingegneria Scienze Informatiche e Matematica,
Universit\`a degli Studi di L'Aquila,
Via Vetoio - Loc.~Coppito,
I-67010 L' Aquila and Gran Sasso Science Institute, L'Aquila, Italy. Email: {\tt giancarlo.nino@gssi.it}}

\renewcommand{\thefootnote}{\arabic{footnote}}

\begin{abstract}
In this paper a novel contour integral method is proposed for linear convection-diffusion equations. The method is based on the inversion of the Laplace transform
and makes use of a contour given by an elliptic arc joined symmetrically to two
half-lines. The trapezoidal rule is the chosen integration method for the numerical
inversion of the Laplace transform, due to its well-known fast convergence properties
when applied to analytic functions. Error estimates are provided as well as careful
indications about the choice of several involved parameters. The method selects the
elliptic arc in the integration contour by an algorithmic strategy based on the computation of pseudospectral
level sets of the discretized differential operator. In this sense the method is general and can be applied to any linear convection-diffusion equation without knowing any \textit{a priori} information about its pseudospectral geometry.
Numerical experiments performed on the Black--Scholes ($1D$) and Heston ($2D$)
equations show that the method is competitive with other contour integral methods available in the
literature.
\end{abstract}

{\bf Keywords:} Contour integral methods, pseudospectra, Laplace transform, numerical inversion of Laplace transform,
trapezoidal rule, quadrature for analytic functions.


\pagestyle{myheadings}
\thispagestyle{plain}
\markboth{N.~Guglielmi, M.~L\'opez-Fern\'andez and G.~Nino}{Numerical inverse LT for convection-difussion equations}

\section{Introduction}
% \label{Intr}
We consider the time discretization of Initial Value Problems for linear systems of ODEs:
\begin{equation}\label{mainpb}
\frac{\partial u}{\partial t}=Au + b(t), \qquad u(0)=u_0,
\end{equation}
for $t>0$, $A$ a discrete version of an elliptic operator and $b$ a source term including possibly boundary contributions. The solution $u$ will thus be a time-dependent vector, of dimension equal to the number of degrees of freedom in the spatial semi-discretization of the reference problem. We are particularly interested in equations arising in mathematical finance, such as Black--Scholes, Heston or Heston-Hull-White equations \cite{BS,He,Hu}.

Classical methods to approximate the solution $u(t)$ to \eqref{mainpb} include Runge-Kutta and multistep integrators. Also splitting schemes, like ADI methods have been proposed to solve the continuous reference problem, see for example \cite{ITHF,ITH,ITHW,ITHW2} for Heston equation. All these methods are of time-stepping type and thus, in order to approximate the solution at a certain time $t_n$, approximations at certain smaller times
$0<t_1<t_2<\ldots<t_n$ must be previously computed. For large times or high accuracy requirements this procedure can be extremely demanding in terms of time and computational cost. An alternative to time-stepping methods to compute the solution at (few) given times, large or not, can be derived based on the Laplace transform and its numerical inversion. This approach has been successfully developed in  \cite{GaMa,LP,LPS,SST} for linear evolutionary problems governed by a {\em sectorial} operator, this is, assuming that $A$ in \eqref{mainpb} has bounded resolvent outside a certain acute sector in the left-half of the complex plane. The magnitude of the resolvent norm $\left\|\left(zI-A\right)^{-1}\right\|$ deeply impacts on the rate of convergence of the method. For this reason, the integration contour must be chosen accordingly to the pseudospectral geometry of $A$. In particular, if $A$ is non normal the pseudospectral geometry of $A$ can be difficult to estimate \cite{TreE}. The spatial discretization of a convection-diffusion operator typically leads to a non normal matrix \cite{ITHWeid,W}. In the present paper we propose a novel contour integral method for \eqref{mainpb}, which includes a preliminary study of the pseudospectral level curves of $A$.

In the present paper we assume that the Laplace transform of the source term $b(t)$ exists, is available, and admits a bounded  analytic extension to a big region of the complex plane outside the spectrum of $A$. The case of more general sources is out of the scope of the present manuscript but has been considered in the literature \cite{SST}. Under these hypotheses, we can apply the (unilateral) Laplace transform ${\cal L}: f\mapsto\hat{f}(z):=\int_{0}^{+\infty}e^{-zt}f(t)\,dt $ to both the sides of the system in \eqref{mainpb}, which leads to the following algebraic equation for $\hat{u}={\cal L}(u)$:
\begin{equation}\label{algEqLap}
\hat{u}(z)=\left(zI-A\right)^{-1}\left(u_0+\hat{b}(z)\right)\,,
\end{equation}where $\hat{b}={\cal L}(b)$ and $I$ stands for the identity matrix. The inversion formula for the Laplace transform provides the following representation of the unknown function $u$:
\begin{equation}\label{bromwich}
u(t)=\frac{1}{2\pi \opi}\int_{{\cal G}}e^{zt}\hat{u}(z)\,dz,
\end{equation}
where ${\cal G}$ is a deformation of a Bromwich contour, which can be taken as an open regular curve running from $-\opi\infty$ to $+\opi\infty$ and such that all singularities of $\hat{u}$ are to its left. Thus, ${\cal G}$ must leave to its left the eigenvalues of $A$ and all possible singularities of $\hat{b}$. The discretization of (\ref{bromwich}) by some quadrature rule will provide an approximation of $u(t)$, for a given $t$.
Assuming that the Laplace transform can be analytically extended to the left half of the complex plane and that this extension is properly bounded with respect to $z$, several authors have proposed different contour profiles and parametrizations for ${\cal G}$. Probably the first relevant related reference is \cite{T}, where the author analyzes a cotangent mapping with horizontal asymptotes and the classical trapezoidal rule for the inversion of scalar Laplace transforms. Much more recently, the cotangent contour have been investigated and improved \cite{DW}. Alternatively hyperbolic contours have been considered in \cite{GaMa,LP,LPS,SST}, with a focus on evolutionary problems governed by sectorial operators. In \cite{W} a parabolic profile for $\cal G$ is chosen. In all these references, the resulting scheme converges with spectral accuracy with different rates of convergence according to the particular application and the range of times at which the inverse Laplace transform is required. An application of the parabolic contour is studied in \cite{ITHWeid} to solve precisely Black--Scholes and Heston equations. These methods require in practice some \textit{a priori} knowledge of the pseudospectral geometry of $A$. In particular, in \cite{ITHWeid,W} a critical parabola bounding the resolvent norm of the (continuous) differential operator under consideration is used. This information is available only for few operators. For more complicated equations, such as Heston's equation, the critical parabola is only guessed. Also in \cite{LPS} a preliminary information about the pseudospectral behaviour is needed in order to efficiently set the sector of the complex plane outside which the resolvent norm is analytic and bounded by some suitable constant. In the present work we present a method that combines a preliminary numerical investigation of the pseudospectral level sets of $A$ with the efficient inversion of the Laplace transform. This feature makes the applicability of our method much wider.

Our error analysis will show that ${\cal G}$ in \eqref{bromwich} can be replaced by a finite open arc of an ellipse. In \cite{RT} the behaviour of the resolvent norm of some simple convection-diffusion operator is studied and it is shown that there exists a parabola containing the portion of the complex plane where this norm grows unboundedly, being this the starting point in \cite{ITHWeid,W}. But if a spatial semi-discretization is applied, the differential operator is approximated by a matrix $A$, which has a finite spectrum and closed curves surrounding the eigenvalues as pseudospectral level curves, see \cite[Figures 5, 6]{RT}. Moreover, since the exponential factor in (\ref{bromwich}) drags down the norm of the integrand function as $\real(z)$ moves to the left, we are interested in observing the pseudospectral behaviour of $A$ only in a vertical strip of the complex plane. For these reasons we found that an open arc of an ellipse is a good candidate to efficiently bound the resolvent norm of $A$ in a region of interest. Moreover, the choice of an elliptic profile has never been investigated in the literature, and its performance turns out to be competitive w.r.t. other methods. Our final selection of the integration profile relies both on the computation of some pseudospectral level curves associated to $A$ and on the optimization of the error bounds which we derive for our quadrature. Our method is strongly based on the theoretical error estimate of the trapezoidal rule when applied to exponentially decaying analytic integrands \cite{BLS,JavT,TW}.

Apart from being able to deal with more general problems than those considered in \cite{ITHWeid,LPS,W}, our new method enjoys the following important advantages:
\begin{itemize}
 \item it is able to approximate the solution of \eqref{mainpb} uniformly for $t$ belonging to large time windows.
 \item  round-off errors are controlled in a robust and systematic way.
 \item it is highly parallellizable.
 \item it provides an approximation of the solution $u$ to \eqref{mainpb} at a desired time (or time window) without computing any history.
 \item its performance is not affected by the lack of regularity (in space) of the initial data $u_0$ in \eqref{mainpb}.
 \item it provides an approximation within a prescribed target accuracy $tol$ by dynamically increasing the number of quadrature points on the integration contour, without changing the integration profile and taking advantage of previous computations.
 \end{itemize}

The article is organized as follows. In Section \ref{PrelAbTra} we review the error estimate for the trapezoidal rule when applied to exponentially decaying integrands. In Section \ref{ANwMet} our new method is fully described and analyzed. In Section \ref{practImpl} we provide details about the practical implementation and the computational cost. In Section \ref{sec:num}, after a first illustrative application of the method to a canonical convection-diffusion equation,  we test the method on Black-Scholes and Heston equations. We compare our new method with the methods in \cite{ITHWeid} and \cite{LPS} in Section \ref{secComp}. Finally, in Section \ref{timInt} we extend our algorithm to approximate the solution $u(t)$ for $t$ in a time window $[t_0,t_1]$ by using a unique integration contour.

\subsection*{The new integration contour}

We propose a contour $\mathcal{G}$ in \eqref{bromwich} which is the union of two half-lines connected with an open arc of an ellipse as shown in Figure~\ref{plot_profile}.

\begin{figure}[h] % \label{plot_profile}
\begin{center}
\includegraphics[scale=0.35]{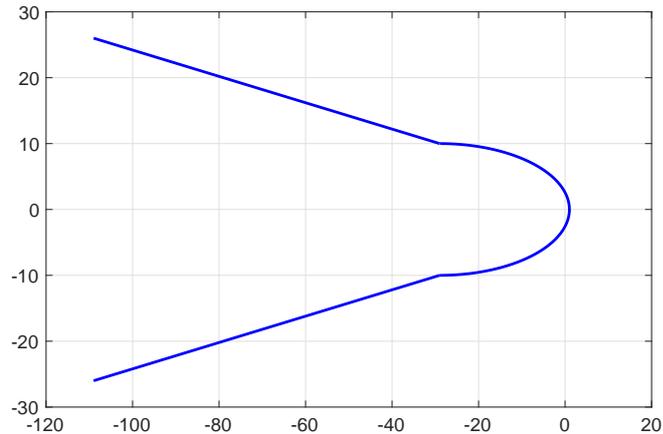}
\end{center}
\caption{Shape of the general integration profile ${\cal G}$. \label{plot_profile}}
\end{figure}

In particular ${\cal G}$ is defined as $\{ z(x) : x \in (-\infty,\infty)\}$, with
\begin{equation}\label{Gamma}
%	{\cal G}:
z(x)=\left\{\begin{array}{lr}
\ell_1(x)    & x\in\left[-\infty,-\frac{\pi}{2}\right] \\[1mm]
\Gamma(x),\,      & x\in\left[-\frac{\pi}{2},\frac{\pi}{2}\right] \\[1mm]
\ell_2(x)\,, & x\in\left[\frac{\pi}{2},+\infty\right], \\
\end{array}\right.
\end{equation}
where, for constant parameters $A_1, A_2, A_3$ to be determined,
$$
\Gamma(x) = A_1\cos x+\opi A_2\sin x+A_3
$$
parameterizes the elliptic arc and
$$ \ell_{1}(x) = A_3 + x+\frac{\pi}{2}-\opi\left(A_2 - d\left(x+\frac{\pi}{2}\right)\right),\quad
\ell_2(x) = A_3-x+\frac{\pi}{2}+\opi\left(A_2+d\left(x-\frac{\pi}{2}\right)\right)
$$
parameterize the half-lines.

{We recall that the constants in the parametrization must be chosen so that the resulting contour ${\cal G}$ leaves to its left the spectrum of $A$ and the singularities of $\hat{b}$. There is quite some freedom in the choice of the two half-lines, as along as their real part goes to minus infinity as $z\to \infty$. Actually, the contribution of the half-lines to the contour integral is expected to be small and will be neglected in practice, as we explain in Subsection \ref{subTruncErr}.}

%\begin{equation}\label{contrLines}
%	\left|\int_{-\infty}^{-\frac{\pi}{2}}e^{\left(A_3+w+\frac{\pi}{2}-i\left(A_2-d\left(w+\frac{\pi}{2}\right)\right)\right)t}\hat{u}(z(w))z'(w)\,dw\right|\leq
%\end{equation}
%\[
%	\leq K_{\ell}\int_{-\infty}^{0}e^{A_3t+w t}\,dw=K_{\ell}\frac{e^{A_3t}}{t}
%\]where $K_{\ell}$ is a bounding constant for $\hat{u}z'$. For $A_3$ negative enough, this contribution is negligible.

\section{Preliminaries about the trapezoidal rule}\label{PrelAbTra}
{We will apply the classical trapezoidal rule to approximate \eqref{bromwich}, after parametrization by \eqref{Gamma}. The resulting integral will be of the form}
\begin{equation}\label{integral}
	I=\int_{-\xi}^{\xi}F(x)\,dx,
\end{equation}with $F$ satisfying the properties listed in Assumption \ref{assumptF}.
\begin{assumpt}\label{assumptF}
The complex extension of the integrand function $F$ in \eqref{integral} satisfies the following properties:

For {some} $a>0$
\begin{enumerate}

\item $F(w)$ ($w \in \bC$) is analytic and bounded {inside} the strip $\left[-\pi/2,\pi/2\right]\times[-\opi a,\opi a]$;\\[.2em]

\item $\displaystyle{\left\|F(w)\right\|=\left\|F(-\overline{w})\right\|}$ inside the strip $\left[-\pi/2,\pi/2\right]\times[-\opi a,\opi a]$;
\\[.2em]

\item $\exists \eta_0>0 $ ($\eta_0<\xi$) and $\exists B_{\pm}>0$ such that $\forall \eta\leq \eta_0$ one has that
\[
	\left|F(x \pm \opi a)\right|\leq B_{\pm}\,,\quad \forall x \in[-\xi-\eta,\xi+\eta]\,.
\]

\item For the same $\eta_0$ of the previous point, $\exists S_{\pm}>0$ such that $\forall \eta\leq\eta_0$
\[
	\left|F(x \pm \opi a)\right|\leq S_{\pm}\,,\quad \forall x \in\left[-\frac{\pi}{2},-\xi-\eta\right]\cup\left[\xi+\eta,\frac{\pi}{2}\right]\,.
\]

\item $\forall x\in \mathbb{R}$ such that $\left|x\right|\geq \xi$, one has $\left|F(x)\right|\leq\theta$,  for a certain $\theta>0$.
\end{enumerate}
\end{assumpt}

Under these hypotheses it is possible to prove the following Theorem
\begin{theorem}\label{convTh}
{Consider the integral $I$ in \eqref{integral}, $N\ge 1$, and the discretization of $I$ by the quadrature formula}
\[
	I_N=\frac{2\xi}{N}\sum_{j=1}^{N-1}F(x_j)\, \quad \mbox{with } \ x_j=-\xi+ j \frac{2\xi }{N},\quad j=1,\ldots,N-1.
\]
Assume that $F$ satisfies Assumption~\ref{assumptF} and take $\eta=\xi/N$. Then
\begin{eqnarray}
\label{firstErr}
\left|I-I_N\right| & \leq & \frac{2(\xi+\eta)\left(B_++B_-\right)+
2\left(\pi/2-\delta-\xi-\eta\right)\left(S_++S_-\right)}{e^{\frac{a\pi N}{\xi}}-1}+\\
\nonumber
&& +4\left(\frac{\pi}{2}-\xi+\frac{\eta}{2}\right)\theta+4\frac{\xi\log2}{N\pi}\max_{w\in[-a,a]}\left|F(-\pi/2+\delta+iw)\right|,
\end{eqnarray}
with
\begin{equation}\label{delta}
\delta=\pi/2-(2k+1)\eta-\xi\,,\qquad k=\left\lfloor{\left(\frac{\pi}{2}-\xi-\eta\right)\frac{1}{2\eta}}\right\rfloor\,.
\end{equation}
\end{theorem}
\begin{proof}
{The proof is a variant of the one in \cite[Appendix]{BLS}}.

We consider the rectangle $R=[-\xi-\eta,\xi+\eta]\times [-\opi a,\opi a]$ and call $\Gamma_1$ the union of its horizontal sides, $\Gamma_2$ and $\Gamma_3$ its vertical left and right sides, respectively (see Figure~\ref{plotRect}).
\begin{figure}
\begin{center}
\includegraphics[scale=0.4]{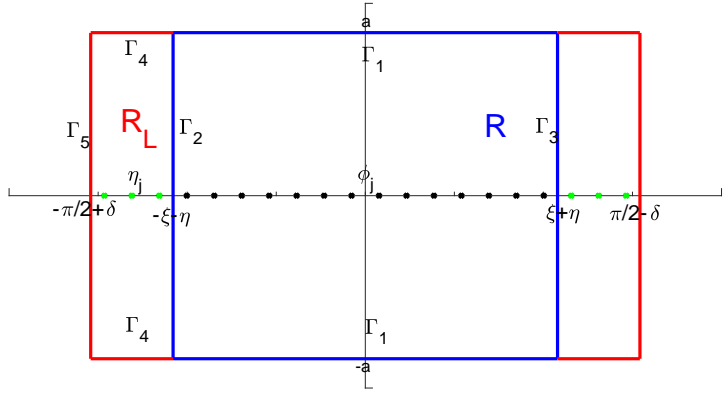}
\end{center}
\caption{The rectangle $R$ \label{plotRect}}
\end{figure}
Consider the integral
\[
	\tilde{I}=\int_{-\xi-\eta}^{\xi+\eta}F(x)\,dx\,.
\]
We have
\begin{equation}\label{estTr}
	\left|I-\tilde{I}\right|\leq 2\eta\theta\,.
\end{equation} {On the one hand, for} $w \in \bC$,
\[
	g(w)=\left\{\begin{array}{rr}
		-\frac{1}{2}\,, & \imag(w)>0\,,\\[.5em]
		\frac{1}{2}\,, &\imag(w)\leq 0\,,\\
	\end{array}	
	\right.
\]{it follows }
\[
	\tilde{I}=\int_{\partial R}g(w)F(w)\,dw\,.
\]On the other hand, define
\[
	\tilde{I}_N=\frac{2\xi}{N}\sum_{j=0}^NF(x_j)\,,\quad x_j=-\xi+\frac{2\xi j}{N}\,, j=0,\ldots,N\,.
\]One has
\[
	\left|I_N-\tilde{I}_N\right|\leq \frac{4\xi}{N}\theta=4\eta\theta\,.
\] The function
\[
	m(w):=\frac{1}{2}\frac{1+e^{-\opi \frac{(w+\xi) N\pi}{\xi}}}{1-e^{-\opi \frac{(w+\xi) N\pi}{\xi}}}
\] satisfies by the residue Theorem
\[
	\tilde{I}_N=\int_{\partial R}m(w)F(w)\,dw\,.
\]Observe that
\[
	err_N=\left|I-I_N\right|\leq\left|I-\tilde{I}\right|+\left|\tilde{I}-\tilde{I}_N\right|+\left|\tilde{I}_N-I_N\right|
\]Let us consider
\[
	\tilde{I}-\tilde{I}_N=\int_{\partial R}\left(g(w)-m(w)\right)F(w)\,dw=
	\left(\int_{\Gamma_1}+\int_{\Gamma_3}-\int_{\Gamma_2}\right)\left[\left(g(w)-m(w)\right)F(w)\right]\,dw\,.
\]
We estimate
\begin{equation}\label{estH}
	\left|\int_{\Gamma_1}\left(g(w)-m(w)\right)F(w)\,dw\right|\leq \frac{2(\xi+\eta)\left(B_++B_-\right)}{e^{\frac{a\pi N}{\xi}}-1}\,.
\end{equation}
We estimate the integrals over $\Gamma_{2}, \Gamma_3$. % Consider the one on $\Gamma_2$.
Take the rectangle $R^L=\left[-\pi/2+\delta,-\xi-\eta\right]\times[-\opi a,\opi a]$ and call $\Gamma_5$ its left vertical side and $\Gamma_4$ the union of its horizontal sides. {We define $\delta$ as in the statement of the Theorem, so that $\left[-\pi/2+\delta,-\xi-\eta\right]$ is the largest segment with length an even multiple of $\eta$}. We have that
\[
	err_N^L:=\int_{-\frac{\pi}{2}+\delta}^{-\xi-\eta}F(x)\,dx-\frac{2\xi}{N}\sum_{j=0}^{k}F(\eta_j)=\int_{\partial R_L}\left(g(w)-m(w)\right)F(w)\,dw\,,
\]
with $\eta_j=-\xi-\frac{2\xi j}{N}$, $j=1,\ldots,k$. In other words $err_N^L$ is the error of the same trapezoidal quadrature rule applied on the integral on the interval
$[-\pi/2+\delta,-\xi-\eta]$ with the same spacing $2\xi/N$. In this way, we get
\[
	\int_{\Gamma_{2}}\left(g(w)-m(w)\right)F(w)\,dw=err_N^L+\left(-\int_{\Gamma_4}+\int_{\Gamma_5}\right)\left[\left(g(w)-m(w)\right)F(w)\,dw\right]\,.
\]
On the one hand, we estimate (with $w=x+ \opi y$)
\begin{equation}\label{estHL}
\left|\int_{\Gamma_4}\left(g(w)-m(w)\right)F(w)\,dw\right|\leq\frac{\left(\pi/2-\delta-\xi-\eta\right)\left(S_++S_-\right)}{e^\frac{a\pi N}{\xi}-1}\,.
\end{equation}On the other hand, we estimate
\[
\left|\int_{\Gamma_5}\left(g(w)-m(w)\right)F(w)\,dw\right|\leq \int_{-a}^a\left|\left(g(-\pi/2+\delta+\opi y)-m(-\pi/2+\delta+\opi y)\right))
F(-\pi/2+\delta+\opi y)\,dy\right|\leq
\]
\[
	\leq\max_{y\in[-a,a]}\left|F(-\pi/2+\delta+\opi y)\right|\int_{-a}^a\left|g(-\pi/2+\delta+\opi y)-m(-\pi/2+\delta+\opi y)\right|\,dy
\]and
\[
	\int_{-a}^a\left|g(-\pi/2+\delta+\opi y)-m(-\pi/2+\delta+\opi y)\right|\,dy= 2\int_{0}^{a}\frac{1}{1+e^{\frac{\pi w N}{\xi}}}\,dy
	\leq
\]
\[
	\leq 2\int_{0}^{+\infty}\frac{1}{1+e^{\frac{\pi w N}{\xi}}}\,dy=\frac{2\log 2}{\pi}\frac{\xi}{N}\,,
\]
so that
\begin{equation}\label{estD}
\left|\int_{\Gamma_5}\left(g(w)-m(w)\right)F(w)\,dw\right|\leq\max_{y \in[-a,a]}\left|F(-\pi/2+\delta+\opi y)\right|\frac{2\log 2}{\pi}\frac{\xi}{N}\,.
\end{equation}
{We finally obtain}
\begin{equation}\label{estErrL}
	\left|err_N^L\right|=\left|\int_{-\frac{\pi}{2}+\delta}^{-\xi-\eta}F(x)\,dx-\frac{2\xi}{N}\sum_{j=1}^{k}F(\eta_j)\right|\leq
	\int_{-\frac{\pi}{2}+\delta}^{-\xi-\eta}\left|F(x)\right|\,dx+\frac{2\xi}{N}\sum_{j=1}^{k}\left|F(\eta_j)\right|\leq
	2\theta\left(\frac{\pi}{2}-\xi-\eta\right)\,.
\end{equation}

{The integral over $\Gamma_3$ can be estimated in an analogous way. The combination of estimates \eqref{estTr}, \eqref{estH}, \eqref{estHL}, \eqref{estD} and \eqref{estErrL} yields the stated result}.
\end{proof}

\section{A new method}\label{ANwMet}
The only portion of the integration contour ${\cal G}$ as defined in \eqref{Gamma} that we will use in practice is the half ellipse. As already done in \cite{LP,LPS,ITHWeid}, this elliptical profile is selected through the construction of a suitable conformal mapping $z:\left[-\pi,\pi\right]\times[-\opi a,\opi a]\rightarrow \mathbb{C}$. In particular we want it to map horizontal segments onto ellipses in the complex plane (we will use the right half of these ellipses). Thus we set
\begin{equation}\label{mapping}
	z(x+iy)=A_1(y)\cos x+iA_2(y)\sin x+A_3(y)\,.
\end{equation}
\begin{figure}[h!]
\includegraphics[scale=0.3]{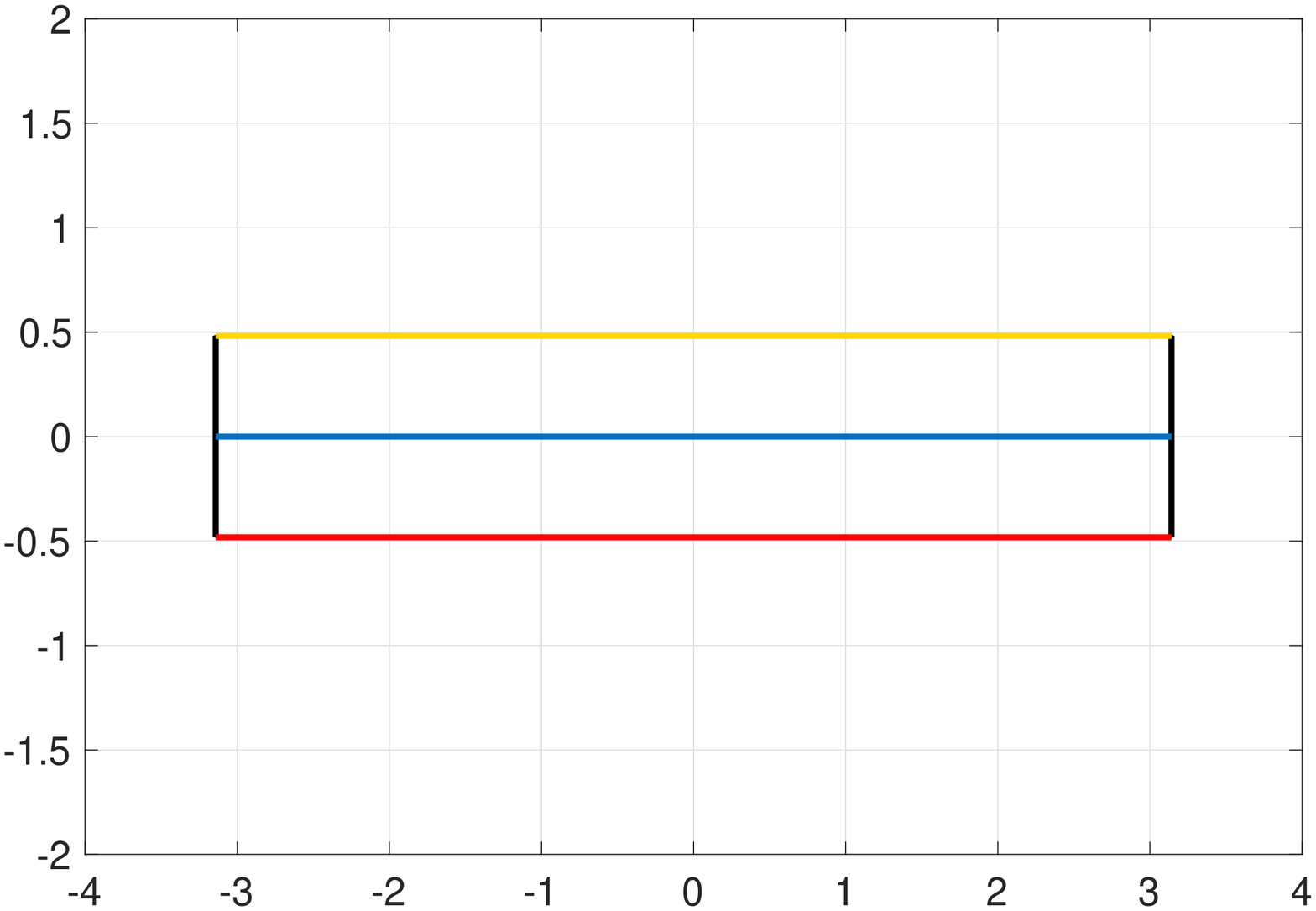}\includegraphics[scale=0.3]{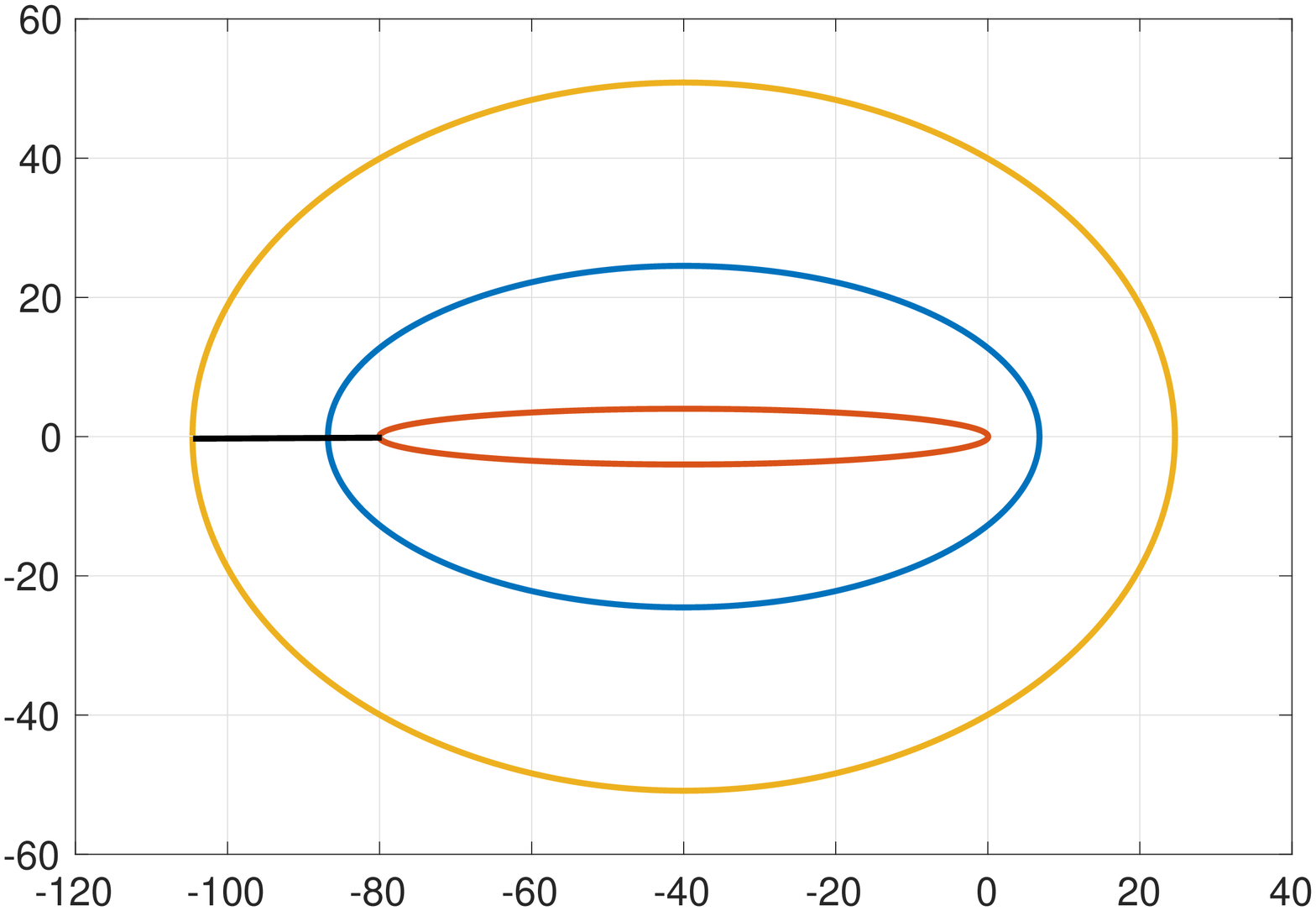}
\caption{Action of the conformal mapping $z(w)$ : it transforms horizontal segments in the complex plane (left) into ellipses (right).}
\end{figure}We ask \eqref{mapping} to be holomorphic and impose the Cauchy-Riemann equations, getting that $A_3$ has to be constant and
\begin{eqnarray}
\label{A1}
A_1(y)=a_1e^y+a_2e^{-y}\\
\label{A2}
A_2(y)=a_2e^{-y}-a_1e^y\nonumber
\end{eqnarray}  where $a_1,a_2$ are constants. The resulting mapping turns out to be entire, since the partial derivatives % of the first order
of both the real and the imaginary parts are everywhere continuous. It is thus invertible and the complex derivative cannot be zero.

The core idea of using the mapping \eqref{mapping} is to let $y$ vary in the segment $[-a,a]$, for a suitable $a$, while the integrand function
\begin{equation}\label{integrandFunction}
G(w)=e^{z(w)t}\left(z(w)I-A\right)^{-1}\left(u_0+\hat{b}(z(w))\right)z'(w)
\end{equation}stays bounded for every $w=x+\opi y$, $x\in[-\pi/2,\pi/2]$, $y\in [-a,a]$. In particular, we will use the mapping \eqref{mapping} to efficiently bound the exponential term $e^{z(w)t}$ and the norm of the resolvent $\left(z(w)I-A\right)^{-1}$. We are doing this by constructing two \textit{external} half ellipses $\Gamma_+,\Gamma_-$ delimiting the part of the complex plane where the integrand is bounded. The mapping \eqref{mapping} is asked to map the segments $[-\pi/2,\pi/2]\times\left[-\opi a, \opi a\right]$ onto this two curves. Then, we will use as integration profile $\Gamma$ the one corresponding to the image of the real segment $[-\pi/2,\pi/2]$. In formulas, the actual profile of integration that we use is parameterized as
\begin{equation}\label{gammaActual}
\Gamma: x\mapsto (a_1+a_2)\cos x+\opi(a_2-a_1)\sin x+A_3, \quad x\in \left[-\frac{\pi}{2}, \frac{\pi}{2} \right].
\end{equation}Let us focus first on the resolvent. We will construct an ellipse $\tilde{\Gamma}$ in the complex plane such that the resolvent norm is bounded on its right half. In particular, $\tilde{\Gamma}$ satisfies the following properties
\begin{assumpt}\label{assumptionG}
We assume that an ellipse $\tilde{\Gamma}$ of center $z_l\in\mathbb{R}$ and right intersection with the real axis $z_r$ is given in such a way that:\begin{enumerate}
\item its right half, that from now on we call $\Gamma_+$ leaves to its left the spectrum of $A$, $\sigma(A)$, and the set of the singularities of the function $\hat{b}$;
\item $e^{z_lt}<\eps$, $\eps$ being the working precision; % (for instance $e^{z_lt}=10^{-18}$);
\item there exists $R_+$ s. t. $\left\|\left(zI-A\right)^{-1}\right\|\leq R_+$ for all $z\in\Gamma_+$; % (for example we can take $R_+=10^{-13}$);
\item there exists $W_+$ s. t. $e^{\real(z)t}\left\|\left(zI-A\right)^{-1}\right\|\leq W_+$, for all $z\in \Gamma_+$. % (as example $W_+=10^{-9}$).
\end{enumerate}
\end{assumpt}

For instance we can choose $e^{z_lt}=10^{-18}$, $R_+=10^{13}$ and $W_+=10^{9}$.

 $\Gamma_+$ is uniquely defined by $z_l,z_r$ and a further point $p \in\Gamma_+$, $p\neq z_r$. The construction of $\Gamma_+$, as explained in Subsection \ref{constrEll}, takes as input parameters the values $z_l,z_r$ and returns a value $p$, which is computed accordingly to the pseudospectral geometry of the operator $A$.

In general, an ellipse satisfying Assumption~\ref{assumptionG} for a given operator $A$ and vector $\hat{b}$ is not known \textit{a priori}. Assume that a closed curve ${\cal C}$ surrounding the spectrum of $A$ and the possible singularities of $\hat{b}$ is known. Moreover , we assume that ${\cal C}$ encloses the portion of the complex plane where the resolvent norm $\left\|\left(zI-A\right)^{-1}\right\|$ is large. Based on $\cal C$, a  general algorithmic strategy for computing numerically $\Gamma_+$ is proposed in Section \ref{constrEll}. Since for a general matrix $A$ we do not have any information about its pseudospectral geometry, we approximate $\cal C$ by using \texttt{eigtool} \cite{eigtool}. The construction is general and it does not require any {\em a priori} information about the spectral behaviour of $A$. Moreover, for the applications under consideration we show that a low resolution when approximating $\cal C$ might be enough for our purposes, see Section \ref{compCost}. Notice that, if more information about the pseudospectral behaviour of $A$ is available, the use of \texttt{eigtool} can be avoided.

Once $\Gamma_+$ is given, we want \eqref{mapping} to map the segment $[-\pi/2 + \opi a,\pi/2 + \opi a]$, for an $a$ to be fixed, onto $\Gamma_+$.
Imposing the ellipse $z(\cdot+\opi a)$ to be centered at $z_l$ and to pass through the points $z_r,\, d+\opi r $, we get
\begin{equation}\label{eq1}
a_1e^a+a_2e^{-a}=z_r-z_l,\quad
a_2e^{-a}-a_1e^a=\frac{r}{\sin\tilde{w}},\quad
A_3=z_l,
\end{equation}
where
\begin{equation*} % \label{phi}
	\tilde{w}=\arccos\left(\frac{d-z_l}{z_r-z_l}\right)\,.
\end{equation*}
 Solving \eqref{eq1} for $a_1,a_2, A_3$ we get
\begin{eqnarray}\label{a1}
a_1&=&\frac{e^{-a}}{2}\left(z_r-z_l-\frac{r}{\sin\tilde{w}}\right),\\\label{a2}
a_2&=&\frac{e^{a}}{2}\left(z_r-z_l+\frac{r}{\sin\tilde{w}}\right),\\\label{C}
A_3&=&z_l,\nonumber
\end{eqnarray}
which only depend on the real parameter $a$.
%Assume to fix $a$ such that
%\begin{equation}\label{bound1}
%\int_{-\frac{\pi}{2}}^{\frac{\pi}{2}}\left|e^{z(w\pm ia)t}\left(z(w\pm ia)I-A\right)^{-1}(u_0+\hat{b}\left(z(w\pm ia)\right))z'(w\pm %ia)\right|\,dw\leq M_{\pm}
%\end{equation}
%where $M_+,M_-$ are bounding constants.
\subsection{Quadrature error estimates for the new integration contour}\label{ChTarget}

Assume we are interested in approximating the unknown function $u$ up to a certain precision that we call $tol$. Because of the presence of the exponential, we expect the integrand function $G$ \eqref{integrandFunction} to become smaller in modulus as $z$ moves from the right to the left on the profile of integration. For this reason, we would like to truncate the integral once the function $|G|$ reaches the value $tol$. The motivation for fixing $z_l$ as in point (2) of Assumption~\ref{assumptionG} is that we are assuming that the center of the half ellipse $\Gamma_+$ (that is also the center of the integration ellipse $\Gamma$) is negative enough to make the integrand function close to the working precision at $z(\pi/2)$. In this way, for every $tol$ greater than the working precision, we can efficiently use the half ellipse of integration $\Gamma$ in order to recover approximations of the solution of order $tol$. In practice, we assume that there exists a \textit{truncation parameter} $c\in]0,1/2[$, defined as
\begin{equation}\label{c1}
\left|G(c\pi)\right|=tol\,.
\end{equation}
The integrand we want to approximate is
\[
I=\int_{-c\pi}^{c\pi}G(x)\,dx= u(t)+{\cal O}(tol)
\]
In other words, we are neglecting not only the two half-lines of profile \eqref{Gamma}, but also the contribution to the integral coming from the portions of ellipse parametrized in the intervals $\left[-\pi/2,-c\pi\right]$ and $\left[c\pi,\pi/2\right]$ as on these intervals the modulus of the integrand function is expected to be lower than the precision $tol$, because of the rapidly decaying behaviour of the exponential.
Applying trapezoidal quadrature rule to $I$, we get the sum
\begin{equation}\nonumber
I_N=\frac{c}{\opi N}\sum_{j=1}^{N-1} e^{z(x_j)t}\left(z(x_j)I-A\right)^{-1}\left(u_0+\hat{b}(z(x_j))\right)z'(x_j)\,,\quad x_j=-c\pi+j\frac{2c\pi}{N}.
\end{equation}
We remark that, since the profile of integration is symmetric w.r.t. the real axis, the quadrature sum can be simplified to
\begin{equation}\label{quadrature2}
I_N=\frac{2c}{N}\imag \left(\sum_{j=\left \lceil{\frac{N}{2}}\right \rceil}^{N-1} e^{z(x_j)t}\left(z(x_j)I-A\right)^{-1}\left(u_0+\hat{b}(z(x_j))\right)z'(x_j)\right), \quad x_j=-c\pi+j\frac{2c\pi}{N}.
\end{equation}
From Theorem~\ref{convTh} we get the following result.
\begin{theorem} \label{th:errquad_ellipse}
Assume that the function \eqref{integrandFunction} is analytic and bounded on the rectangle $\left[-\frac{\pi}{2},\frac{\pi}{2}\right]\times[-\opi a,\opi a]$ for a certain $a>0$, with $z(w)=\left(a_1+a_2\right)\cos w+\opi\left(a_2-a_1\right)\sin w+A_3$ and $a_1,a_2,A_3$ given by \eqref{a1}, \eqref{a2}, \eqref{C}. Assume moreover that the ellipse of integration has foci on the real axis. Set
\begin{eqnarray}\label{MP}
M_{+}=\frac{1}{2\pi}\max_{x\in[-\pi/2,\pi/2]}\left|e^{z(x+ \opi a)t}\left(z(x+ \opi a)I-A\right)^{-1}\left(u_0+\hat{b}\left(z(x+ \opi a)\right)\right)z'(x+ \opi a)\right|
\\
\label{MM}
M_{-}=\frac{1}{2\pi}\max_{|x|\in[0,c\pi+c\pi/N]}\left|e^{z(x- \opi a)t}\left(z(x- \opi a)I-A\right)^{-1}\left(u_0+\hat{b}\left(z(x- \opi a)\right)\right)z'(x- \opi a)\right|
\\
\label{SM}
S_{-}=\frac{1}{2\pi}\max_{\left|x\right|\in[c\pi+c\pi/N,\pi/2]}\left|e^{z(x- \opi a)t}\left(z(x- \opi a)I-A\right)^{-1}\left(u_0+\hat{b}\left(z(x- \opi a)\right)\right)z'(x- \opi a)\right|
\end{eqnarray}
Finally, we assume that the integrand function $\left|G(w)\right|\leq tol$ for all $w\in[-\pi/2,c\pi]\cup[c\pi,\pi/2]$.
Then the quadrature error can be estimated by
\begin{eqnarray}\nonumber
err_N :=\left|I-I_N\right|\leq\frac{2\pi c\left(1+\frac{1}{N}\right)M_-+2(\pi/2-c\pi-c\pi/N)S_-+\pi M_+}{e^{\frac{a}{c}N}-1} + \\\label{estimateErr}
+4\left(\frac{\pi}{2}-c\pi+\frac{c\pi}{2N}\right)tol+\Delta\frac{2c\log 2}{N\pi}e^{\left(\left(a_1e^{-a}+a_2e^a\right)\cos\left(\pi/2-\delta\right)+A_3\right)t}
\end{eqnarray}
with $\left\|\hat{u}\left(z(\pi-\delta\pm\opi a)z'(\pi/2-\delta\pm\opi a)\right)\right\|\leq\Delta$ and $\delta$ given by \eqref{delta} (with $\xi=c\pi$).
\end{theorem}
\begin{proof}
The result follows immediately from Theorem \ref{convTh} with $\xi=c\pi$. Following the notation of Theorem \ref{convTh}, we set $B_+=S_+=M_+$ and $B_-=M_-$. In particular, we observe that $F(w)=e^{z(w)t}\hat{u}(z(w))z'(w)$ satisfies Assumption \ref{assumptF}. Now, let us estimate $\left|F(\pi/2-\delta+iy)\right|$ for $y\in[-a,a]$. We have
\[
	\left|e^{z\left(\pi/2-\delta+\opi y\right)t}\hat{u}\left(z\left(\pi/2-\delta+\opi y\right)\right)z'\left(\pi/2-\delta+\opi y\right)\right|\leq \Delta e^{\left(\left(a_1e^{y}+a_2e^{-y}\right)\cos\left(\pi/2-\delta\right)+A_3\right)t}
\]
We observe that, from the hypothesis that both the foci are real, the horizontal semi-axis of the ellipse is longer than the vertical one, and so $a_1+a_2>a_2-a_1$. Then $a_1$ is positive and so is $a_2$ (as from its definition in \eqref{a2}). So, it is straightforward to prove that the maximum of the exponential is attained for $y=-a$. Indeed consider the function $f(y)=a_1e^{y}+a_2e^{-y}$. Its derivative is $f'(y)=a_1e^{y}-a_2e^{-y}$ and it is positive if and only if
\[
	e^{2y}>\frac{a_2}{a_1}\,.
\]Recalling \eqref{a1}, \eqref{a2} this reads as
\[
	e^{2y}>Ce^{2a}
\]yhere $C=\frac{z_r-z_l+r/\sin\tilde{y}}{z_r-z_l-r/\sin\tilde{y}}>1$. Then, $f$ is increasing if and only if $y>a+\frac{\log C}{2}$ and, since $y\in[-a,a]$, $f$ attains its maximum for $y=-a$.
\end{proof}

\begin{rmk}\label{rmkB} In considering the term
\begin{equation}\label{B}
B=\frac{2c\log 2}{N\pi}e^{\left(\left(a_1e^{-a}+a_2e^a\right)\cos\left(\pi/2-\delta\right)+A_3\right)t},
\end{equation}
we have that $\delta\leq c\pi/N$ and then $\cos(\pi/2-\delta)\rightarrow 0$ for $N$ growing. Moreover, as from equation \eqref{C}, $A_3$ is chosen in order to have $e^{A_3t}$ smaller than the working precision. In the end, we expect $B$ to be much smaller than $tol$, at least for $N$ big enough. In practice, $B$ is very small (and negligible) also for very small values of $N$. To show this, we report in the following tables the size of the error B \eqref{B} that we observe in the numerical experiments displayed in Section \ref{sec:num}. The term $B$, which is computed for $N=5$, is much smaller than the accuracy $tol$ in every case.
\[	\begin{array}{c}
	\mbox{\textbf{Black-Scholes}}\\
	\begin{array}{|c|c|c|c|}\hline
		& t=1& & t=10\\\hline
		tol=5e-3 &1.1431e-18 &tol=5e-2& 1.5555e-17\\\hline
		tol=5e-6 &4.8763e-18 &tol=5e-4 &3.8418e-17\\\hline
		tol=5e-9 &4.1133e-14 & tol=5e-6 & 3.4109e-14\\\hline
		tol=5e-11 &1.6566e-17 &tol=5e-9 &2.0858e-14\\\hline
	\end{array}
	\end{array}
\]
\[
	\begin{array}{c}
	\mbox{\textbf{Heston}}\\
	\begin{array}{|c|c|c|c|}\hline
		& t=1& & t=10\\\hline
		tol=5e-2 &1.5409e-16 &tol=5e-2& 2.8627e-16\\\hline
		tol=5e-4 &3.3070e-16 &tol=5e-4 &5.7855e-14\\\hline
		tol=5e-6 &2.0416e-16 & tol=5e-5 & 1.3791e-16\\\hline
		tol=5e-8 &1.5943e-14 &tol=5e-6 &3.7098e-12\\\hline
	\end{array}
	\end{array}
\]
\end{rmk}

\begin{rmk}
% It is possible to select a positive $\nu$ and consider the trapezoidal rule with a quadrature step $(\pi+2\nu)/M$ for $M$ a natural number, and to change the proof of Theorem 1 in order to make $B$ as defined in \eqref{B} the evaluation of the modulus of the integrand function along the segments $\left\{\pm\pi/2\right\}\times[-\opi a,\opi a]$.
%\red{\bf Io non capisco la frase sopra, inoltre non risulta chiaro quali sono i segmenti}
%\red{\bf Proporrei per esempio:}
It is possible to select a positive $\nu$ and consider the trapezoidal rule with a quadrature step $(\pi+2\nu)/M$ for some natural number $M$, and to change the proof of Theorem 1 in order to make $B$ as defined in \eqref{B}, by evaluating the modulus of the integrand function along boundary of the rectangle $[-\pi/2, \pi/2] \times [-\opi a,\opi a]$.

In this way, the double exponential appearing in \eqref{B} vanishes because of the multiplication by $\cos(\pi/2)$ and the term $B$ is even smaller. This change will made the length of the interval used to select the quadrature step dependent on the number of the nodes and so, it won't be possible any more to double this number saving the information computed on the previous nodes.
\end{rmk}

\subsection{Selection of the optimal integration contour}\label{innerConstr}
Once $a$ is fixed, the profile of integration $\Gamma$ is uniquely defined by equations \eqref{a1}, \eqref{a2}, \eqref{C}. To complete our construction, we need to ask the exponential part in the integrand $G$ to be bounded on the external half ellipse $\Gamma_-=z\left([-\pi/2,\pi/2]\times\left\{-\opi a\right\}\right)$. We have
\[
	e^{\real(\zeta)t}\leq e^{\real\left(z(-\opi a)\right)}\,,\quad \forall\zeta\in\Gamma_-\,.
\]Then, setting $z(-\opi a)=D$, being $D>0$, we get that
\[
	e^{zt}\leq e^{Dt}\,,\quad \forall z\in \Gamma_-\,.
\]Using \eqref{mapping}, \eqref{A1}, \eqref{A2}, we get
\begin{equation}\label{eq4}
a_1e^{-a}+a_2e^{a}+A_3=D
\end{equation}
and, recalling equations \eqref{a1}, \eqref{a2}, \eqref{C}, we get
\begin{equation}\label{D}
D=\frac{e^{-2a}}{2}\left(z_r-z_l-\frac{r}{\sin\tilde{w}}\right)+\frac{e^{2a}}{2}\left(z_r-z_l+\frac{r}{\sin\tilde{w}}\right)+z_l.
\end{equation}
The  term $D$ is made dependent on $a$ so that it is not fixed \textit{a priori} but it results from the optimization process of the parameter $a$.
The construction summarized by the equations \eqref{a1}, \eqref{a2}, \eqref{C}, \eqref{D} is still theoretical. In order to have it working, we need to find a parameter $a$ (and consequently the truncation parameter $c$ as defined in \eqref{c1}) giving us the actual rate of convergence of the quadrature rule. We notice that the profile of integration $z$ is given once $a$ is fixed by formulas \eqref{a1}, \eqref{a2}.
The tool we use to make the selection is the estimate \eqref{estimateErr}.
In order to simplify this estimate, we identify a leading term and neglect all the remaining ones, whose contribution is expected to be smaller. First of all, we notice that we expect $M_+<M_-$ (defined in \eqref{MP}, \eqref{MM}): this is due to the rapid decaying property of the exponential part in the integrand function and, in practice, it holds in most cases. Moreover, we assume $S_-\ll M_-$ (where $S_-$ is the one in \eqref{SM}), since the exponential $e^{zt}$ is larger for $z$ parametrized in $[-c\pi,c\pi]$ than the one in $[-\pi/2,-c\pi]\cup [c\pi,\pi/2]$. In the end, we neglect the contribution of the term of order $tol$ ($2\left(\frac{\pi}{2}-c\pi+\frac{2c\pi}{N}\right)tol$), and the one of the term $B$ \eqref{B} that is expected to be small (see Remark \ref{rmkB}). Finally we simplify the error estimate \eqref{estimateErr} to % the formula
\begin{equation}\label{estimateErr2}
	err_N\approx\frac{2\pi cM_-}{e^{\frac{a}{c}N}-1}\,.
\end{equation} Recalling \eqref{MM}, we roughly estimate % $M_-$ as
% \begin{equation*} % \label{mm}
$M_-\approx e^{Dt}$.
% \end{equation*}
Since we are interested in the order of magnitude of $M_-$ the estimate is enough precise for our purposes. In the end, within the accuracy, the quadrature error is % approximately
\begin{equation}\label{errOpt}
	err_N \lesssim 2\pi c e^{Dt-\frac{a}{c}N}\,.
\end{equation}
Assuming we seek approximation of $u(t)$ with a prescribed precision $tol$, we impose \eqref{errOpt} to equal the precision. Solving for $N$, we compute
\begin{equation}\label{NN}
		N=\frac{c}{a}\left(Dt-\log\left(\frac{tol}{2\pi c}\right)\right)
\end{equation}as the (theoretical) minimum number of quadrature nodes letting us reach the fixed accuracy.
We aim to minimize $N$ in order to make the convergence as fast as possible. The minimization of the function in \eqref{NN} appears to be technically complicated since it depends on two variables ($a,c$) and the truncation parameter $c$ depends on the profile of integration by the nonlinear constraint \eqref{c1}. For this reason, recalling that $c\leq 1/2$, we estimate
\begin{equation}\label{NN2}
N=\frac{c}{a}\left(Dt-\log\left(\frac{tol}{2\pi c}\right)\right)\leq \frac{1}{2a}\left(Dt-\log\left(\frac{tol}{\pi}\right)\right)\,.
\end{equation}The best that we can do to minimize $N$ is to minimize the function
\begin{equation}\label{fToMin}
f(a)=\frac{1}{2a}\left(Dt-\log\left(\frac{tol}{\pi}\right)\right)\,.
\end{equation}We fix an upper bound $a_M$ and we seek minimizers of $f$ in the interval $[0,a_M]$. In the numerical experiments we use $a_M=1$. Once the minimizer value $a$ has been computed, it defines uniquely the profile of integration by formulas \eqref{gammaActual}, \eqref{a1}, \eqref{a2}, \eqref{C}. Recall that $D$ is a function of $a$ by formula \eqref{D}.

\subsection{Truncation error}\label{subTruncErr}
The profile \eqref{Gamma} that we use to apply the Bromwich inversion formula is never used in practice: we just use a ``small" portion of the ellipse contained in it. How much this portion is ``small" depends on the fixed precision $tol$. In practice, we disregard all the points of \eqref{Gamma} whose contribution is estimated to be smaller than $tol$. We resume the global error analysis in the following
\begin{theorem}\label{th:totalerr}
Consider the integration profile defined in \eqref{Gamma}, where the elliptical part $\Gamma$ is constructed as in Section~\ref{innerConstr}. We denote $\ell_1,\ell_2$ as in \eqref{Gamma}, the parametrization of the two half lines making part of ${\cal G}$, whose slope is assumed to be such that ${\cal G}$ encloses all the singularities of the integrand function. Moreover, assume that
\[
	\left\|\hat{u}(z(x))z'(x)\right\|\leq K_{\ell}\,,\quad \mbox{ for } z(x)=\ell_{1}(x) \ \mbox{ and }\ z(x)=\ell_{2}(x).
\]
In the end, we assume that all the hypotheses of Theorem \ref{th:errquad_ellipse} are satisfied. Then, the total error of our approximation of \eqref{bromwich} is given by
 \[
	\left|u(t)-I_N\right|\leq err_N+err_{T}
 \]
 where $err_N$ is bounded by \eqref{estimateErr} and the truncation error $err_T$ is bounded by
 \[
	\left|err_{T} \right| \le K_{\ell}\frac{e^{z_lt}}{\pi t} +\left(\frac{1}{2}-c\right)tol
 \]
and does not depend on $N$.
\end{theorem}
\begin{proof}
Recall that we are approximating the following integral
\[
	\frac{1}{2\pi \opi}\int_{{\cal G}}e^{zt}\hat{u}(z)\,dz = \frac{1}{2\pi\opi}\int_{-\infty}^{\infty}e^{z(x)t}\hat{u}(z(x)) z'(x)\,dx,
\]
with $z(x)$ defined by \eqref{Gamma}.
First of all, we estimate the contribution of the two half-lines parameterized by the mappings $\ell_1$ and $\ell_2$. Choosing $\ell_1$ as in \eqref{Gamma}, we have
\begin{equation}\label{contrLines}
\left|\frac{1}{2\pi\opi}\int_{-\infty}^{-\frac{\pi}{2}}e^{\left(A_3+x+\frac{\pi}{2}-\opi\left(A_2-d\left(x+\frac{\pi}{2}\right)\right)\right)t}\hat{u}(z(x))z'(x)\,dx\right|\leq \frac{K_{\ell}}{2\pi}\int_{-\infty}^{0}e^{A_3t+x t}\,dx = K_{\ell}\frac{e^{z_lt}}{2 \pi t},
\end{equation}
with our choice $A_3=z_{l}$ in \eqref{eq1}. An analogous bound follows for $z(x)=\ell_2(x)$. We notice that the same estimate is not valid if we choose the two half-lines to be vertical.

Next we notice that the only portion of the ellipse that we approximate by using the quadrature formula \eqref{quadrature2} is the one parameterized on the interval $[-c\pi,c\pi]$. We estimate now the contribution of the other two intervals $[-\pi/2,-c\pi]$, $[c\pi,\pi/2]$. Recalling that we assume $|G(c\pi)|=tol$ and $|G(w)|\leq tol$ for all $w\in[-\pi/2,-c\pi]\cup[c\pi,\pi/2]$ as in Theorem \ref{th:errquad_ellipse}, we have
\[
	\left|\frac{1}{2\pi\opi}\int_{-\frac{\pi}{2}}^{-c\pi}e^{z(w)t}\hat{u}(z(w))z'(w)\,dw\right|\leq \frac{1}{2}\left(\frac{1}{2}-c\right)tol
\]and analogously for the integral on $[c\pi,\pi/2]$. In the end the error
\[
	\left|\frac{1}{2\pi\opi}\int_{-c\pi}^{c\pi}e^{z(w)t}\hat{u}(z(w))z'(w)\,dw-I_N\right|\,,
\]where $I_N$ is defined in \eqref{quadrature2}, is estimated by Theorem \ref{th:errquad_ellipse}. The thesis then follows. % straightforwardly.
\end{proof}

\subsection{Stability of the method}\label{secStab}
One of the most attractive features of the method is its stability.  In particular, we are able to compute the stability constant of the method.

In practice, we approximate the exact solution $u(t)$ by the linear combination
\begin{equation}\label{IN}
	\tilde{I}_N=\frac{c}{N\opi}\sum_{j=1}^{N-1}e^{z(x_j)t}\hat{u}_j z'(x_j)
\end{equation}where
% \begin{equation}\label{hatU}
	$\hat{u}_j=\hat{u}(z(x_j))+\rho_j$
% \end{equation}
and $\rho_j$ is the error in the numerical solution of the linear system
\begin{equation}\label{systemLap}
	\left(z(x_j)I-A\right)\hat{u}=u_0+\hat{b}(z(x_j)),
\end{equation}for $x_j$ our quadrature nodes. We assume that the quadrature nodes $x_j$, the parametrization $z(x)$ and its derivative $z'(x)$ are computed exactly. We state the following result.
\begin{prop} % \label{stability}
	The described method is numerically stable and the stability constant is given by
	\begin{equation}\label{stabilityConst}
	 2a_2c\,e^{(a_1+a_2+z_l)t},
	\end{equation}
	with $a_1,a_2$ given by formulas \eqref{a1}, \eqref{a2}.
\end{prop}
\begin{proof}
 The actual error in our computation is given by
\[
	\tilde{err}_N=\left|u(t)-\tilde{I}_N\right|\,.
\]
We can estimate it in the following way
\[
	\tilde{err}_N=\left|u(t)-\frac{c}{N\opi}\sum_{j=1}^{N-1}e^{z(x_j)t}\hat{u}_jz'(x_j)\right|\leq  err_N+err_{T}+err_N^{num}
\]
with $err_N$ and $err_{T}$ as in Theorem \ref{th:totalerr} and
\[
err_N^{num} = \left|\frac{c}{N\opi}\sum_{j=1}^{N-1}e^{z(x_j)t}\left(\hat{u}(z(x_j))-\hat{u}_j\right)z'(x_j)\right|.
\]
Recalling that the integration contour $\Gamma$ is parameterized as $z(x)=A_1\cos x+\opi A_2\sin x+A_3$, with $A_1=a_1+a_2\,,A_2=a_2-a_1$ positive and $A_3=z_l$ we have
\[
	err^{num}_N\leq\frac{c}{N}\sum_{j=1}^{N-1}e^{\left(A_1\cos x_j+A_3\right)t}|\rho_j|\left|-A_1\sin x_j+A_2i\cos x_j\right|\leq
\]
\[
	\leq \frac{c}{N}\left(|A_1|+|A_2|\right)e^{(A_1+A_3)t}\rho N=c\left(A_1+A_2\right)e^{(A_1+A_3)t}\rho
\]where $\rho=\max_{j}|\rho_j|$.
\end{proof}

\begin{rmk}\end{rmk}Given an integration contour $\Gamma$, depending on the working precision, it is possible to compute the maximal precision we can get along $\Gamma$. First we compute the condition number of the matrix $z(x_j)I-A$, for a set of $x_j\in[-c\pi,c\pi]$. Then we use this information to estimate the numerical error $\rho$ introduced when solving \eqref{systemLap}. At this point the computation of \eqref{stabilityConst} is straightforward. We can incorporate this feasibility check in our algorithm, asking the user, in case the required precision is too high, to adjust the parameter $tol$ by choosing a larger value.

\subsection{The approximation of the integrand tail}\label{selConst}
Once the integration contour is computed, we need to approximate the truncation parameter $c$ as defined in \eqref{c1}. Let us recast \eqref{c1} into the system
\begin{eqnarray}\label{firstEquation}
Ke^{\real(z(c\pi))t}=tol;\\\label{secondEquation}
K=\frac{1}{2\pi}\left\|\hat{u}(z(c\pi))z'(c\pi)\right\|.\\\nonumber
\end{eqnarray}
Observe that, once the profile of integration is fixed, considering its parametrization $z(x)=A_1\cos x+\opi A_2\sin x+A_3$, from \eqref{firstEquation} one has
\begin{equation}\label{c}
c=\frac{1}{\pi}\arccos\left(\frac{1}{A_1t}\log\left(\frac{tol}{K}\right)-\frac{A_3}{A_1}\right)\,.
\end{equation}
We suggest Algorithm \ref{algK} to compute iteratively $c,K$ (where $prec$ is the precision we ask for the constant $K$).

\begin{algorithm}[th]
\DontPrintSemicolon
%\Begin{
	\KwData{$K^{(1)}$ given, $K^{(0)}=K^{(1)}-2prec$, $j=0$} %\rightarrow \mbox{ given}
	\While{$\left|K^{(j+1)}-K^{(j)}\right|\geq prec$}{
	$c^{(j)}=\frac{1}{\pi}\arccos\left(\frac{1}{A_1t}\log\left(\frac{tol}{K^{(j)}}\right)-\frac{A_3}{A_1}\right)\,;$\
	
	$K^{(j+1)}=\frac{1}{2\pi}\left\|\hat{u}\left(z(c^{(j)}\pi)\right)z'(c^{(j)}\pi)\right\|\,;$\
	
	$j=j+1\,;$\
	}
%	}
	\caption{Numerical algorithm for approximating $c,K$.}\label{algK}
\end{algorithm}
\smallskip

In all the numerical tests we observe that Algorithm \ref{algK} converges quickly to the sought values of $c,K$. We fix $prec=10^{-1}$ and in at most $4-5$ iterations Algorithm \ref{algK} approximates the constant $K$ with a precision approximately equal to $10^{-3}$.
% It is possible to show the following Remark:
\begin{rmk}\label{rmkK}
  The sequence $(c^{(j)},K^{(j)})$ resulting from Algorithm \ref{algK} can be seen as a fixed point iteration method.
	Concerning the convergence of Algorithm \ref{algK}, assume that a pair $(c,K)$ exists such that
	\begin{equation}\label{iterativeThesis}
		\begin{array}{l}
		K e^{\real\left(z(c\pi)\right)t}=tol\\[2mm]
		K=\frac{1}{2\pi}\left\|\hat{u}\left(z(c\pi)\right)z'(c\pi)\right\|\,\\
		\end{array}
	\end{equation}where \begin{equation}\label{fIter}
	F(x):=\frac{e^{A_3t}tol^{-1}}{2\pi}\left\|\hat{u}(z(x))z'(x)\right\|
	\end{equation}
	 Assume further that $F$ \eqref{fIter} is differentiable and there exists a constant $0<\mu<1$ such that
\begin{eqnarray}\label{assumptK1}
			\left|F'(x)\right|\leq  \mu A_1te^{-A_1t\sqrt{1-\mu^2}} \quad \forall x\in I;\\\label{assumptK2}
			F(x)\in\left[e^{-A_1t\sqrt{1-\mu^2}},e^{A_1t\sqrt{1-\mu^2}}\right]\quad \forall x\in I\\\nonumber
	\end{eqnarray}being $I$ an interval containing $c\pi$, being $(c,K)$ solution to system \eqref{firstEquation}, \eqref{secondEquation} and all the points $c^{(j)}\pi$, $j\geq1$, with $c^{(j)}$ as defined in Algorithm \ref{algK}. Then both the sequences $\left\{K^{(j)}\right\}$, $\left\{c^{(j)}\right\}$ converge and in particular
	\[
		\lim_{j\rightarrow+\infty}K^{(j)}= K\,,	\qquad
		\lim_{j\rightarrow+\infty}c^{(j)}=c\,.	
	\]
	
\end{rmk}
\section{Practical implementation}\label{practImpl}
\subsection{Construction of the inner ellipse $\Gamma_+$}\label{constrEll}
In the described method, it is crucial to construct properly the inner ellipse $\Gamma_+$. It is conceived to bound the resolvent norm $\left\|\left(zI-A\right)^{-1}\right\|$ and it must be adjusted in order to leave to its left all the possible singularities of the vector $\hat{b}$. Let us focus, for the moment, on the resolvent norm: $\Gamma_+$ should be far enough from the spectrum of $A$ but, if it's too far, it will produce a slowdown of the rate of convergence of the numerical scheme. For this reason, we decided to construct the ellipse by a procedure based on the computation of some pseudospectral level curve associated to the matrix $A$. In particular, we approximate these curves using \texttt{eigtool} \cite{eigtool}. If some theoretical information about these curves is already known, we can skip the computation by \texttt{eigtool} and directly construct $\Gamma_+$ from it. First of all, we define the region were we need to ``place" our ellipse. This region is defined as $\left\{z\in\mathbb{C}\,|\,z_l\leq\real(z)\leq z_r\right\}$, where $z_l,z_r$ are real parameters to be chosen. This choice is partially heuristic and we ask the user to make a selection of these two values. We suggest the following criteria:
\begin{itemize}\label{choicez}
\item[(i)] choice of $z_l$: once the time $t$ is fixed, we need to have $e^{z_lt}$ smaller than the working precision. This is due to the fact that we use $z_l$ as the center of the ellipse and, having in mind formula \eqref{contrLines}, we need to be sure that the contribution of the two lines of the profile \eqref{Gamma} is negligible. In our experiments we choose $z_l\approx \frac{1}{t}\log(10^{-18})$.
\item[(ii)] choice of $z_r$: this point is going to be the right intersection of our half ellipse with the real axis. We can choose $z_r$ as the rightmost intersection of the pseudospectral boundary $\partial\sigma_{\varepsilon}(A)$ with the real axis taking, for example, $\varepsilon=10^{-9}$. In case $\hat{b}$ has some singularity to the right of this point, we need to move $z_r$ as to be sure that all these singularities are to the left of the half ellipse. In our examples $\hat{b}$ has a singularity in the origin. For this reason we take $z_r$ as a small positive number. Several values of $z_r$ are tested in the numerical examples.
\end{itemize}
Now we need to estimate the resolvent norm in the strip $z_l<\real(z)<z_r$. We do this by computing some suitable pseudospectral level curve of the matrix $A$. We recall that, given $\varepsilon>0$, the $\varepsilon-$pseudospectral level curve of $A$ is defined as
\begin{equation}\label{weightLvStPs1}
\sigma_{\veps}(A)=\left\{z\in\mathbb{C}\,:\, \left\|\left(zI-A\right)^{-1}\right\|=\frac{1}{\varepsilon}\right\}\,.
\end{equation}
Moreover, we call \textit{weighted} $\varepsilon-$pseudospectral level curve the set of points
\begin{equation}\label{weightLvStPs}
\sigma_{\veps,\omega}(A):=\left\{z\in\mathbb{C}\,:\, e^{\real(z)t}\left\|\left(zI-A\right)^{-1}\right\|=\frac{1}{\varepsilon}\right\}\,.
\end{equation}
Let us fix two positive values $\varepsilon_1,\varepsilon_2$ (we use $\varepsilon_1^{-1}, \varepsilon_2^{-1}$ as bounding constants). We define ${\cal C}_1$ as the weighted $\varepsilon_1-$pseudospectral level curve and ${\cal C}_2$ the $\varepsilon_2-$pseudospectral level curve. We compute ${\cal C}_1$, ${\cal C}_2$ in the strip $z_l<	\real(z)<z_r$. We notice that both ${\cal C}_1$, ${\cal C}_2$ are symmetric w.r.t. the real axis (and then we can just compute the two curves in the upper complex plane). For using our algorithm, we need the two level curves to be defined in the whole strip $\left\{z\in\mathbb{C}\,:\,z_l\leq\real(z)\leq z_r\right\}$. For this reason, if it happens that for some $x\in[z_l,z_r]$ there is no $y\in\mathbb{R}$ such that $x+\opi y$ is on the curve, we set $y=0$ and we add the point $x$ to the curve. After that, we define
\begin{equation}\label{critical}
{\cal C}=\left\{x+iy\in\mathbb{C}\,|\, x+\opi y_1\in{\cal C}_1\,, x+\opi y_2\in{\cal C}_2\mbox{ and } |y|=\max\left(|y_1|,|y_2|\right)\right\}\,.
\end{equation}
We call ${\cal C}$ \textit{critical curve}. The meaning of the computation of ${\cal C}$ is the following: we approximate ${\cal C}_1$ in order to have a bound of $\left\|e^{zt}\left(zI-A\right)^{-1}\right\|$ and then of the integrand function \eqref{integrandFunction}, while on ${\cal C}_2$ we require that the norm of the resolvent is not too high in order to ensure that the condition number of the system \eqref{algEqLap} on the points of the curve $\Gamma_+$ (and then of the ellipse $\Gamma$) is not too big to cause a loss of accuracy. Defining ${\cal C}$ as in \eqref{critical} will ensure that both the conditions are satisfied.
From a practical point of view, we computed the value of the resolvent norm on a grid in the strip $z_l<\real(z)<z_r$ using the Matlab code \texttt{eigtool}, \cite{eigtool}. It is straightforward to compute an approximation of ${\cal C}_1$, ${\cal C}_2$, and then ${\cal C}$ from it. In particular we end up with a set of points
$
P_{\mathcal{C}}:=	\left\{x_i+\opi y_i\,|\, y_i\geq 0\,,i=1\ldots,m_{\cal C}\right\}
$
approximating the upper part of the critical curve ${\cal C}$. The ellipse $\Gamma_+$ should be chosen in order to approximate in the sharpest way the curve ${\cal C}$. Moreover, since it may happen that the grid chosen by \texttt{eigtool} is too coarse, the critical curve may not capture some of the eigenvalues of $A$ (in the case where the pseudospectral level curve near those ones are very small closed curves surrounding the eigenvalues). In addition, we have to consider also the possible singularities of $\hat{b}$ (that we assume to form a finite set). We call $\lambda_k+\opi \eta_k$, $k=1,\ldots,m_A$, the eigenvalues of $A$ and $s_j+\opi r_j$, $j=1,\ldots,m_b$, the singularities of $\hat{b}$. Now define
\[
\Phi=\left\{\rho_n+\opi\sigma_n\,|\,n=1\ldots m\right\}
	=P_{\mathcal{C}}\cup\sigma(A)\cup\left\{s_j+\opi r_j\,|\, j=1,\ldots,m_b\right\}.
\]
The ellipse $\Gamma_+$ has to be chosen in order to enclose all the points of $\Phi$. In this way, the two conditions (3), (4) of Assumption \ref{assumptionG} will be satisfied with $W_+=\varepsilon_1^{-1}$ and $R_+=\varepsilon_2^{-1}$.
We recall that, the ellipse $\Gamma_+$ is completely defined by its center (that we choose as $z_l$), and two other points: the point $z_r$ and any other point $d+\opi r\in\Gamma_+$, that we used in the equations \eqref{eq1}, \eqref{eq4}. We can always take $r>0$ since the curve is symmetric w.r.t. the real axis.
The basic idea of the construction of $\Gamma_+$ is to start with a large ellipse (a circle) and then shrink it as much as possible so that all the points of $\Phi$ lay below the upper arc of $\Gamma_+$. The construction we suggest is the following:
\begin{itemize}
\item[(i) ] We start considering a circle centered at $z_l$ with radius $z_r-z_l$.
\item[(ii) ] If all the points of $\Phi$ are inside the circle, we try to shrink the curve. To do this, we consider a partition of the interval $[z_l,z_r]$, say $\left\{x_{r}\right\}_{r=1,\ldots,m_{\ell}}$ with $x_1=z_l$, $x_{m_{\ell}}=z_r$. We consider the ellipse $\gamma$ centered at $z_l$, passing through the point $z_r$ and having right focus at $x_2$. Then, we check, from the right to the left if all the points of $\Phi$ are enclosed by $\gamma$. If there is a point $\rho_j+\opi\sigma_j$ that is outside $\gamma$, we go back to the circle of the previous case and we choose it as $\Gamma_+$. In particular we set $d=\rho_j$ and $r$ as the corresponding imaginary part on the circle.
In our experiments we set $m_{\ell}=1000$.
\item[(iii) ] If, on the contrary, all the points of $\Phi$ are inside $\gamma$, we go on choosing a more elliptical curve having right focus at $x_3$ and then we check the position of $\Phi$ w.r.t. the new $\gamma$. We stop when the curve $\gamma$ excepts some points of $\Phi$, choosing the previously computed $\gamma$ as $\Gamma_+$ and the point $d+\opi r$ on it as we did in the previous case.
\item[(iv) ] If the last computed ellipse $\gamma$, the one with right focus at $x_{m_{\ell}-1}$ is still large enough to enclose $\Phi$, we set $d+\opi r$ as the point on $\gamma$ corresponding to $d=z_l$.
\item[(v) ] If in the very first point the circle already excepts some points of $\Phi$, we choose among them the point with higher (in modulus), imaginary part, say it is the point $\rho_j+\opi \sigma_j$ and we set $d+\opi r=\rho_j+\opi(\sigma_j+\varepsilon)$ for a small $\varepsilon>0$. The curve $\Gamma_+$ is the one centered at $z_l$, and passing through $d+\opi y$. We can adjust the value of $\varepsilon$ in order to ensure that $\Phi$ is all inside $\Gamma_+$.
\end{itemize}
\begin{figure}[h!]
\begin{center}
\includegraphics[scale=0.35]{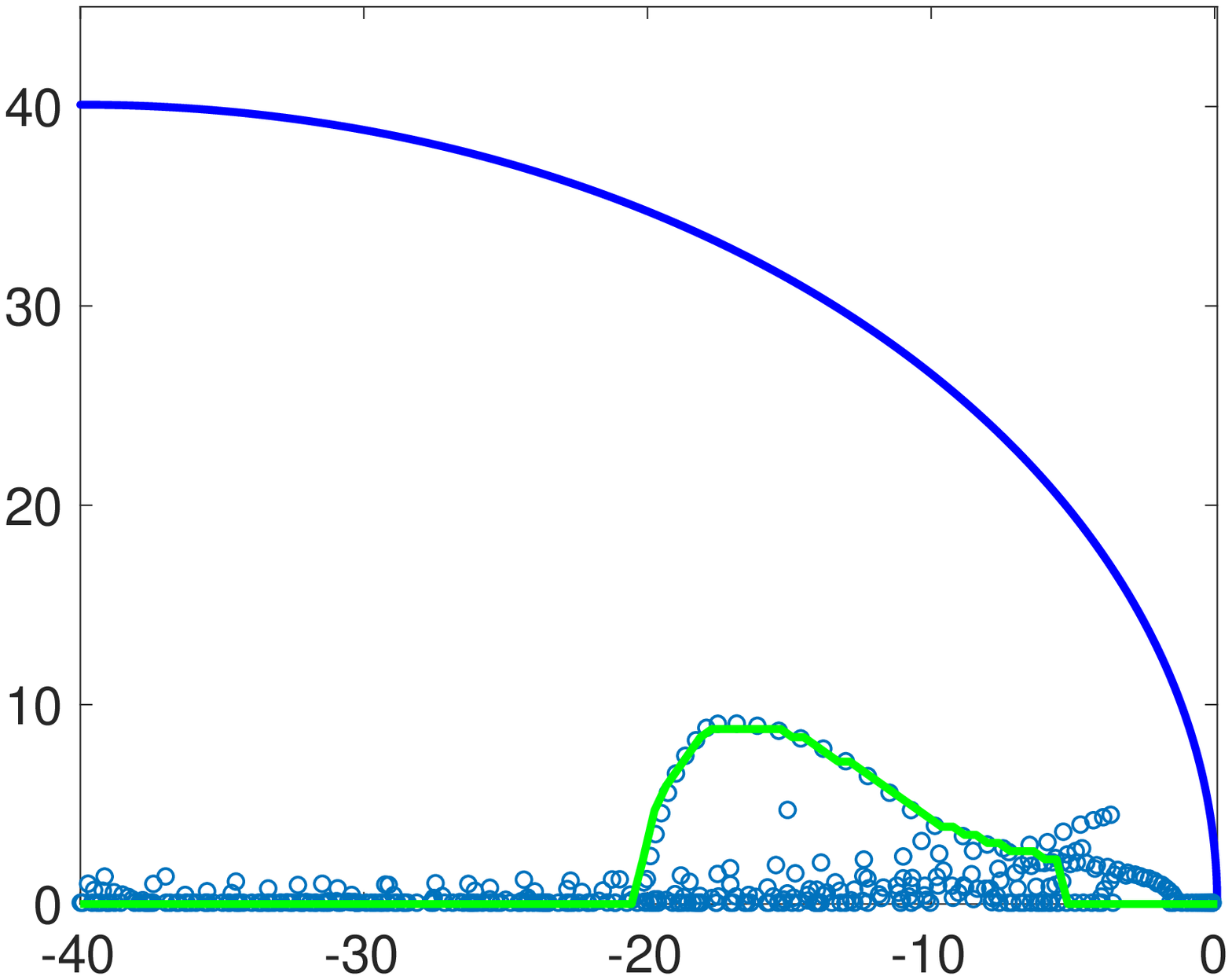}\includegraphics[scale=0.35]{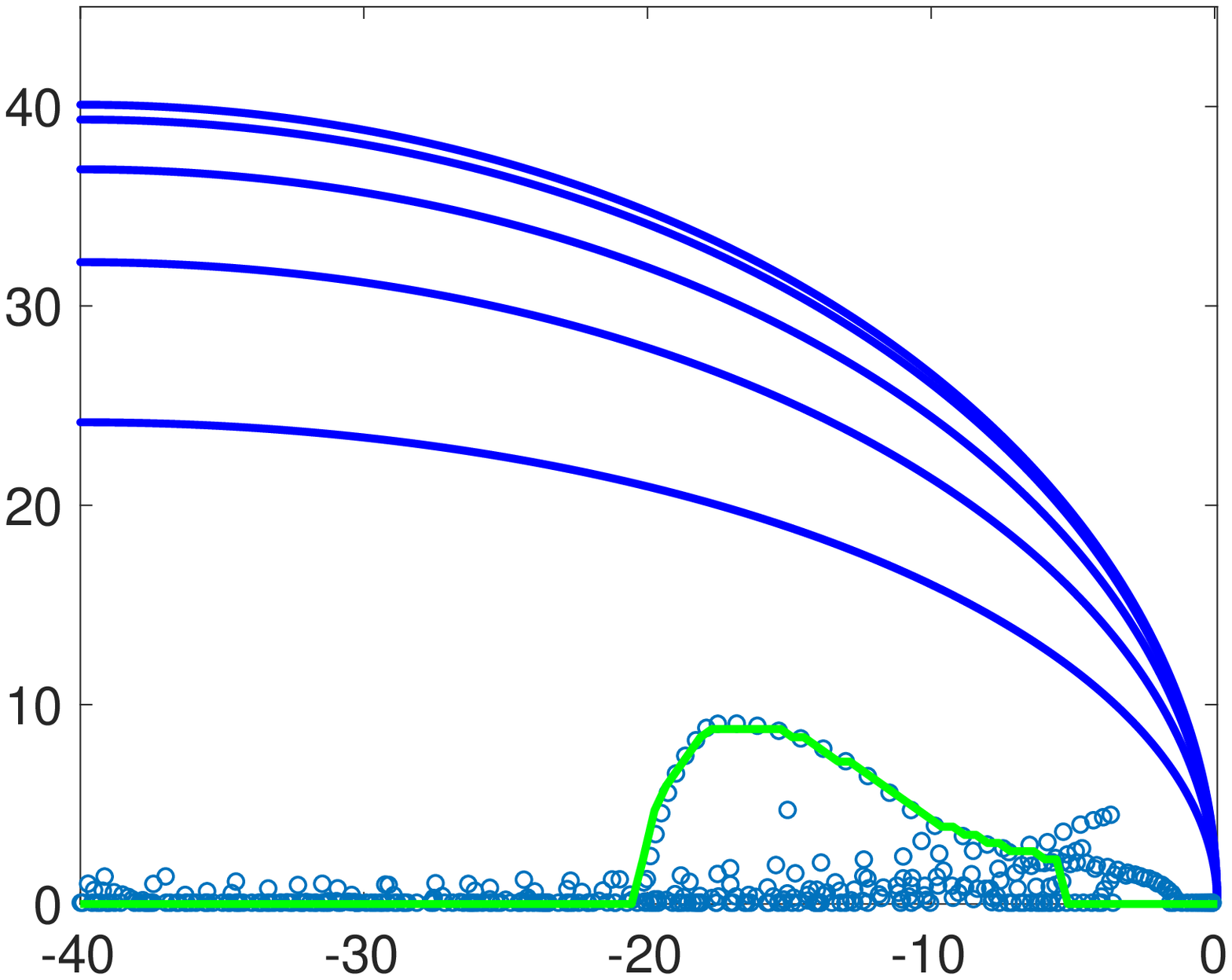}\\
\includegraphics[scale=0.35]{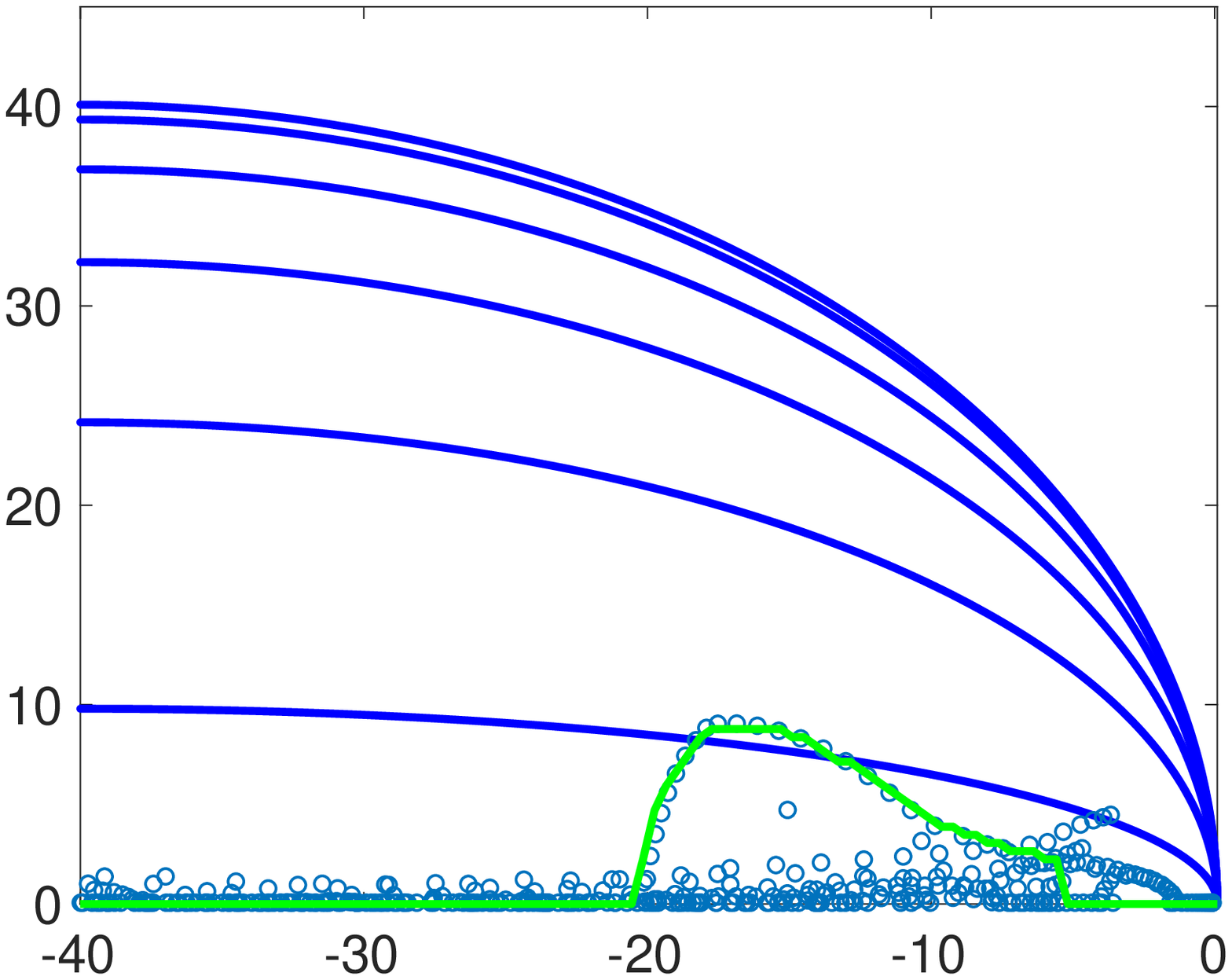}\includegraphics[scale=0.35]{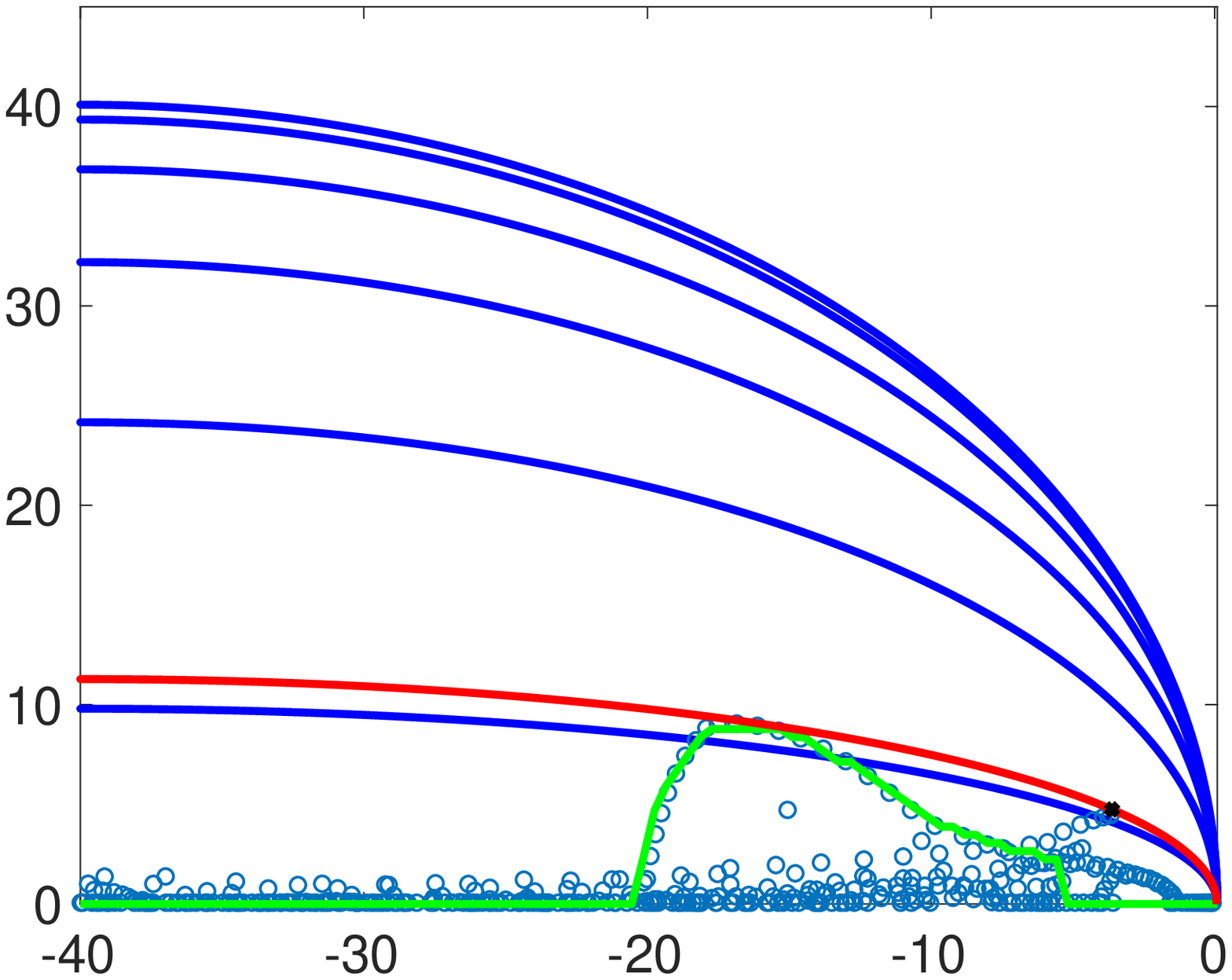}
\caption{Algorithm behaviour when applied to a matrix arising in the discretization of Heston operator.}\label{plotAlgo}
\end{center}
\end{figure}
In Figure \ref{plotAlgo}, we show how the algorithm works on a particular matrix $A$ coming from the spatial discretization of the Heston operator
(that we will use in Section \ref{sec:num}). The green line is a 100 points approximation of the critical curve ${\cal C}$ in the region defined by $z_l=-40\,,z_r=0.09$, while the blue dots are the eigenvalues of the matrix $A$ in the same strip of the complex plane. \textit{Top left}: the blue line is the first guess for the ellipse $\Gamma_+$ (a circle). \textit{Top right}: the guesses for $i=1,20,40,60,80$ are plotted in blue (where each ellipse has right focus $x_i$, a point of the partition of the interval $[-40,0.09]$). \textit{Bottom left}: in addition to the previous guesses, also the first ellipse excepting some of the eigenvalues of $A$ and crossing the curve ${\cal C}$ is plotted (it corresponds to the index $i=97$). \textit{Bottom right}: $\Gamma_+$ is drown in red (and it is chosen as the ellipse corresponding to $i=96$) and the point $d+\opi r$ is marked with a black cross\\

\subsection{Computational cost}\label{compCost}
The method consists in two main parts: the \textit{precomputing} part, where we construct the ellipse $\Gamma_+$, see Section \ref{constrEll}, and compute the actual ellipse of integration $\Gamma$ (Section \ref{ChTarget}) and the \textit{computing} part, where the quadrature formula \eqref{quadrature2} is applied to get the approximation of the unknown function $u(t)$. The precomputing part is essential in order to make the algorithm \textit{general}, since by means of it all the pseudospectral information about the discrete operator $A$ is approximated and, then, used to set the optimal profile of integration. In particular, in this part the software \texttt{eigtool} is employed. This software allows to compute an approximation of the resolvent norm at the points of a rectangular grid in the complex plane. As a byproduct, \texttt{eigtool} also computes the eigenvalues of $A$. In the Table \ref{execEig}, we report the time employed by \texttt{eigtool} in computing the approximations for both Black-Scholes and Heston operator (the software is run on a laptop with a 2,4 GHz Intel Core i5 processor, using \texttt{Matlab R2016b}). In each case, we report the region of the complex plane (\textit{Box}) where we seek the approximation. For Black-Scholes (size $200\times200$), the number of points defining the rectangular grid is set to $200$, while it is set to $50$ in the case of Heston equation (size $2500\times2500$).
\begin{table}
	\begin{tabular}{|c|c|}\hline
	Black-Scholes & Heston\\\hline
	\begin{tabular}{c|c}
	$t=1$&\\
	Box $=[-40,0.05]\times[-10,10]$& $15.251361$ sec\\\hline
	$t=10$ &\\
	Box $=[-4,0.01]\times[-6,6]$ &$15.523646$ sec\\
	\end{tabular} &\begin{tabular}{c|c}
	$t=1$ &\\
	Box $=[-40,0.09]\times[-10,10] $& $53.806976$ sec\\\hline
	$t=10$&\\
	Box $=[-4,0.06]\times[-6,6] $& $41.270690$ sec\\
	\end{tabular}\\\hline
	\end{tabular}
	\caption{Execution time for \texttt{eigtool}}\label{execEig}
\end{table}The computation of the pseudospectral information is used to compute the critical curve ${\cal C}$ \eqref{critical}. The ellipse $\Gamma_+$ is selected by the algorithm explained in Section \ref{constrEll}. We remark that the use of \texttt{eigtool} produces a larger computational cost w.r.t. other methods \cite{ITHWeid,LP,LPS,W}. However, we remark that these methods can be efficiently applied only if some pseudospectral information on $A$ is already known. The choice of an algorithmic approach to investigate the behaviour of $\left\|\left(zI-A\right)^{-1}\right\|$ makes our method ready to be used for every system of ODEs of the form \eqref{mainpb}. In the case $A$ is normal or some curve ${\cal C}$ bounding the region where the resolvent norm of $A$ grows unboundedly is already known, the use of \texttt{eigtool} can be avoided.\\
For the computation of the truncation parameter $c$, we apply Algorithm \ref{algK}. Every iteration of this procedure requires the evaluation of the function
\[
	\left\|\left(zI-A\right)^{-1}\left(u_0+\hat{b}(z)\right)z'\right\|\,.
\]The cost of this evaluation depends on the size of $A$. We observe that for every iteration of Algorithm \ref{algK} the linear system
\begin{equation}\label{systComp}
\left(\zeta I-A\right)\hat{u}(\zeta)=u_0+\hat{b}(\zeta)
\end{equation}is solved for $\zeta=z(c\pi)$. % In all the examples, this construction is made in less than $0.5$ seconds.
%The optimal profile $\Gamma$ is computed by minimizing the function \eqref{fToMin}. The parameters $c,K$ are computed by Algorithm \ref{algK}. The cost of this iterative procedure depends on the number of evaluations of the integrand function \eqref{integrandFunction} which is connected with the number of iterations needed to reach a good approximation of $K$. This part is quite fast (for example, in the case of Heston equation, for $t=1$ and $tol=10^{-6}$, it is done in $2.271847$ seconds).

 Concerning the integration part, i.e., the computation of the quadrature sum \eqref{quadrature2}, the execution time depends on the size of the matrix $A$ and the number of quadrature nodes that are used to get the desired accuracy. In order to save time, it is possible to precompute in parallel the integrand function at the quadrature nodes. Assuming the critical curve ${\cal C}$ already computed by \texttt{eigtool}, in Tables \ref{execTot1}, \ref{execTot2} we report the time needed to select the profile $\Gamma_+$, to approximate the truncation parameter $c$ by Algorithm \ref{algK} of Section \ref{innerConstr} and to apply the quadrature sum \eqref{quadrature2} for $n=30$ nodes. In all cases the resolution of the system \eqref{systComp} is done using the backslash command of Matlab. We also report the total number of solved linear systems.
\begin{table}
\begin{center}
	\begin{tabular}{|c|c|c|c|c|}\hline
		& Select. $\Gamma_+$ & Algorithm \ref{algK} & Numerical quadrature& \# Solved linear systems\\\hline
		$t=1, tol=5e-6$ &&$0.0498 \ sec$&\\
		$z_l=-40,z_r=0.05$ &$0.6850 \ sec $& $4$ iterations&$0.1254 \ sec$ &$34$\\\hline
		$t=10,tol=5e-6$ &&$0.0263 \ sec $&\\
		$z_l=-4,z_r=0.01$ &$0.3203 \ sec$ & $4$ iterations&$0.1048 \ sec$ &$34$\\\hline
	\end{tabular}
	\caption{Execution time for Black--Scholes.}\label{execTot1}
	\end{center}
\end{table}
\begin{table}
\begin{center}
	\begin{tabular}{|c|c|c|c|c|}\hline
		&Select. $\Gamma_+$ & Algorithm \ref{algK} & Numerical quadrature&\# Solved linear systems\\\hline
		$t=1, tol=5e-6$ &&$4.1404 \ sec$&\\
		$z_l=-40,z_r=0.09$ &$0.5360 \ sec $& $4$ iterations&$15.7459 \ sec$ &$34$\\\hline
		$t=10,tol=5e-6$ &&$5.0225 \ sec $&\\
		$z_l=-4,z_r=0.06$ &$0.1315 \ sec$ & $4$ iterations&$17.6780 \ sec$ &$34$\\\hline
	\end{tabular}
	\caption{Execution time for Heston.}\label{execTot2}
	\end{center}
\end{table}\\
We notice that it is possible to save the computation done for $n$ when approximating the solution on $2n$ nodes (taking an extra node between each pair of consecutive points). In this way the execution time is halved doubling the number of nodes.

An extra computational cost is needed for the \textit{feasibility check}: it is possible, indeed, to approximate the stability constant of the method by computing the condition number of the system \eqref{systemLap} for a set of points on the integration ellipse $\Gamma$, as explained in Section \ref{secStab}. However, this feasibility check can be skipped and it is not necessary in order to run the algorithm even if it is useful to make a forecast on the maximal precision attainable on the computed $\Gamma$ and to check whether the chosen tolerance is too sharp w.r.t. the working precision.
It is worth noticing that the execution time needed to run \texttt{eigtool} strongly depends on the fixed number $N$ of grid points. How to set this number? Numerical experiments suggest that we need a very low resolution of the computed pseudospectral level curves and the algorithm is robust w.r.t. the choice of $N$. For example, for the Black-Scholes case, with $t=1$, $tol=5\cdot 10^{-6}$, we seek approximations of the resolvent norm in the box $[-40,0.05]\times[-10,10]$. Then, we compute a reference solution corresponding to $N=1000$ and we measure the error w.r.t. this approximation of the solutions computed using $N=500,250,100,50,25$. We also report the time employed by \texttt{eigtool} to perform the approximation. The results are listed in Table \ref{tabRob}.
\begin{table}
\[
	\begin{array}{|c|c|c|}\hline
	\mbox{Points} &\mbox{Execution time}& \mbox{Error w.r.t. reference solution}\\\hline
	1000 & 247.435454 \ sec & 0\\\hline
	500 & 60.914193 \ sec& 1.3927e-12\\\hline
	250 & 18.149750 \ sec& 0\\\hline
	100 & 4.917106 \ sec & 0\\\hline
	50 & 2.684333 \ sec & 9.7751e-07\\\hline
	25 &  2.675066 \ sec & 9.6867e-07\\\hline
	\end{array}
\]
\caption{}\label{tabRob}
\end{table}

\subsection{Summary of the method}\label{subsec:outline}
The main parts of our method are: construction of the inner ellipse $\Gamma_+$, selection of the optimal integration contour $\Gamma$, truncation of the profile $\Gamma$. In the end, the trapezoidal quadrature rule is applied to the selected portion of ellipse to get the sought approximation of the solution to \eqref{mainpb}. For the sake of clarity, we briefly list the sequential steps which are needed to implement the method:
\begin{enumerate}\label{steps}
\item Given $A$, $b$ and $u_0$ in \eqref{mainpb}, a time $t$ and a precision $tol$, the method provides an approximation of the unknown solution $u(t)$ of \eqref{mainpb} with accuracy $tol$;\\
\item Computation of the inner ellipse $\Gamma_+$. As explained in Section \ref{constrEll}, the user is asked to choose the values of $z_l,z_r$ (respectively the center of $\Gamma_+$ and its right intersection with the real axis). This choice is partially heuristic but it is guided by (i), (ii) on Page~\pageref{choicez}. The procedure of Section \ref{constrEll} returns a point $d+\opi r$ which defines uniquely the profile $\Gamma_+$, together with $z_l,z_r$. The construction uses \texttt{eigtool} for the computation of the pseudospectral sets $\sigma_{\varepsilon_2}(A), \sigma_{\varepsilon_1,\omega}(A)$ (as defined in \eqref{weightLvStPs1}, \eqref{weightLvStPs}). In all the numerical experiments we found that the choice $\varepsilon_1=10^{-9}$, $\varepsilon_2=10^{-13}$ was effective.\\
\item Computation of the integration profile as explained in Section \ref{innerConstr}. In particular, the ellipse of integration is parameterized as
\[
	\Gamma: \quad z(x)=(a_1+a_2)\cos x+\opi(a_2-a_1)\sin x+A_3,
\]with coefficients $a_1,a_2,A_3$ depending on just one free parameter $a$ by formulas \eqref{a1}, \eqref{a2}, \eqref{C}. In order to find the optimal profile of integration we minimize the scalar function of $a$ in \eqref{fToMin};\\
\item Truncation of the Bromwich integral: since we are interested in approximating $u(t)$ within precision $tol$, we only consider the portion of the Bromwich integral parameterized in $[-c\pi,c\pi]$, for a certain truncation parameter $c$. The computation of $c$ is performed by Algorithm \ref{algK} as explained in Section \ref{innerConstr};\\
\item Apply quadrature formula \eqref{quadrature2} to get the sought approximation of $u(t)$.\\
\end{enumerate}

\section{Numerical results}\label{sec:num}
In this section, we collect some numerical results of our method. We first consider a canonical convection-diffusion equation, used as an academic test. We then test our approach with the Black-Scholes and Heston equations. The  Black-Scholes model here is the same as the one considered in \cite{ITHWeid}, while for the Heston model we consider a slightly different boundary condition from that in \cite{ITHWeid}, following \cite{ITHF}. The function \eqref{fToMin} is minimized by means of the built-in Matlab function \texttt{fminbnd}. In all the examples, we construct the inner ellipse $\Gamma_+$ as explained in Subsection \ref{constrEll} taking $\varepsilon_1=10^{-9}$, $\varepsilon_2=10^{-13}$.

\subsection{A canonical convection-diffusion operator}
As a first illustration of our method we apply it, as in \cite{W}, to
\begin{equation}\label{twoPoints}
u_t=u_{xx}+u_x\,,\quad t\geq 0\,, x\in[0,d],
\end{equation}with initial and boundary conditions
\begin{equation}\label{twoPointsCond}
u(x,0)=0\,,\quad u(0,t)=0\,,\quad u(d,t)=1\,,\quad t\geq0,\  x\in[0,d]\,.
\end{equation}We consider $d=400$ and following the following steps, according to Section~\ref{subsec:outline}:
\begin{itemize}
\item Initial data: as done in \cite{W}, we discretize \eqref{twoPoints}, \eqref{twoPointsCond} by a Chebyshev spectral method. The Chebyshev differentiation matrix used has order $64\times 64$. We fix $t=1$, $tol=5\cdot10^{-8}$;\\
\item Computation of the inner ellipse $\Gamma_+$: following (i), (ii) of page \pageref{choicez}, we choose $z_l=-40$ and $z_r=0.09$. We remark that in this case the vector $b$ in \eqref{mainpb} is constant and then $\hat{b}(z)=\frac{b}{z}$ has a singularity in the origin. For this reason, we must select $z_r>0$. The procedure of Section \ref{constrEll} gives back the point $d+\opi r=-0.1071+ 0.3075\opi$. The ellipse $\Gamma_+$ is plotted in Figure \ref{plot_twoPoints3} together with the critical parabola recovered in \cite{W} ($x=-y^2$). The green line is the one called ${\cal C}$ in Section \ref{constrEll} and it is computed by \texttt{eigtool}.

\item Computation of the integration profile: once the parameters $z_l,z_r,d,r$ are fixed, we minimize the function \eqref{fToMin}. It reaches its minimum at $a=0.4543$. The minimization is done numerically using the built-in Matlab function \texttt{fminbnd}. %In Figure \ref{plot_twoPoints2} we plot the actual ellipse of integration together with the half ellipse $\Gamma_+$ computed in the previous step.
%\begin{figure}
%\begin{center}
%\includegraphics[scale=0.3]{plot_twoPoints2.eps}
%\caption{Plot of $\Gamma_+$ and $\Gamma$. In green the critical parabola $x=-y^2$ and in blue the parabolic profile selected for $n=10$ by the method \cite{ITHWeid}.}\label{plot_twoPoints2}
%\end{center}
%\end{figure}\\
\item Truncation of the Bromwich integral: we apply Algorithm \ref{algK} of Section \ref{innerConstr} to recover the values of the truncation parameter $c$ and the constant $K$ defined by equations \eqref{firstEquation}, \eqref{secondEquation}. We fix $prec=10^{-2}$ and $K^{(1)}=100$. In three iterations Algorithm \ref{algK} computes $K=0.2251$ and $c=0.3160$. %We observe that the error on $K$ is about $8.1438\cdot10^{-4}$ and the one on $c$ is around $1.4\cdot10^{-3}$ (these errors are computed by taking the difference on two consecutive approximations of $K$ and $c$ respectively).
\item We apply quadrature formula \eqref{quadrature2}. We apply it on $n$ nodes, for $n=5,\ldots,30$. We compare the resulting approximation of $u(t)$ with the one computed by direct evaluation of the exponential matrix (\texttt{expm} function in Matlab). The results are plotted in Figure \ref{plot_twoPoints3}.
%In Figure \ref{plot_twoPoints3} we plot the resulting truncated profile $\Gamma$ together with $\Gamma_+$.
\begin{figure}[h!]
\begin{center}
\includegraphics[scale=0.3]{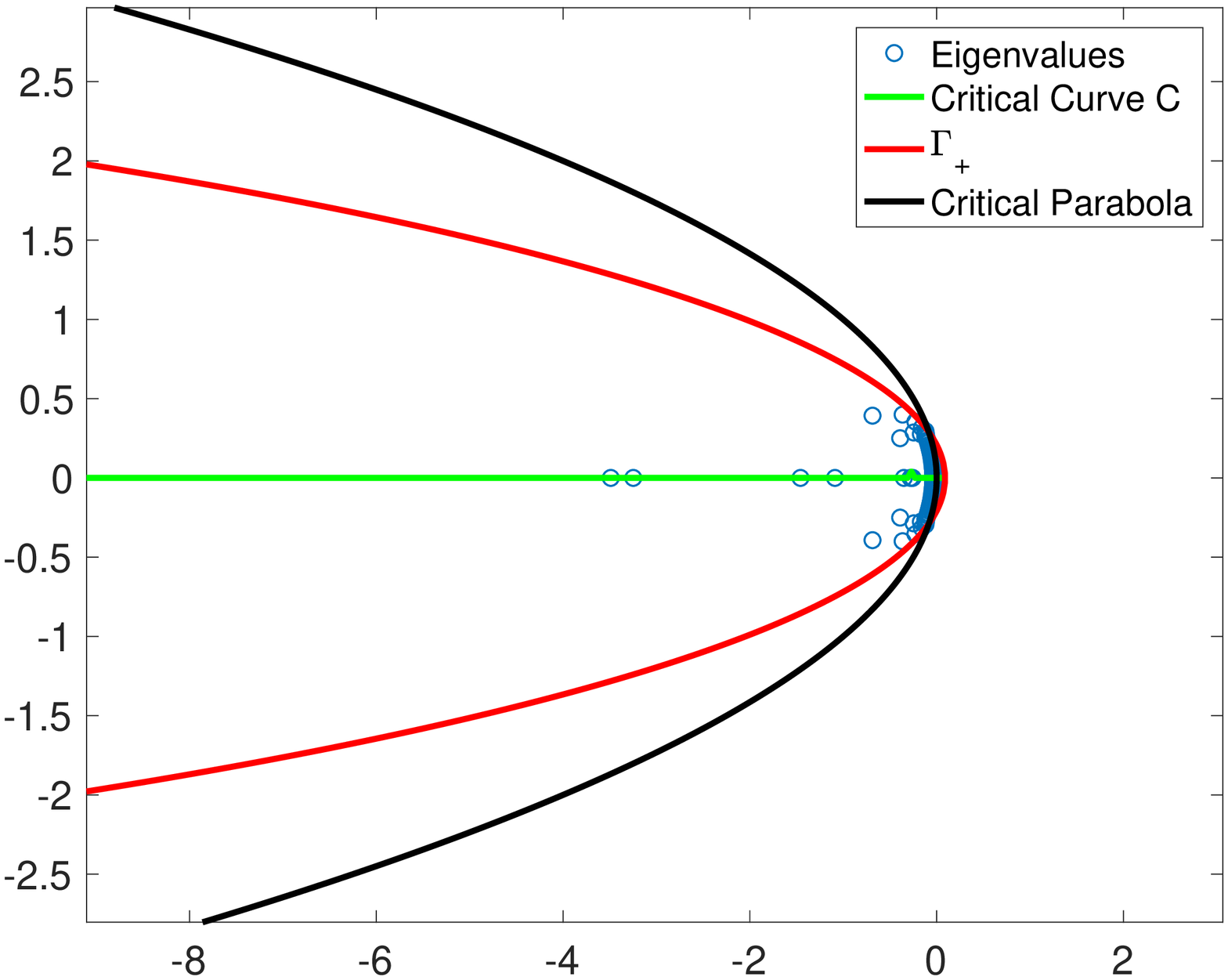}\includegraphics[scale=0.3]{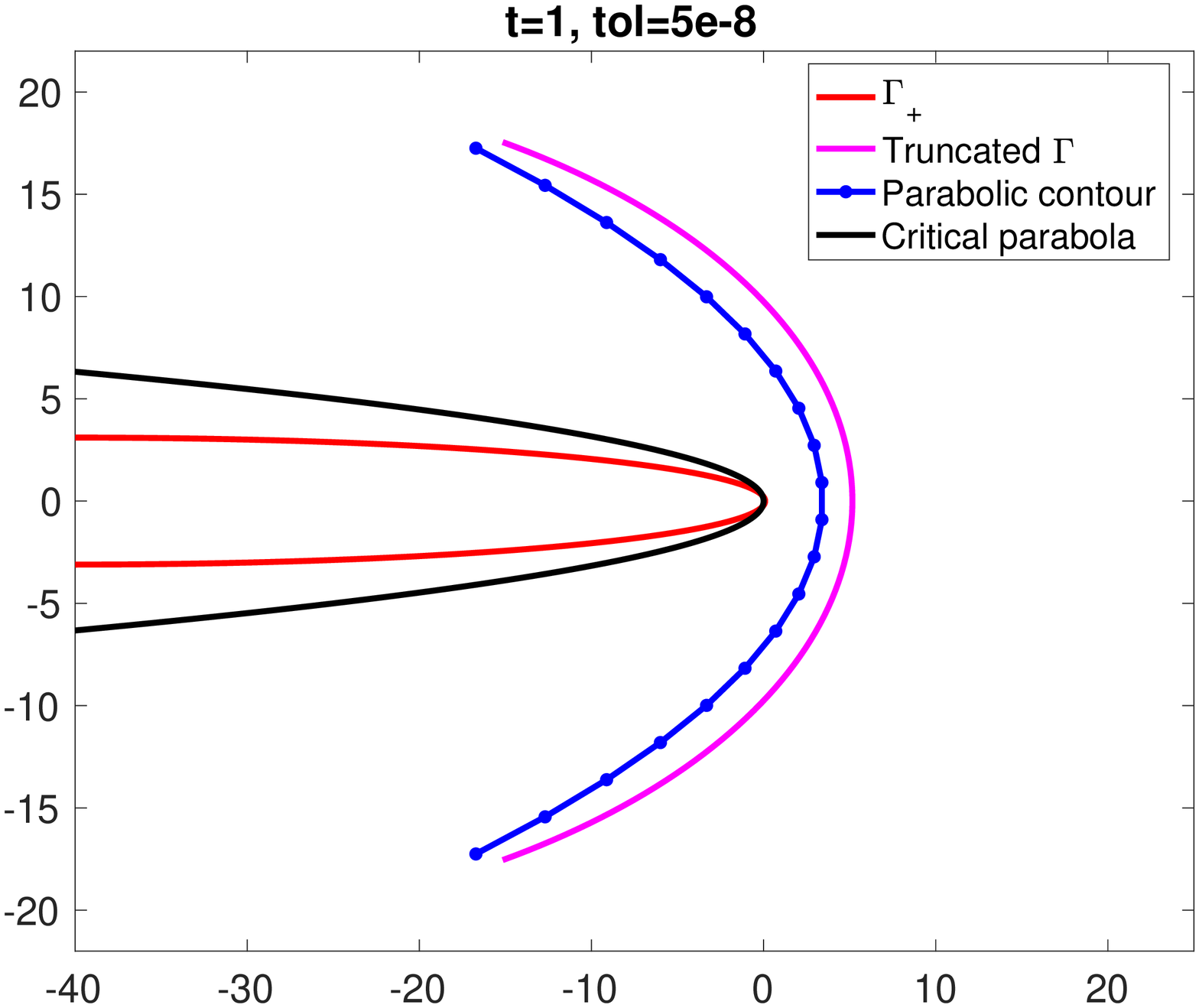}\\
\begin{center}
\includegraphics[scale=0.35]{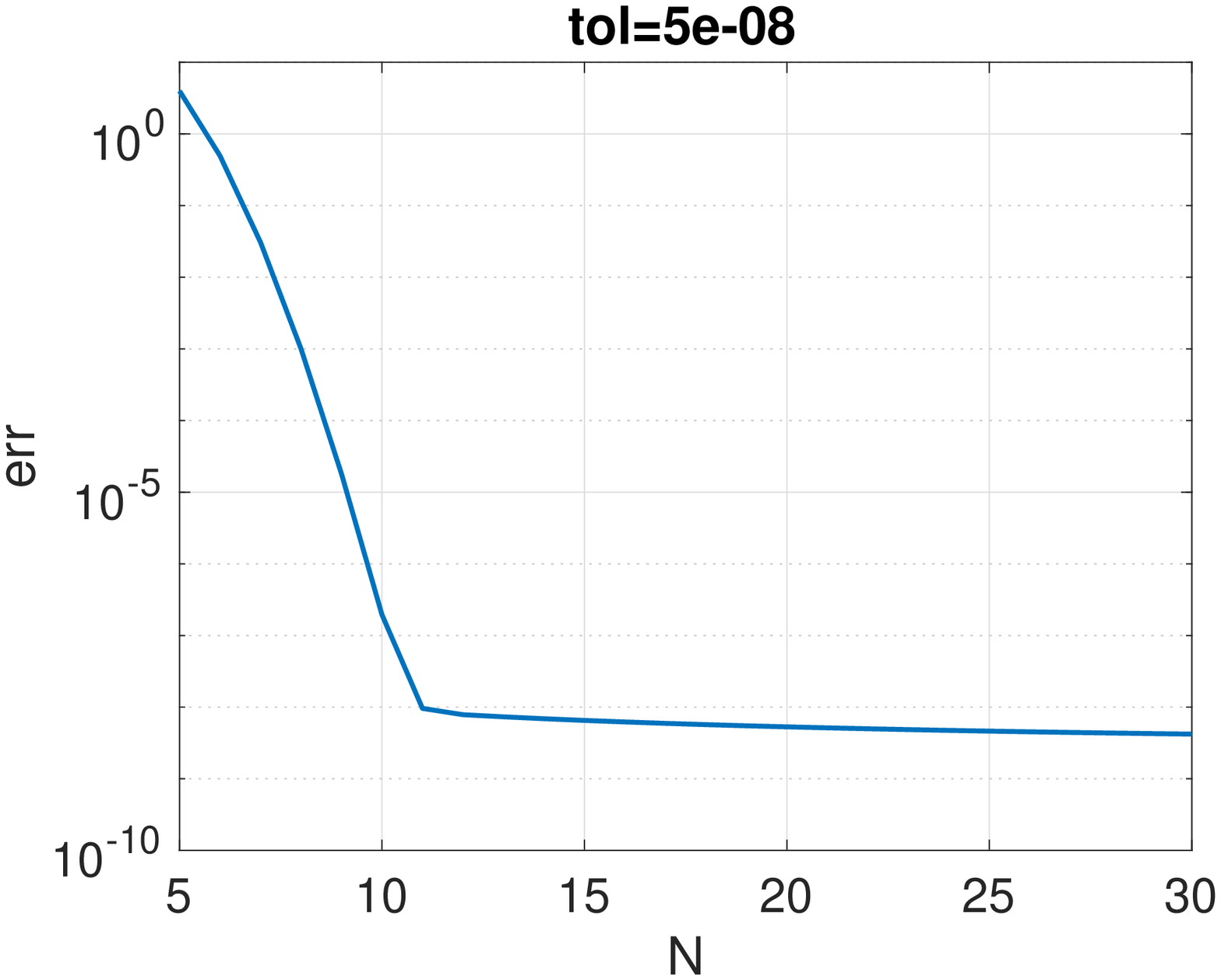}
\end{center}
\end{center}
\caption{\textit{Top Left:} comparison between $\Gamma_+$ and the theoretical critical parabola used in \cite{W}. \textit{Top Right: }plot of $\Gamma_+$ and truncated $\Gamma$. In green the critical parabola $x=-y^2$ and in blue the parabolic profile selected for $n=10$ by the method \cite{ITHWeid}. \textit{Bottom:} Error VS number of nodes.}\label{plot_twoPoints3}
\end{figure}\\
\end{itemize}
\clearpage

\subsection{Black-Scholes equation} % \label{subBS}
The well known (deterministic) Black-Scholes equation \cite{BS} has the following form
\begin{equation}\label{BlackScholes}
\frac{\partial u}{\partial \tau}=\frac{1}{2}\sigma^2s^2\frac{\partial^2 u}{\partial s^2}+rs\frac{\partial u}{\partial s}-ru\,,\quad s>L\,,\quad 0< \tau\leq t
\end{equation}for $L,t$ given, where the unknown function $u(s,\tau)$ stands for the fair price of the option when the corresponding asset price at time $t-\tau$ is $s$ and $t$ is the maturity time of the option. Moreover, $r\geq0$, $\sigma>0$ are given constants (representing the interest rate and the volatility respectively).
In practice, for the sake of numerical approximation, we consider a bounded spatial domain, considering
\[
	L<s<S
\]for a sufficiently large $S$. We take \eqref{BlackScholes} together with the following conditions, typical for the European option call
\[
u(s,0)=\max(0,s-K)
\]
\[
u(L,\tau)=0\,,\quad 0\leq \tau\leq t
\]
\[
u(S,\tau)=S-e^{-r\tau}K\,,\quad 0\leq \tau\leq t\,.
\]
Following the same strategy adopted in \cite{ITHWeid}, we discretize in space on a uniform space grid of $200$ points in $[L,S]$, for $L=0,S=200$, using the classical centered finite difference scheme. We choose $r=0.06$, $\sigma=0.05$ and $K=80$. We plot the error for a selection of tolerances for the cases $t=1$ (Fig. \ref{BSt1}) and $t=10$ (Fig. \ref{BSt10})
\begin{figure}[h!]
\begin{center}
\includegraphics[scale=0.35]{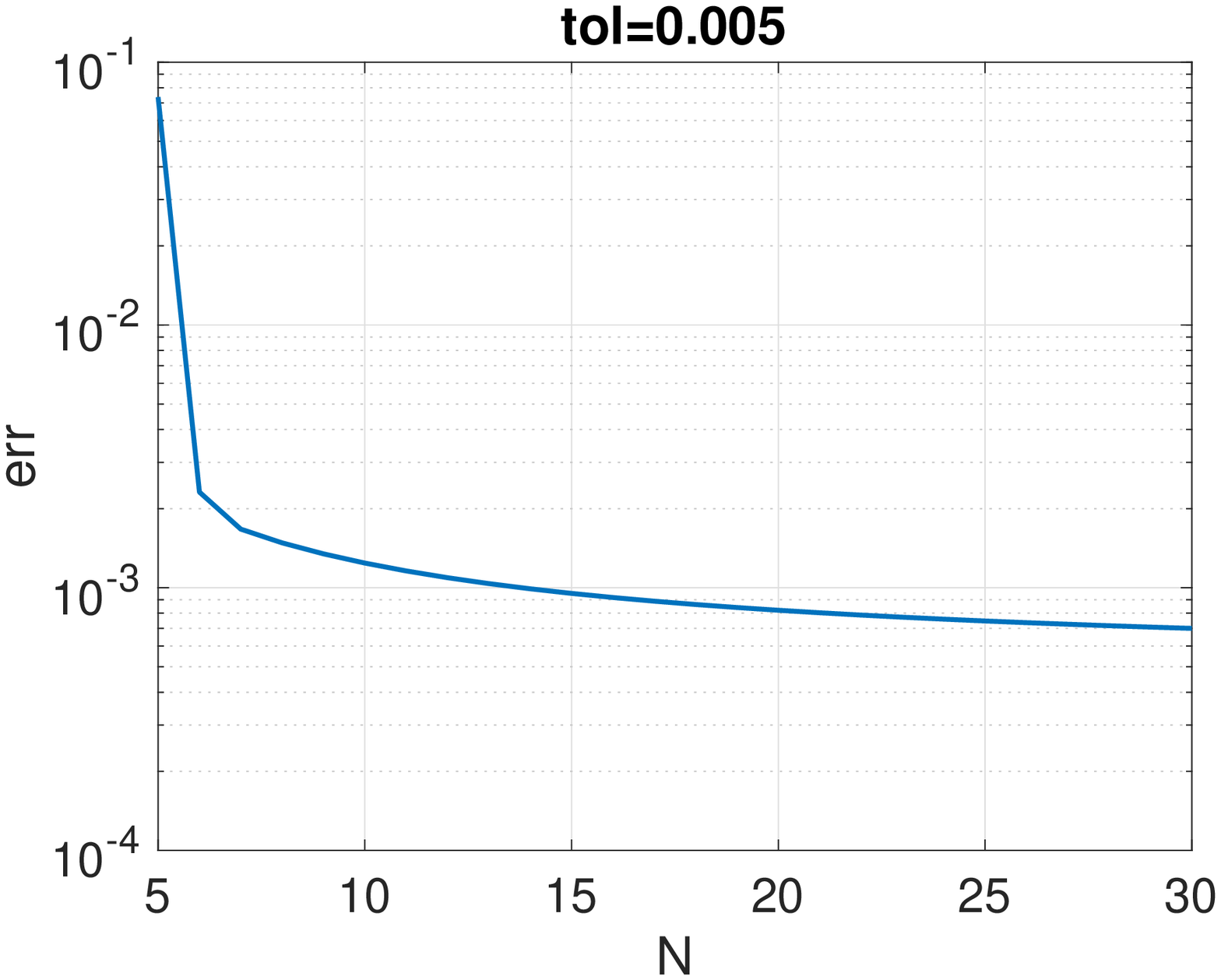}\includegraphics[scale=0.35]{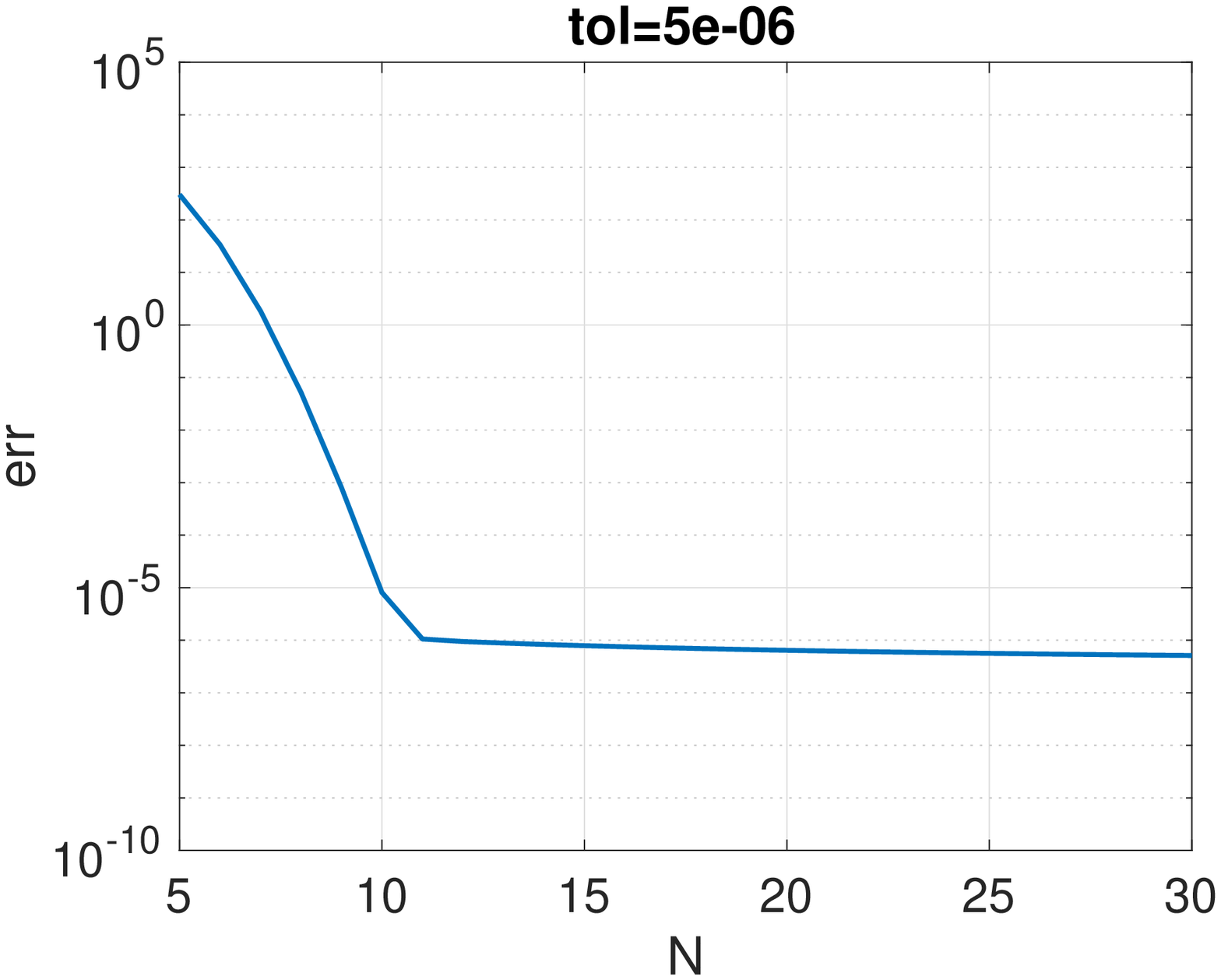}\\
\includegraphics[scale=0.35]{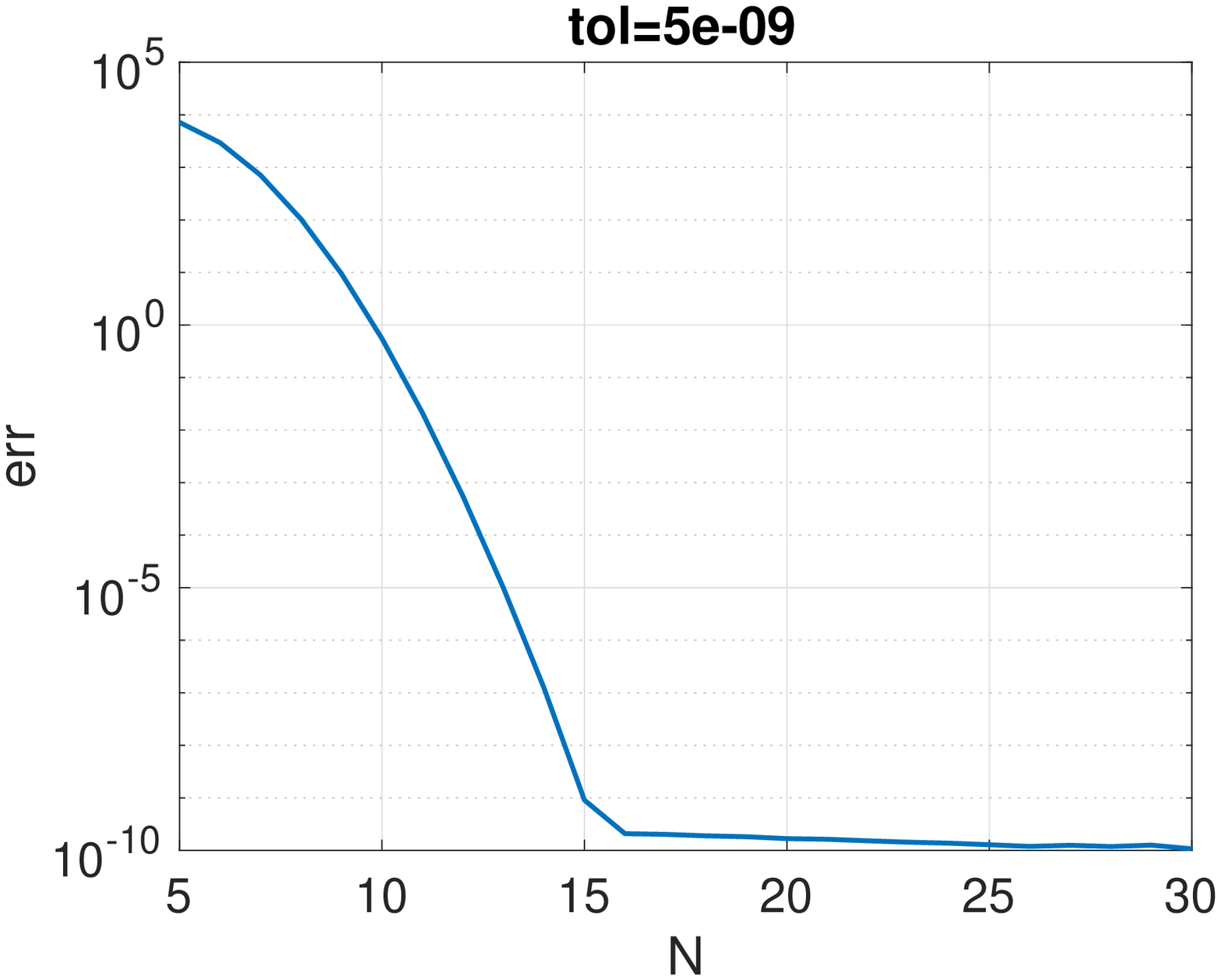}\includegraphics[scale=0.35]{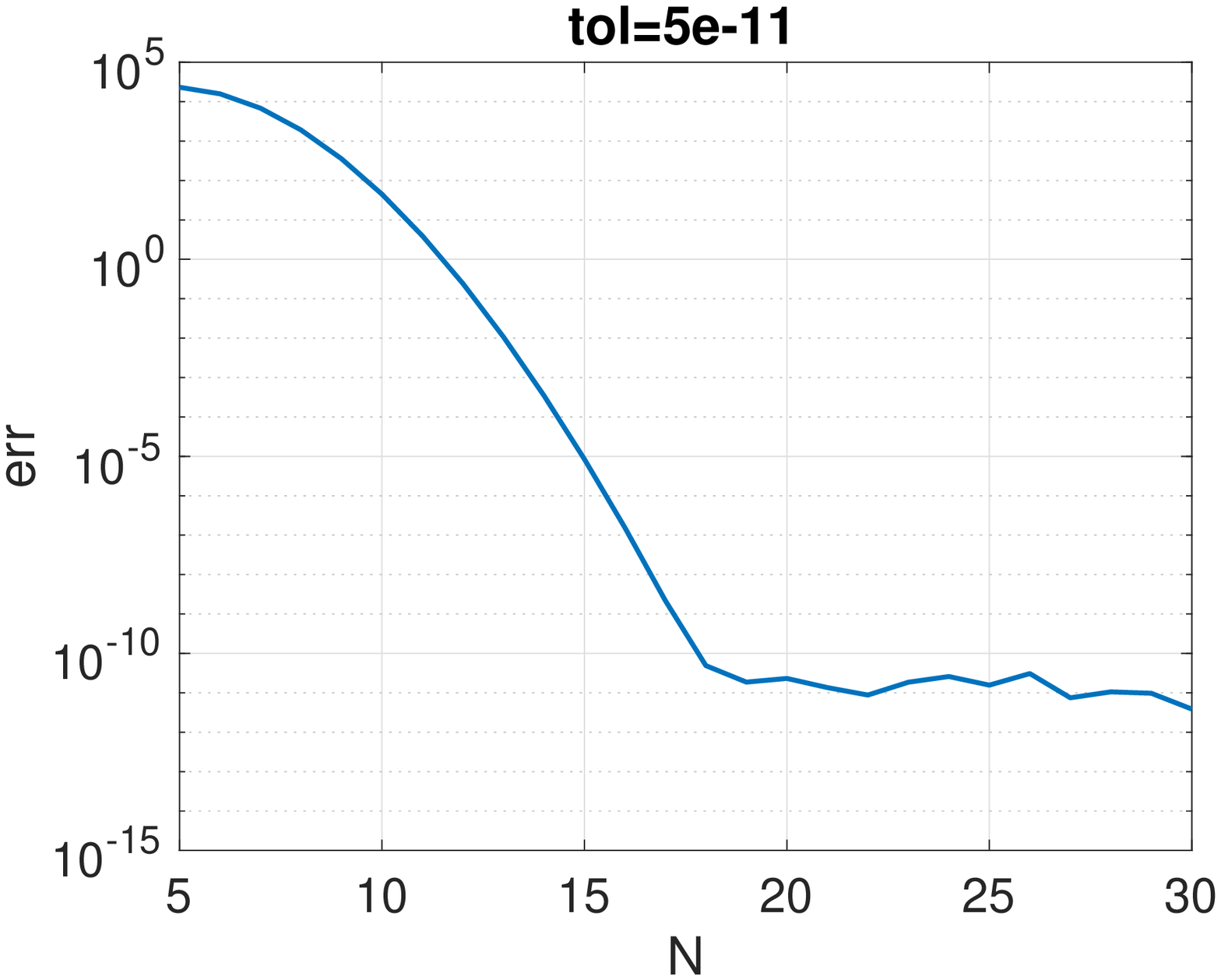}
\end{center}
\caption{Error vs number of nodes for Black-Scholes, $t=1$ ($z_l=-40,z_r=0.05$). \textit{Top left:} $tol=5\cdot 10^{-3}$ (maximal precision attainable $\approx 10^{-12}$). \textit{Top right:} $tol=5\cdot 10^{-6}$ (max. prec. $\approx 10^{-11}$). \textit{Bottom left:} $tol=5\cdot 10^{-9}$ (max. prec. $\approx10^{-11}$). \textit{Bottom right:} $tol=5\cdot 10^{-11}$ (max. prec. $\approx10^{-10}$, but in practice we are still able to reach a sharper precision).}\label{BSt1}
\end{figure}
\begin{figure}[h!]
\begin{center}
\includegraphics[scale=0.35]{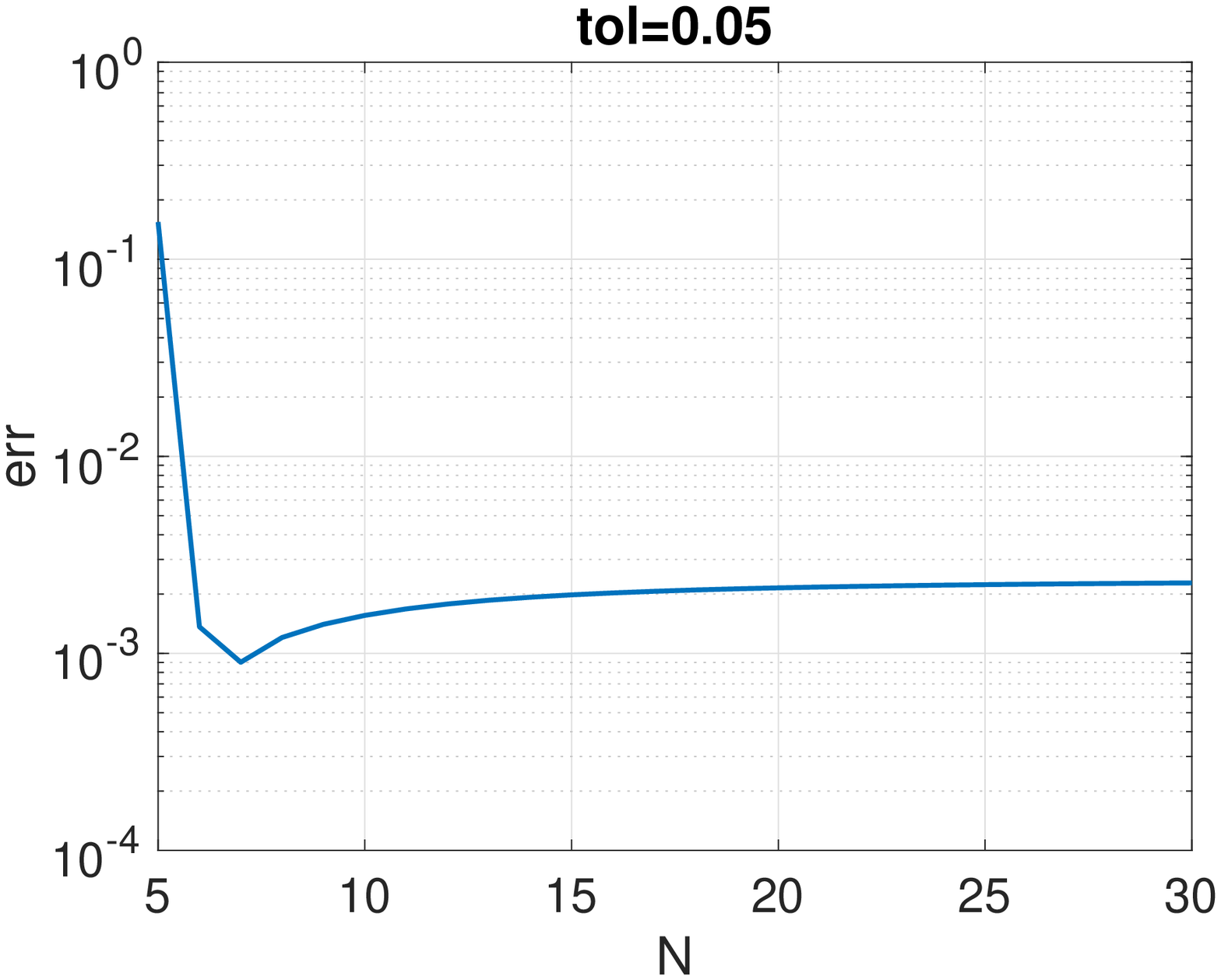}\includegraphics[scale=0.35]{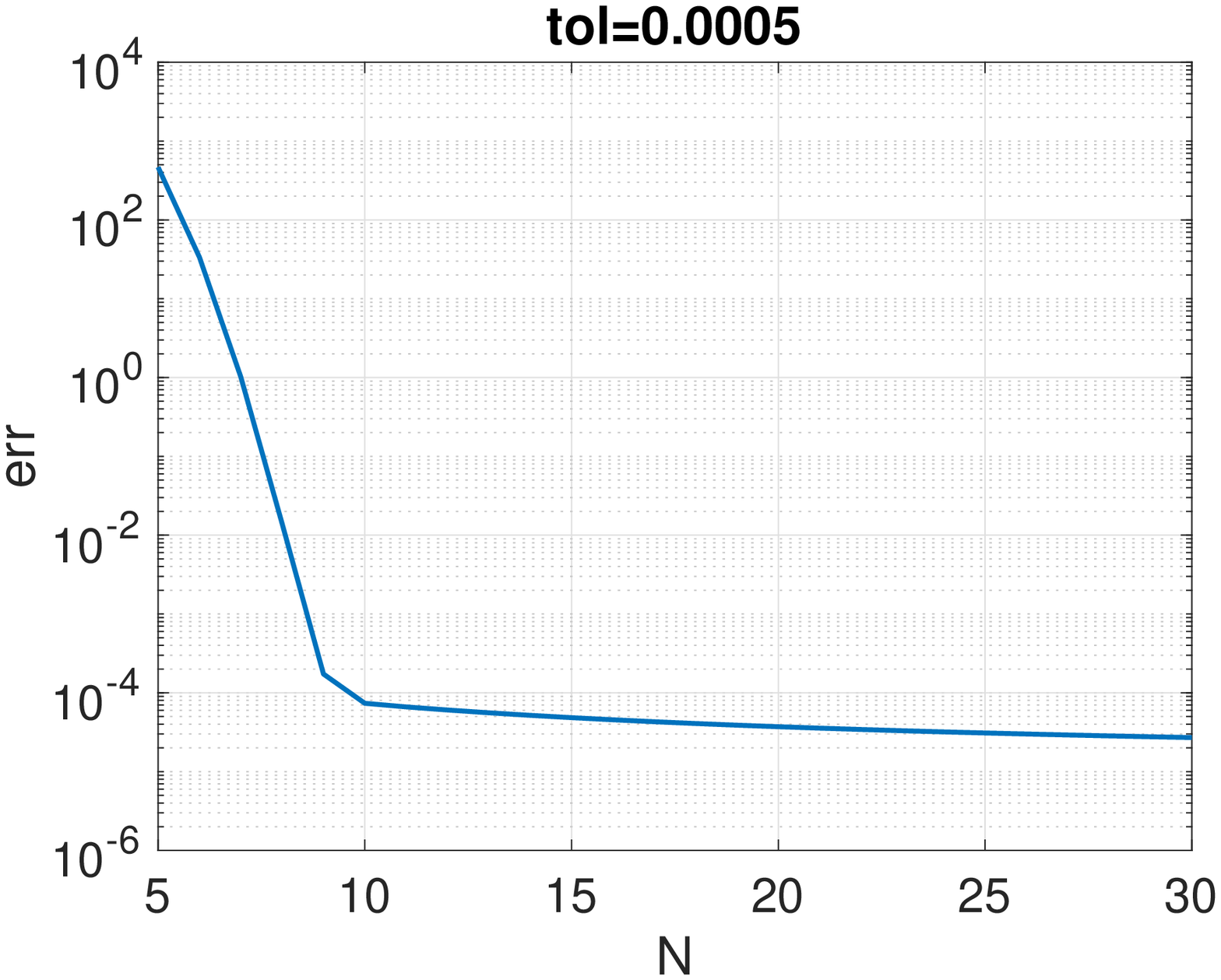}\\
\includegraphics[scale=0.35]{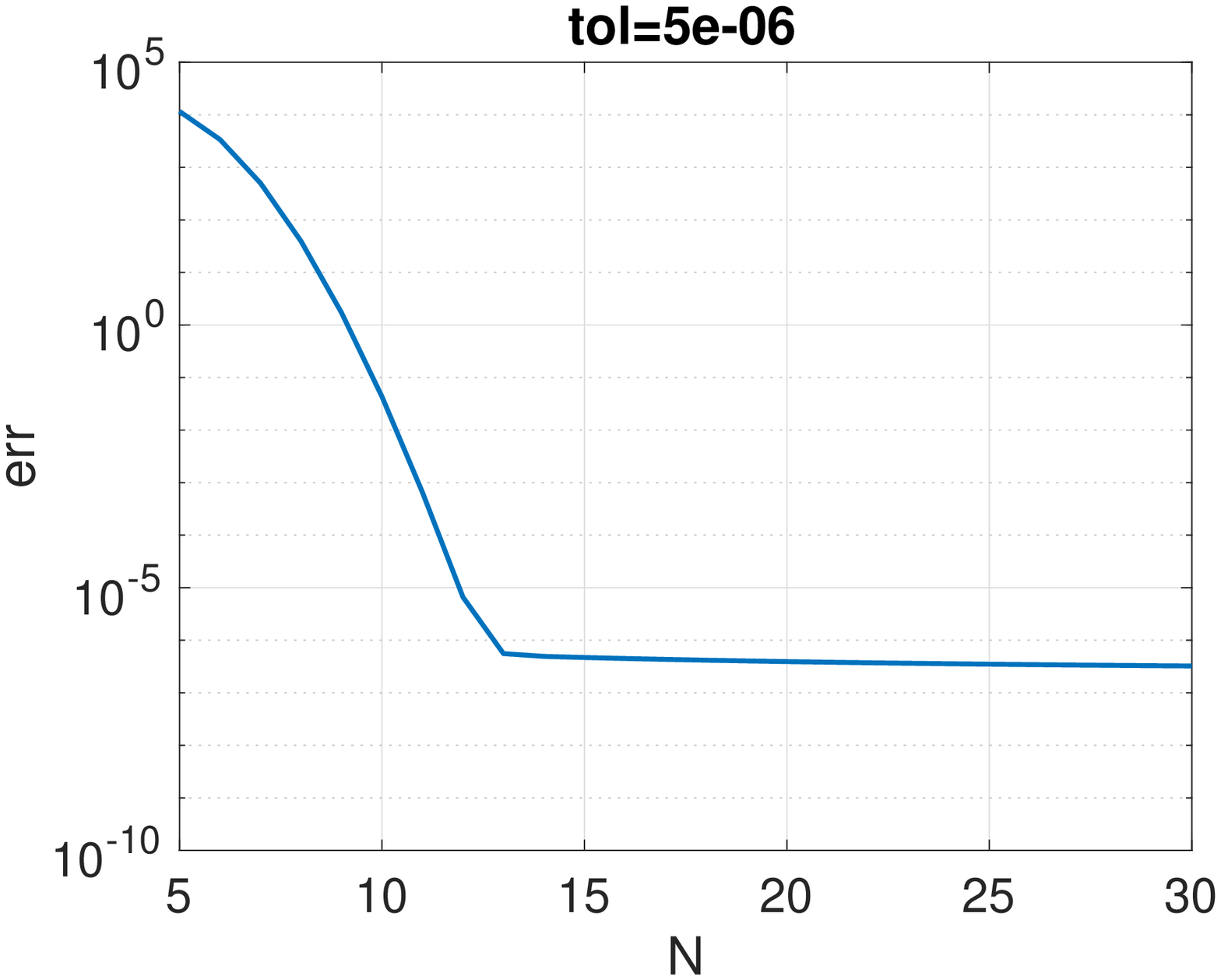}\includegraphics[scale=0.35]{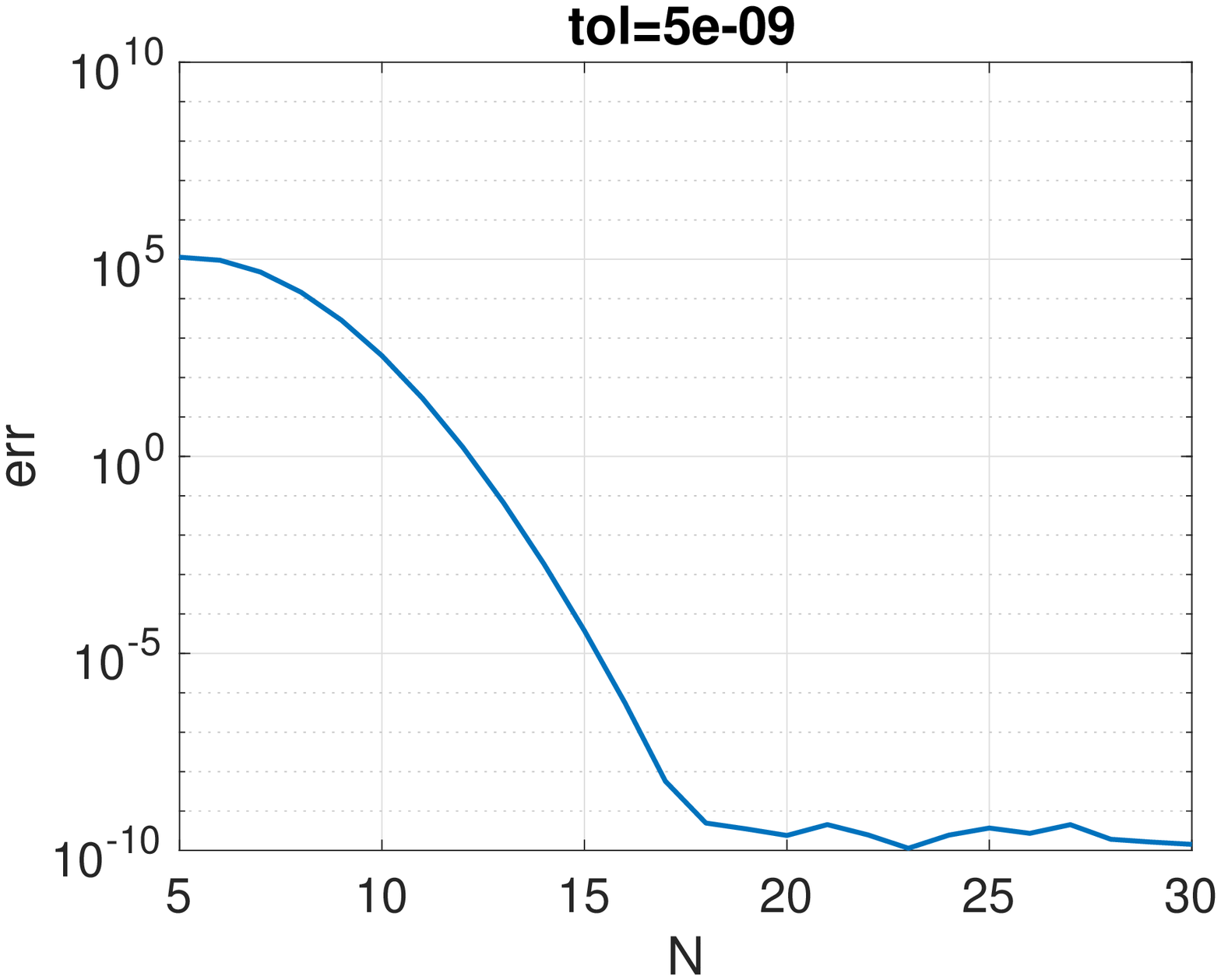}
\end{center}
\caption{Error vs number of nodes for Black-Scholes, $t=10$ ($z_l=-4,z_r=0.01$). \textit{Top left:} $tol=5\cdot 10^{-2}$ (maximal precision attainable $\approx 10^{-12}$). \textit{Top right:} $tol=5\cdot 10^{-4}$ (max. prec. $\approx 10^{-11}$). \textit{Bottom left:} $tol=5\cdot 10^{-6}$ (max. prec. $\approx10^{-11}$). \textit{Bottom right:} $tol=5\cdot 10^{-9}$ (max. prec. $\approx10^{-10}$).}\label{BSt10}
\end{figure}
% \newpage
\subsection{Heston equation} % \label{subHest}
The Heston equation \cite{He} is
\begin{equation}\label{Heston}
\frac{\partial u}{\partial \tau}=\frac{1}{2}s^2v\frac{\partial^2 u}{\partial s^2}+\rho\sigma sv\frac{\partial^2 u}{\partial s\partial v}+\frac{1}{2}\sigma^2v\frac{\partial^2 u}{\partial v^2}+(r_d-r_f)s\frac{\partial u}{\partial s}+\kappa(\eta-v)\frac{\partial u}{\partial v}-r_du\,.
\end{equation}
The unknown function $u(s,v,\tau)$ represents the price of an European option when at time $t-\tau$ the corresponding asset price is equal to $s$ and its variance is $v$. We consider the equation on the unbounded domain
\[
	0\leq \tau\leq t\,,\quad s>0\,,v>0
\]where the time $t$ is fixed. The parameters $\kappa>0$, $\sigma>0$ and $\rho\in[-1,1]$ are given. Moreover, equation \eqref{Heston} is usually considered under the condition $2\kappa\eta>\sigma^2$ that is known as Feller condition.
We take equation \eqref{Heston} together with the initial condition
\[
u(s,v,0)=\max(0,s-K)
\]
where $K>0$ is fixed \textit{a priori} (and represents the strike price of the option), and boundary conditions
\[
u(L,v,\tau)=0\,,\quad 0\leq \tau\leq t\,.
\]
For the numerical solution of \eqref{Heston}, we need to choose a bounded domain of integration. In particular we fix two positive constants $S,V$ and we let the two variables $s,v$ vary in the set
\[
	0\leq s\leq S\,,\quad 0\leq v\leq V\,.
\]On the new boundary, we need to add two more conditions (specific for the European call option)
\begin{eqnarray}\label{boundCond3}
&& \frac{\partial u}{\partial s}(S,v,\tau) = e^{-r_f\tau}\,,\quad 0\leq \tau\leq t
\\
&& u(s,V,\tau)=se^{-r_f\tau}\,,\quad 0\leq \tau\leq t\,.  \nonumber
\end{eqnarray}

The spatial discretization we adopted is the one introduced in \cite{ITHF}. We take $\kappa=1.5, \eta = 0.04, \sigma = 0.3, \rho=-0.9, r_d = 0.025, r_f=0, K = 100, L = 0, S = 8K, V = 5$. We plot the error for a selection of tolerances for the cases $t=1$ (Fig. \ref{Ht1}) and $t=10$ (Fig. \ref{Ht10}).
\begin{figure}[h!]
\begin{center}
\includegraphics[scale=0.35]{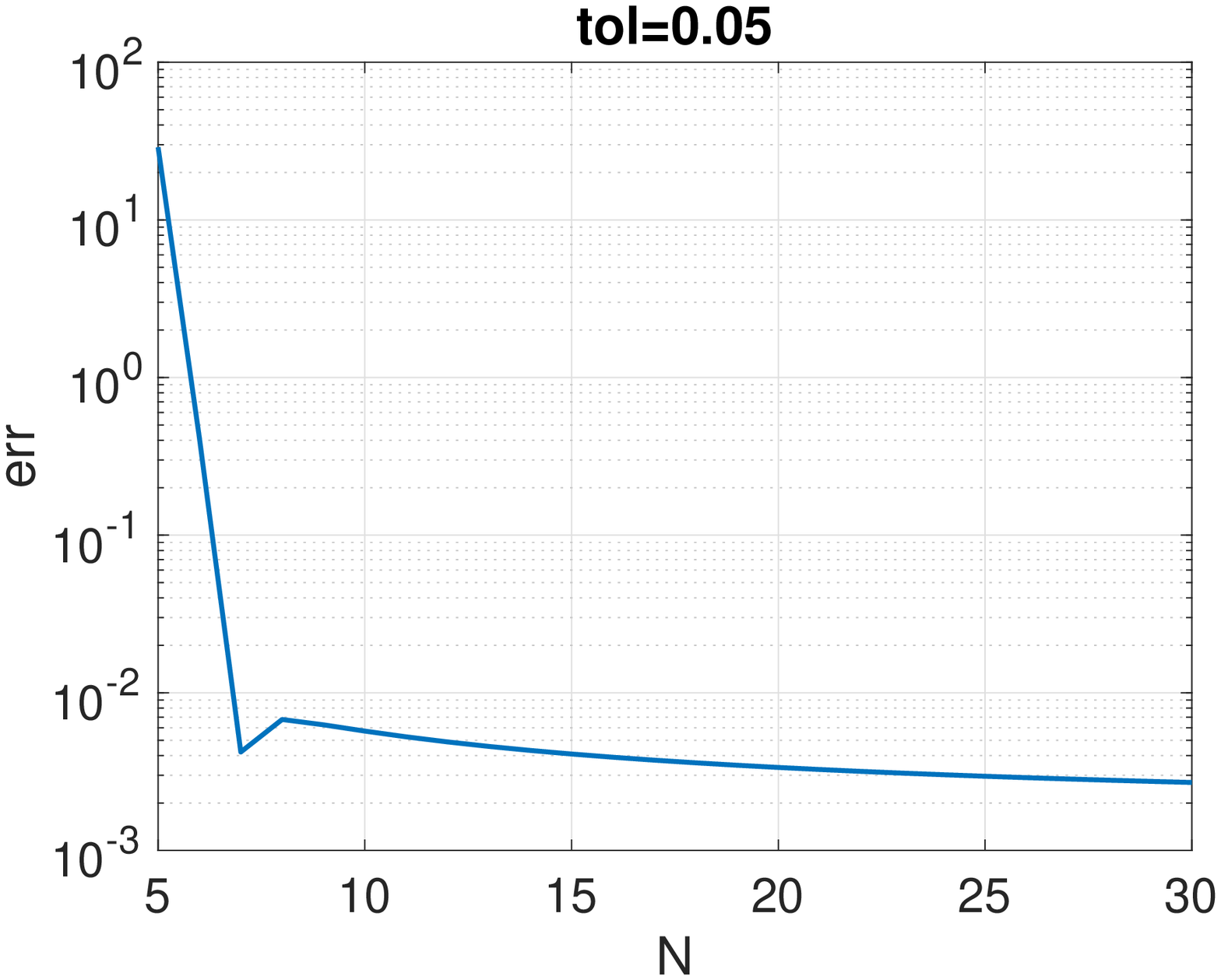}\includegraphics[scale=0.35]{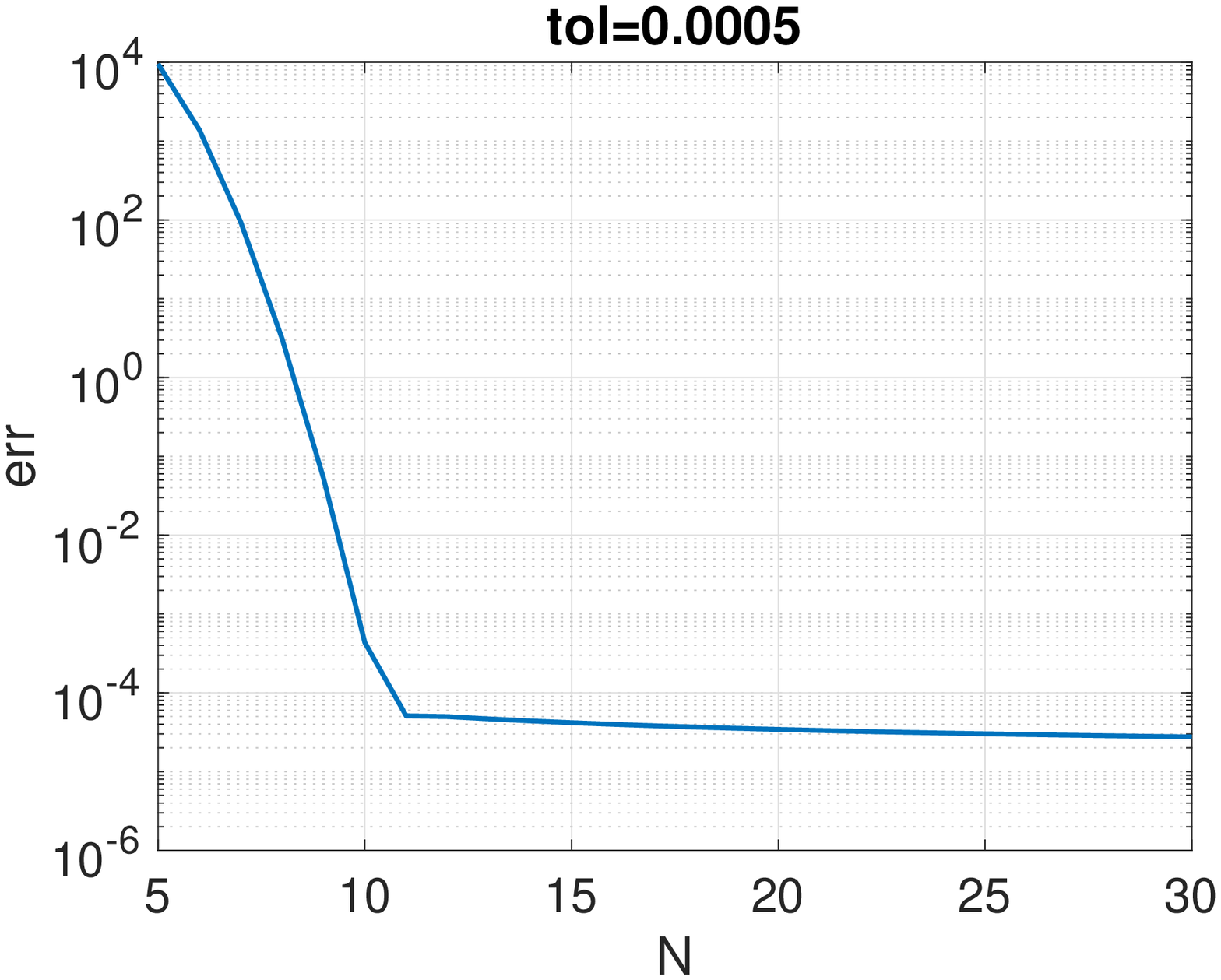}\\
\includegraphics[scale=0.35]{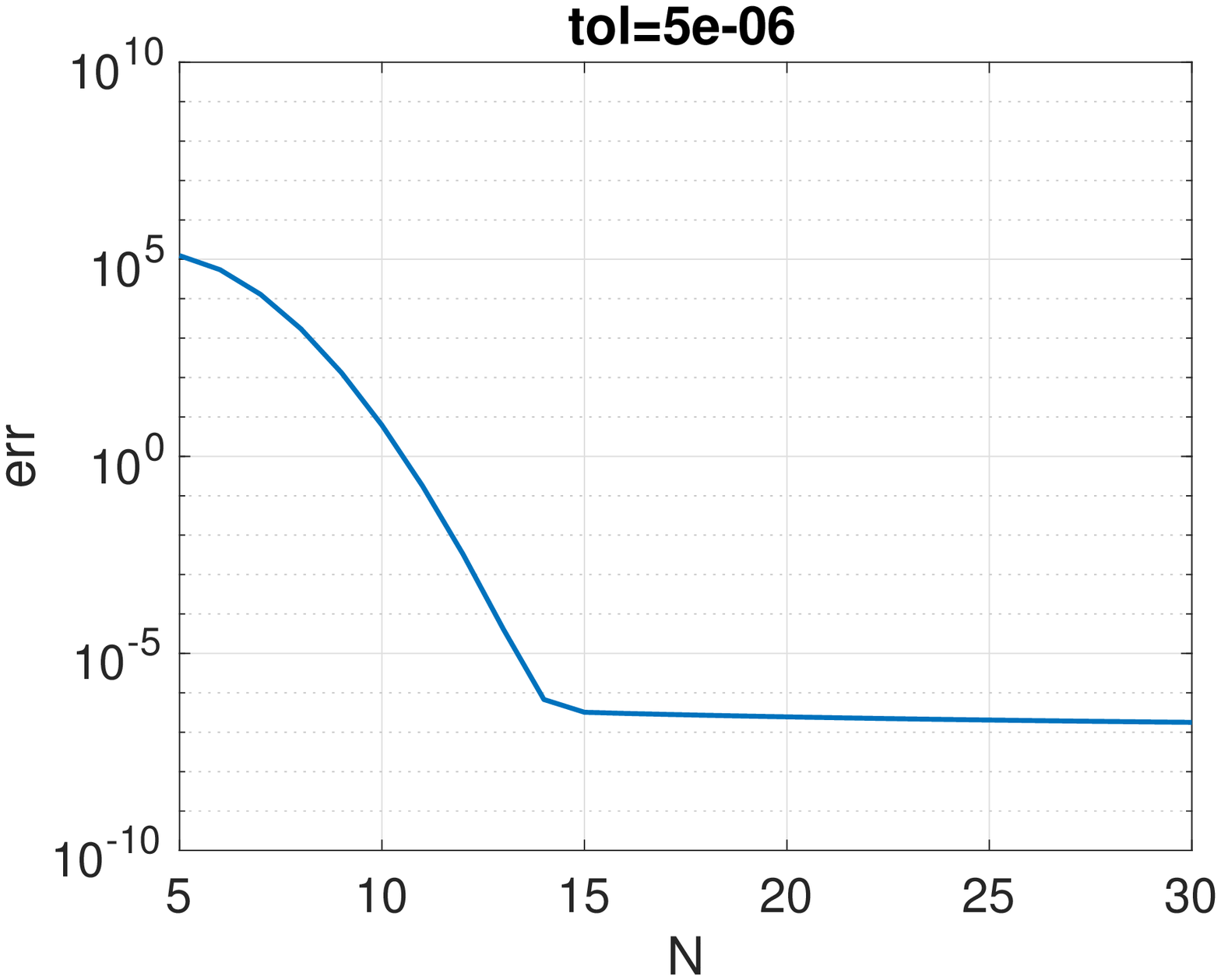}\includegraphics[scale=0.35]{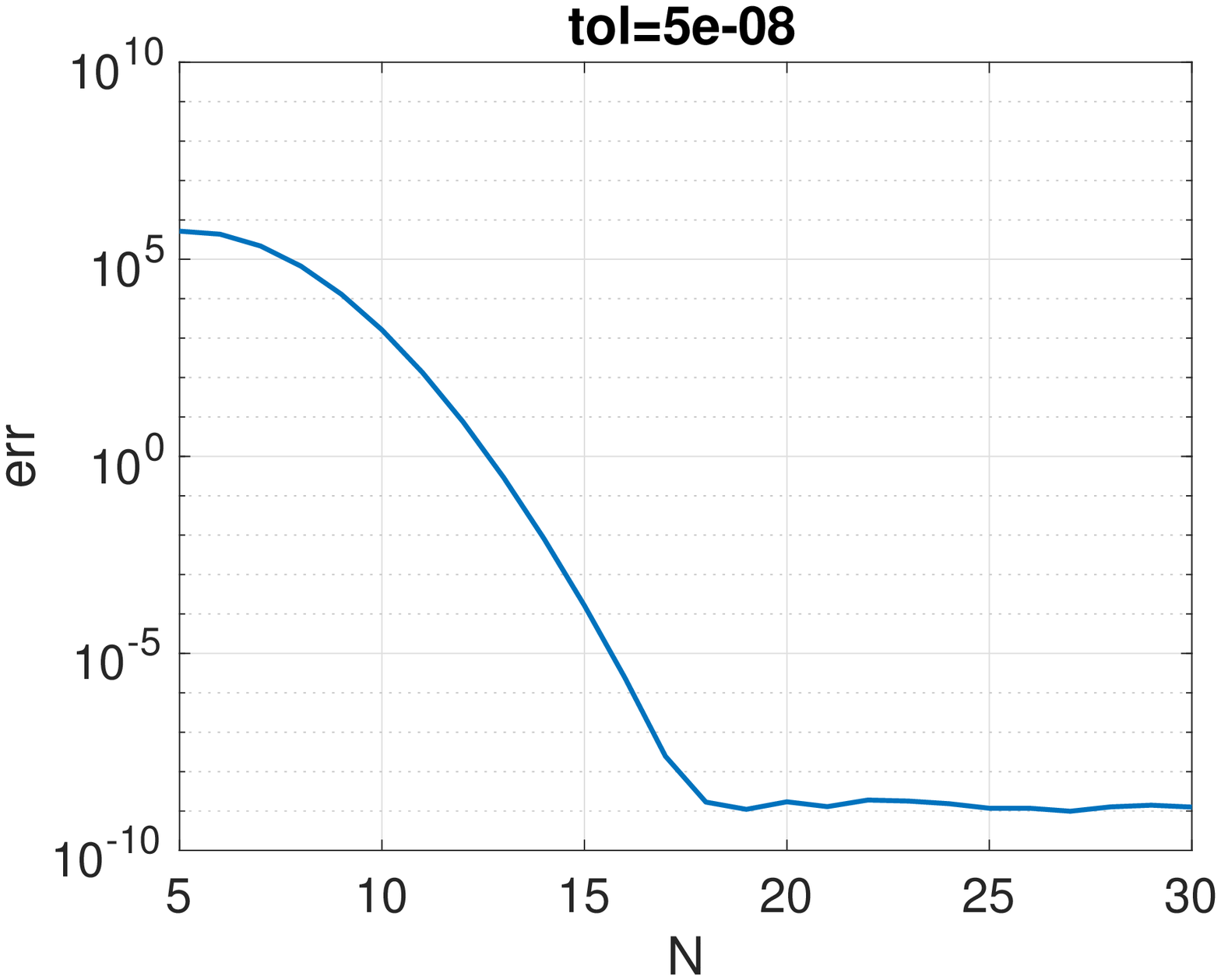}
\end{center}
\caption{Error vs number of nodes for Heston, $t=1$ ($z_l=-40,z_r=0.09$). \textit{Top left:} $tol=5\cdot 10^{-2}$ (maximal precision attainable $\approx 10^{-9}$). \textit{Top right:} $tol=5\cdot 10^{-4}$ (max. prec. $\approx 10^{-9}$). \textit{Bottom left:} $tol=5\cdot 10^{-6}$ (max. prec. $\approx10^{-8}$). \textit{Bottom right:} $tol=5\cdot 10^{-8}$ (max. prec. $\approx10^{-7}$, but in practice we are still able to reach a sharper precision).}\label{Ht1}
\end{figure}

\begin{figure}[h!]
\begin{center}
\includegraphics[scale=0.35]{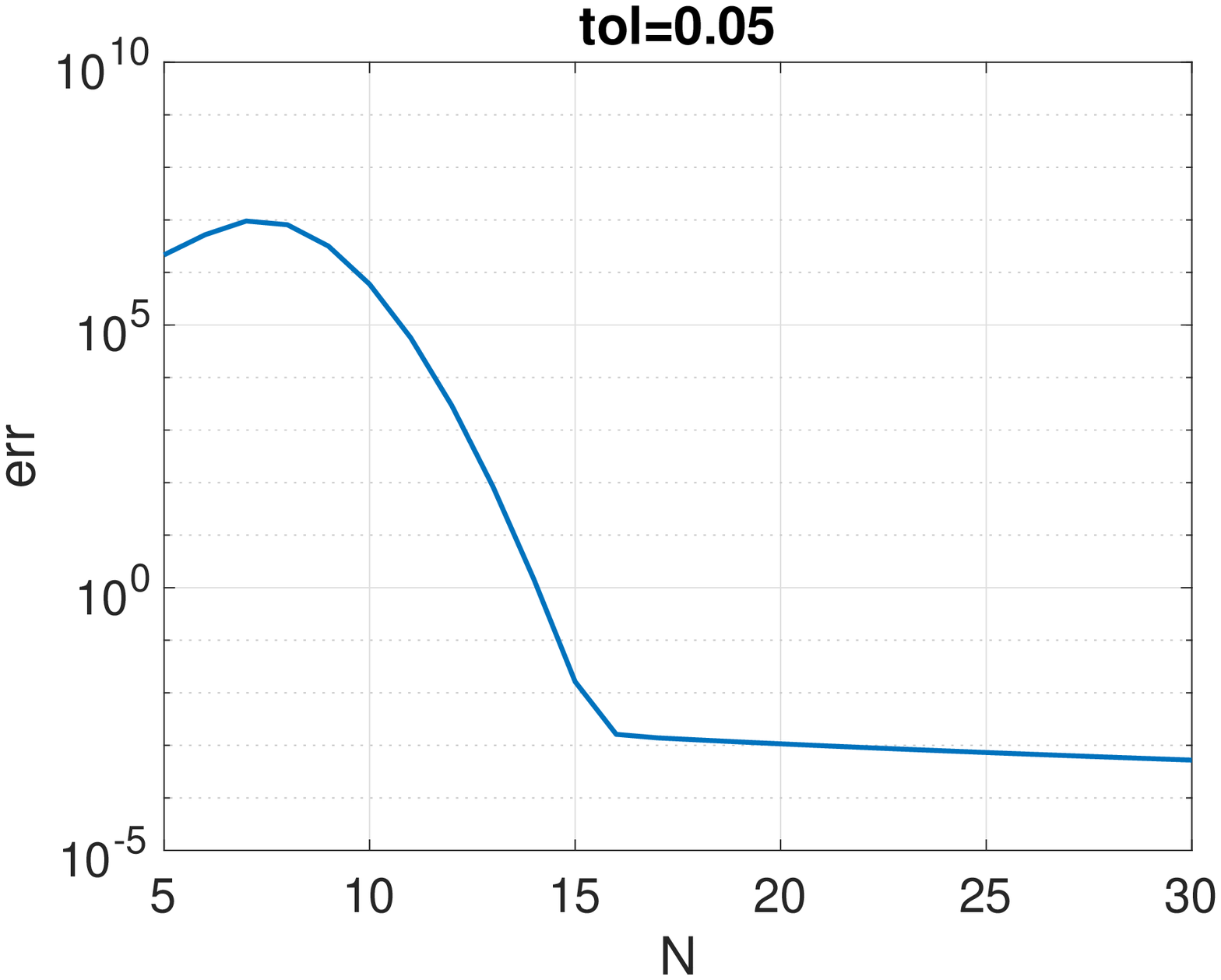}\includegraphics[scale=0.35]{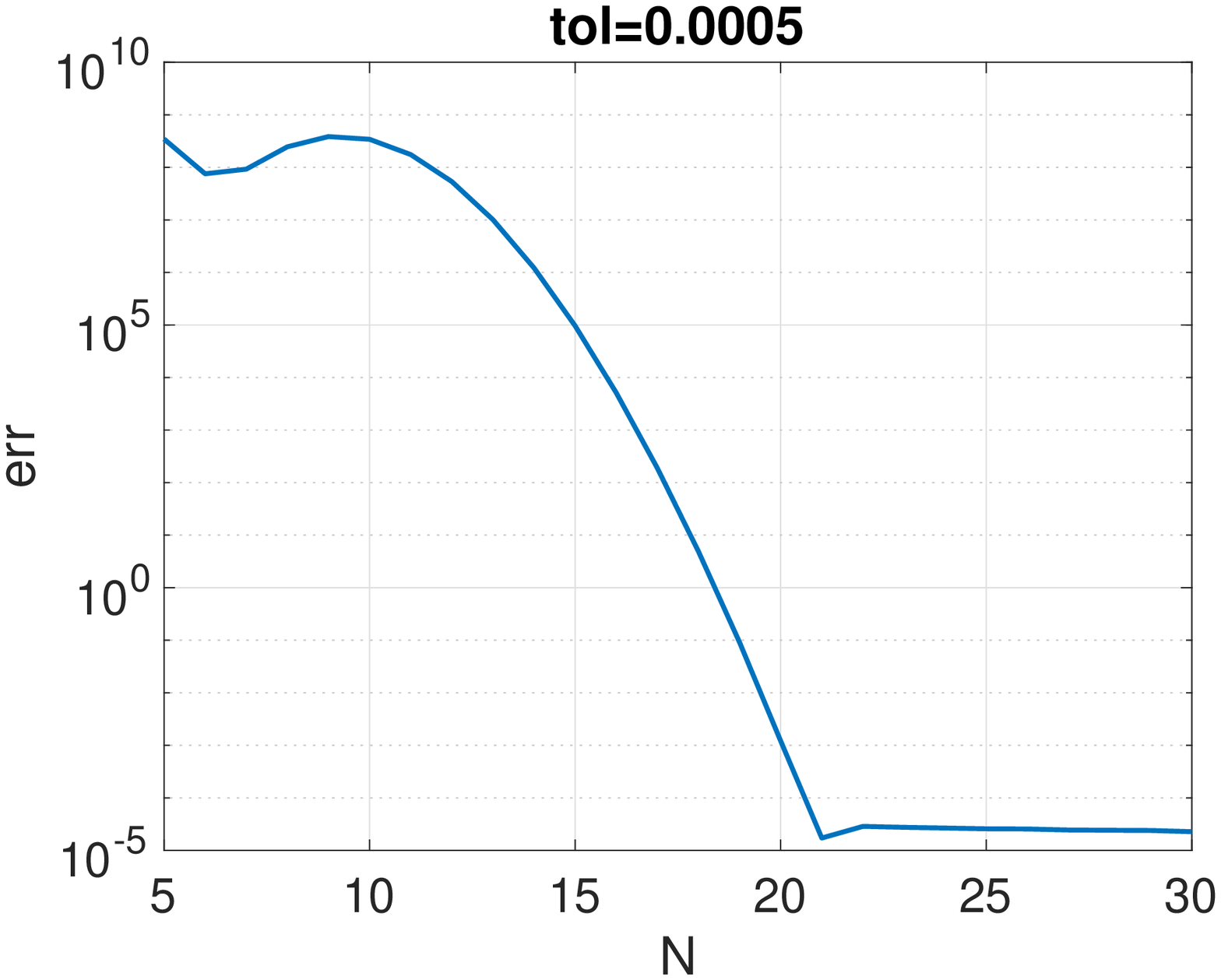}\\
\includegraphics[scale=0.35]{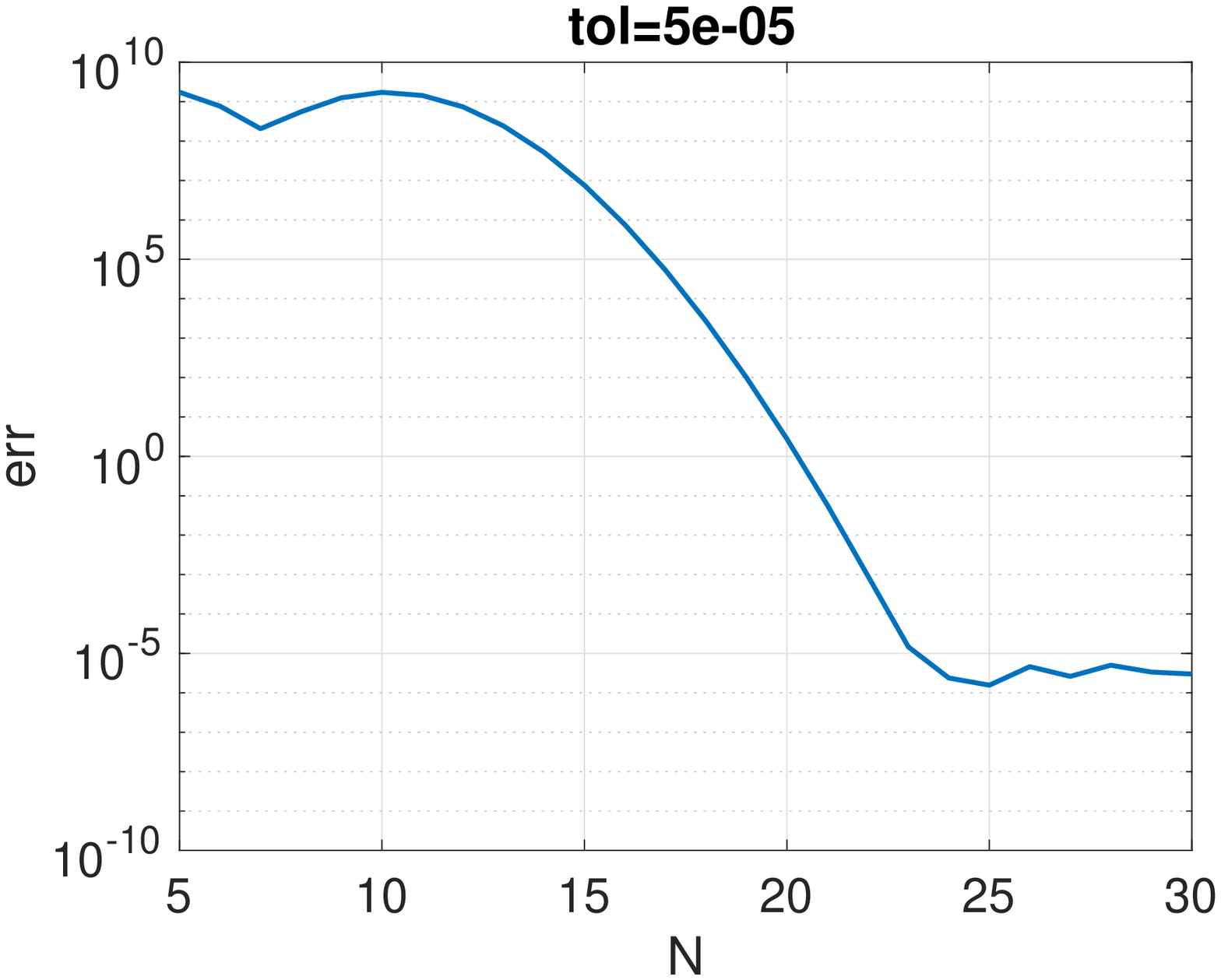}\includegraphics[scale=0.35]{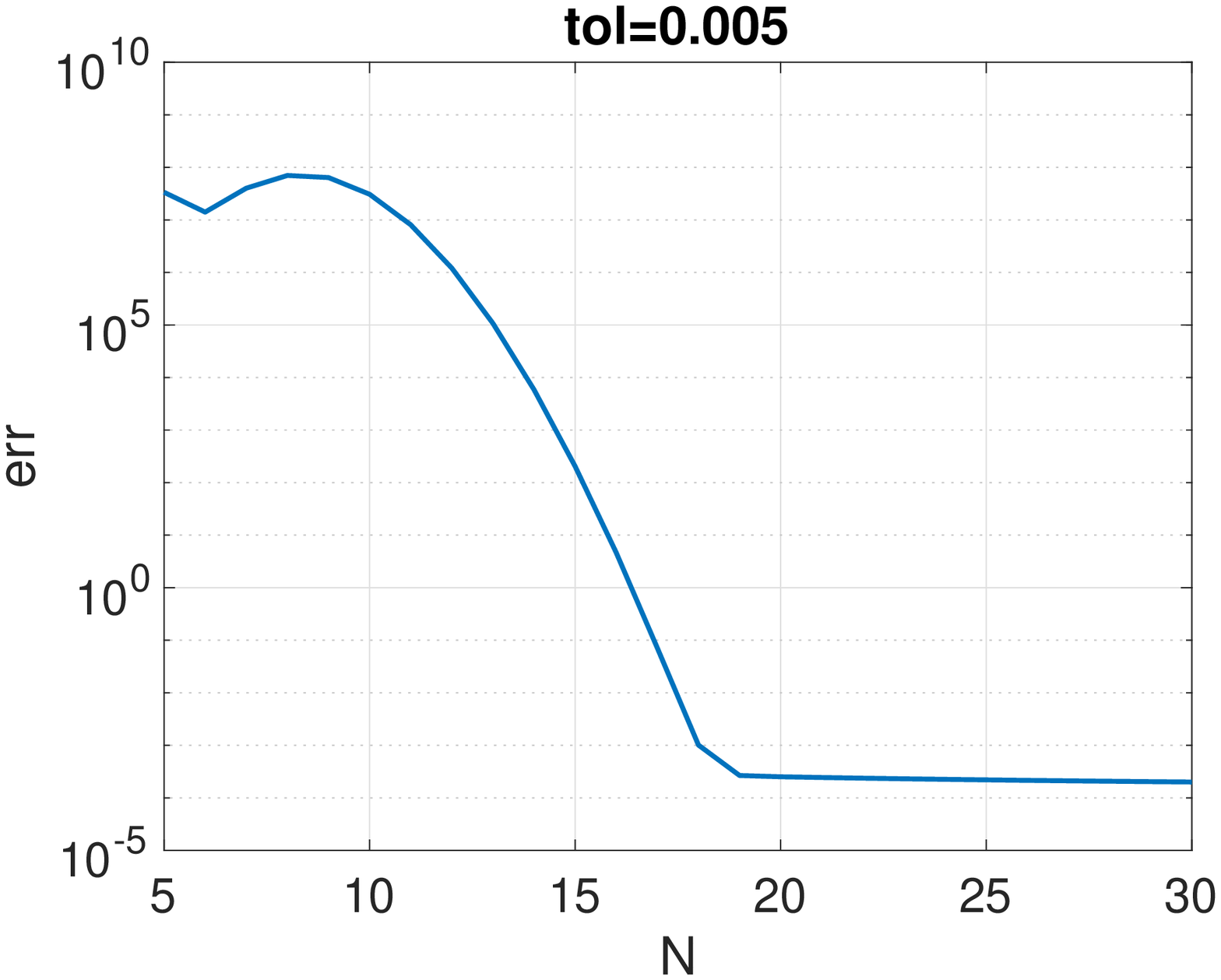}
\end{center}
\caption{Error vs number of nodes for Heston, $t=10$ ($z_l=-4,z_r=0.06$). \textit{Top left:} $tol=5\cdot 10^{-2}$ (maximal precision attainable $\approx 10^{-8}$). \textit{Top right:} $tol=5\cdot 10^{-4}$ (max. prec. $\approx 10^{-6}$). \textit{Bottom left:} $tol=5\cdot 10^{-5}$ (max. prec. $\approx10^{-5}$). \textit{Bottom right:} $tol=5\cdot 10^{-6}$ (max. prec. $\approx10^{-4}$, but the real accuracy we can get is sharper).}\label{Ht10}
\end{figure}
% \newpage
\section{Comparison with other methods}\label{secComp}
\subsection{Comparison with parabolic contours} % \label{subCompParab}
A direct comparison with the method in \cite{ITHWeid} is not possible, since our algorithm works with the goal of a fixed accuracy while the one reported in \cite{ITHWeid} aims to reduce the error as $N$ grows. Anyway, for the sake of comparison, we can run our algorithm as follows:

\begin{itemize}
\item[-] for a set of target precisions ($tol=10^{-1},10^{-2},\ldots$ for example), we run our algorithm;
\item[-] for each tolerance we save the smallest number of quadrature nodes $N$ for which % the precision
$tol$ is reached;
\item[-] for each tolerance we save the corresponding error $err(N)$.
\end{itemize}

Once we get the array $[N,err(N)]$, we can compare it with the corresponding error coming from the method of \cite{ITHWeid}. We make our experiments both for Black-Sholes and for Heston equation, for a selection of times $t$. In Figure \ref{compBSW} the comparison for Black-Scholes is depicted
\begin{figure}[h!]
\begin{center}
\includegraphics[scale=0.35]{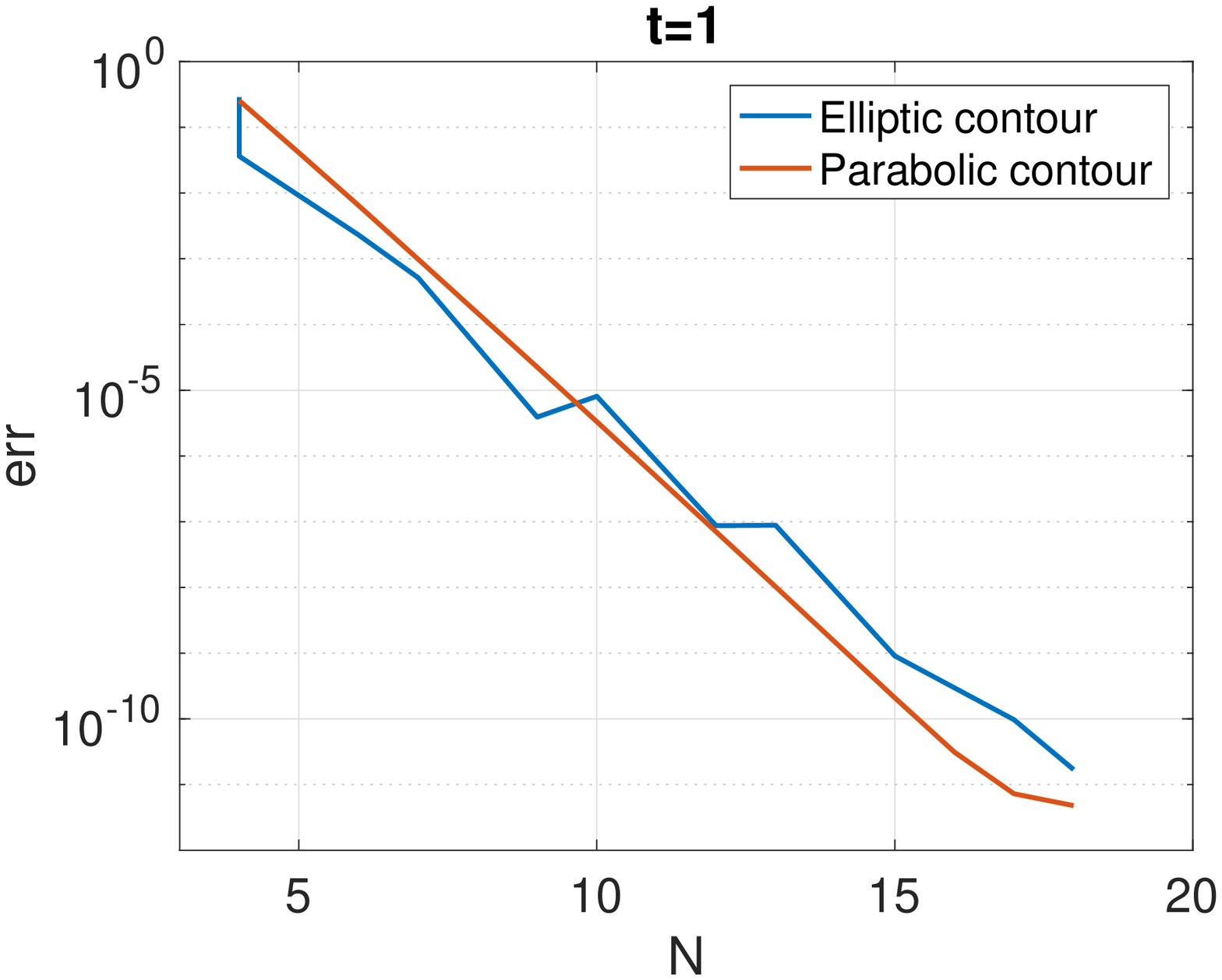}\includegraphics[scale=0.35]{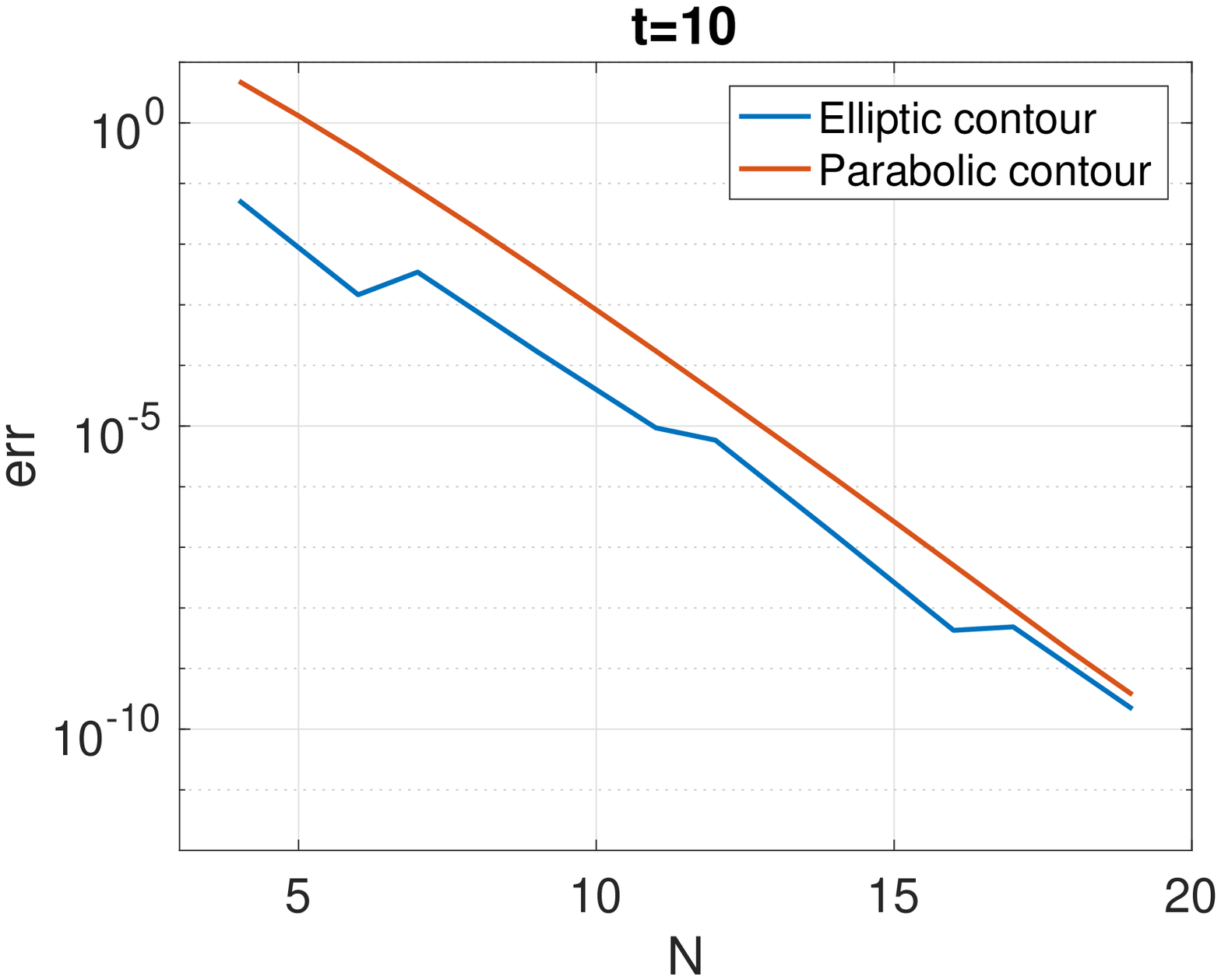}\\

\includegraphics[scale=0.35]{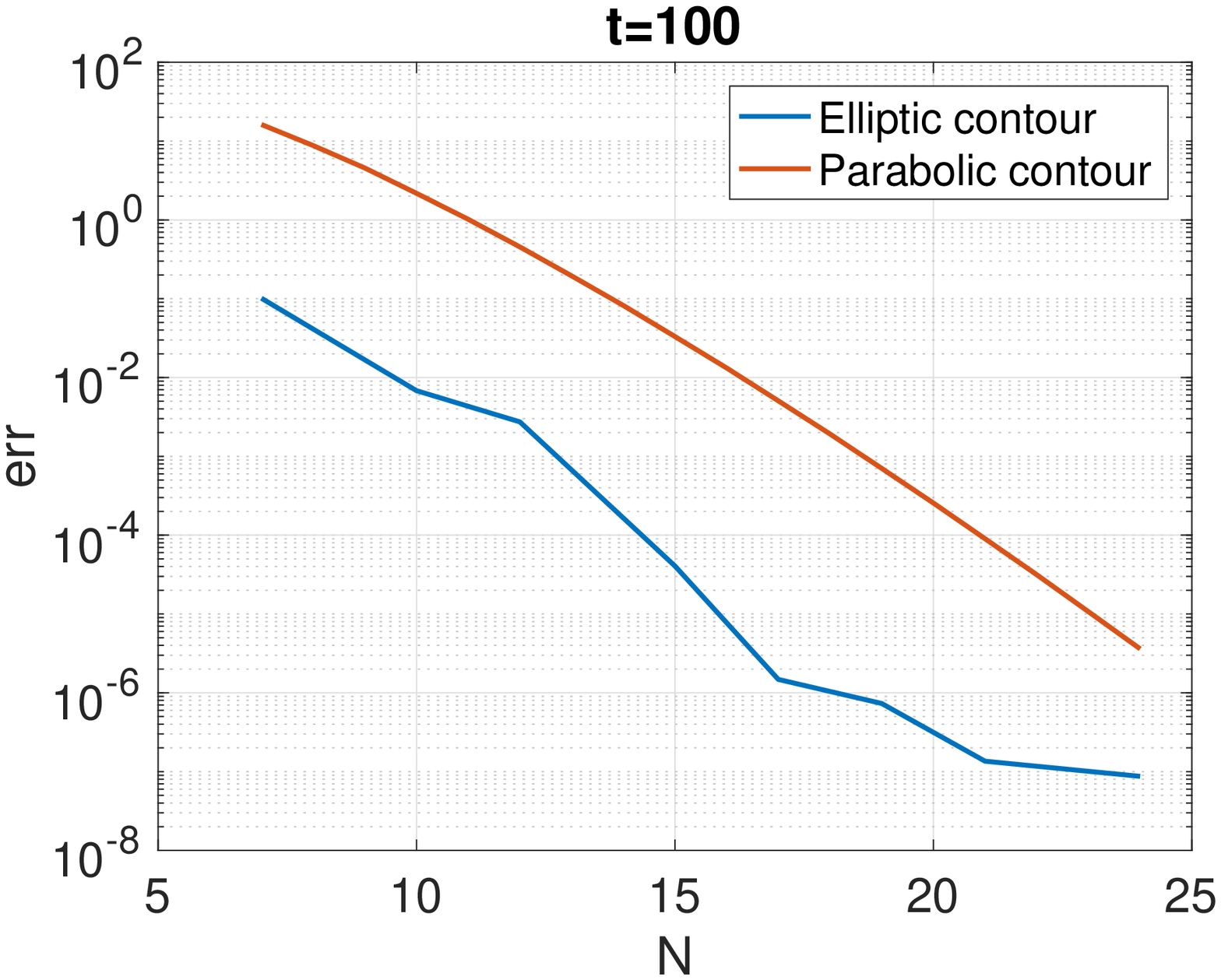}
\caption{Black-Scholes equation, comparison between \cite{ITHWeid} and our method. \textit{Top left:} $t=1$, $z_l=-40$, $z_r=0.05$. \textit{Top right:} $t=10$, $z_l=-4$, $z_r=0.01$. \textit{Bottom: }$t=100$, $z_l=-0.5$, $z_r=0.001$.}\label{compBSW}
\end{center}

\end{figure}

% \newpage
For the comparison in the case of Heston equation, we recall that the boundary condition considered in \cite{ITHWeid} is slightly different from \eqref{boundCond3}. However the spectrum of the discrete operator $A$ is quite similar, and  thus we implement the method in \cite{ITHWeid} using the same inner parabola. Qualitatively, the numerical results we get for $t=1$ are the same as the one in \cite{ITHWeid}. The comparison is showed in Figure \ref{compHW}
\begin{figure}[h!]
\begin{center}
\includegraphics[scale=0.35]{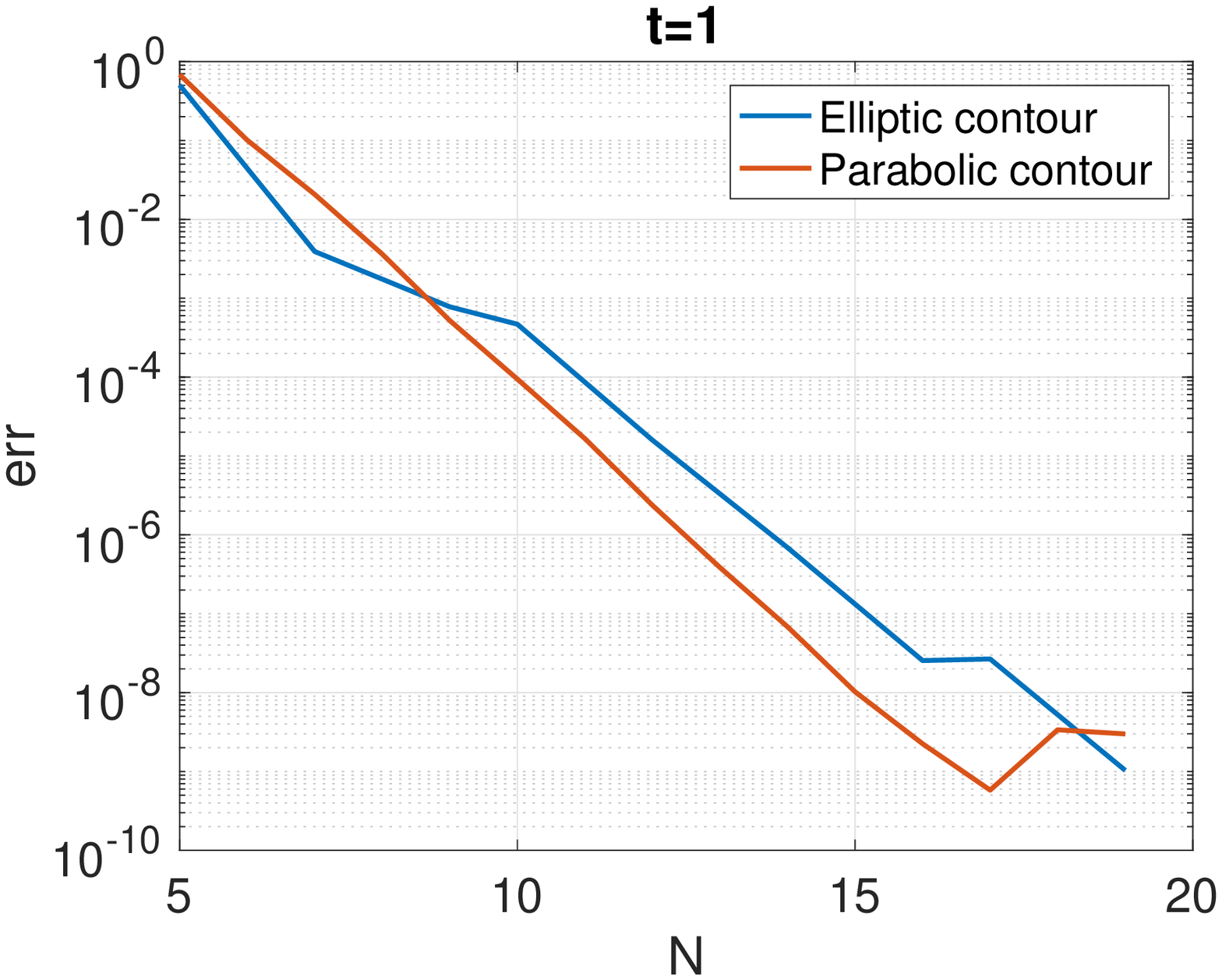}\includegraphics[scale=0.35]{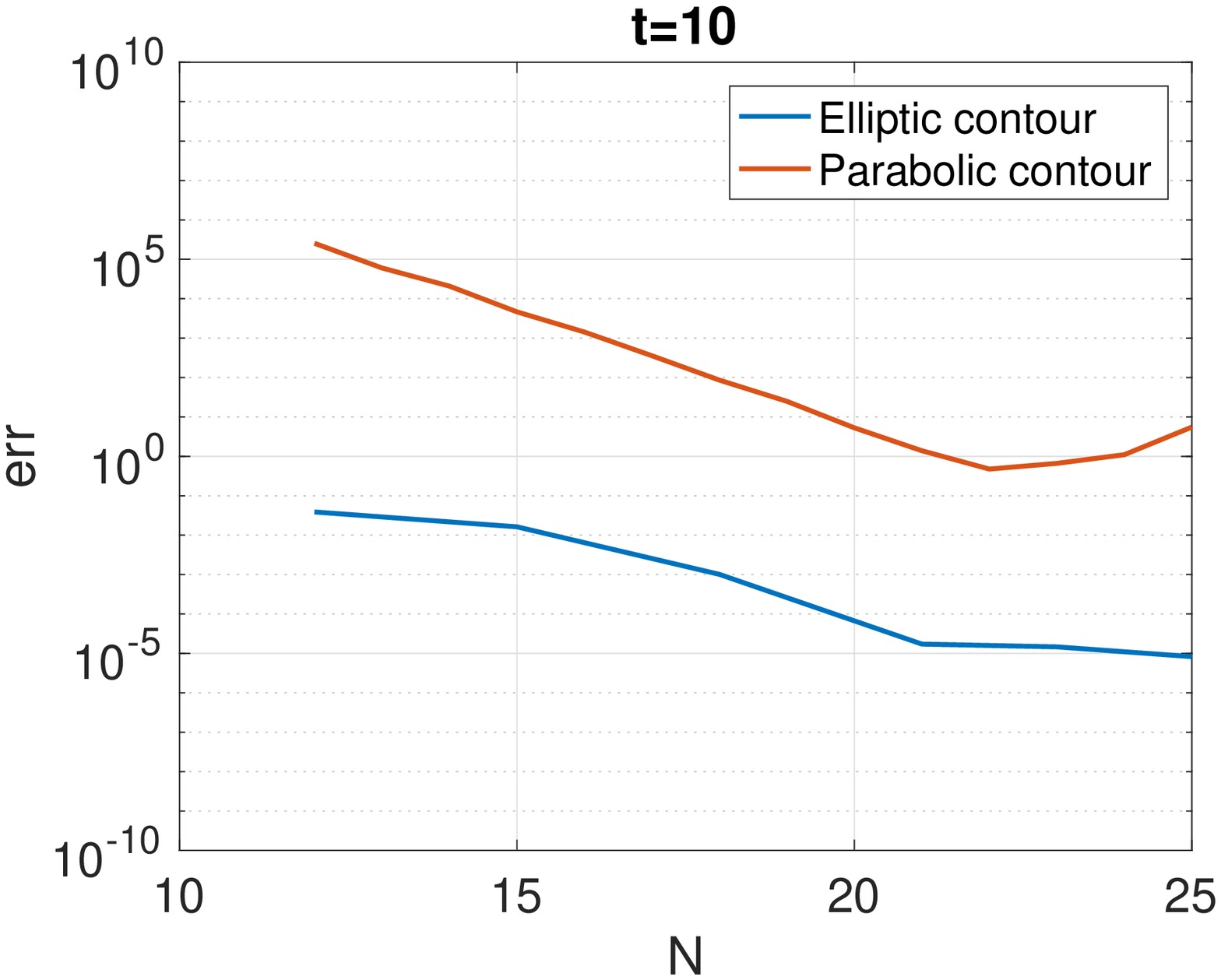}
\end{center}
\caption{Heston equation, comparison between \cite{ITHWeid} and our method. \textit{Left:} $t=1$, $z_l=-40$, $z_r=0.09$. \textit{Right:} $t=10$, $z_l=-4$, $z_r=0.06$.}\label{compHW}
\end{figure}
% \newpage
\subsection{Comparison with hyperbolic contours} % \label{subCompHyp}
By using the same startegy as the previous subsection, we build the array $[N,err(N)]$. We compare these results with the method in \cite{LPS} for both the cases of Black-Scholes and Heston equations. We set $\alpha=0.4$, $d=0.4$ as geometric parameters to bound the resolvent norm, as explained in \cite{LPS}. This choice turns out to be effective and the method seems to converge. However, in \cite{LPS} no optimality criteria are given to select $\alpha,d$ and they necessarily depend on the spectral geometry of $A$. Since \cite{LPS} works on time intervals of the form $[t_0,\Lambda t_0]$, we provide a comparison for a selection of time ratios $\Lambda$ and  for the times $t=1,10$. The results are reported in Figures \ref{compBSLP1}, \ref{compBSLP10}, \ref{compHLP1} and \ref{compHLP10}
\begin{figure}[h!]
\begin{center}
\includegraphics[scale=0.35]{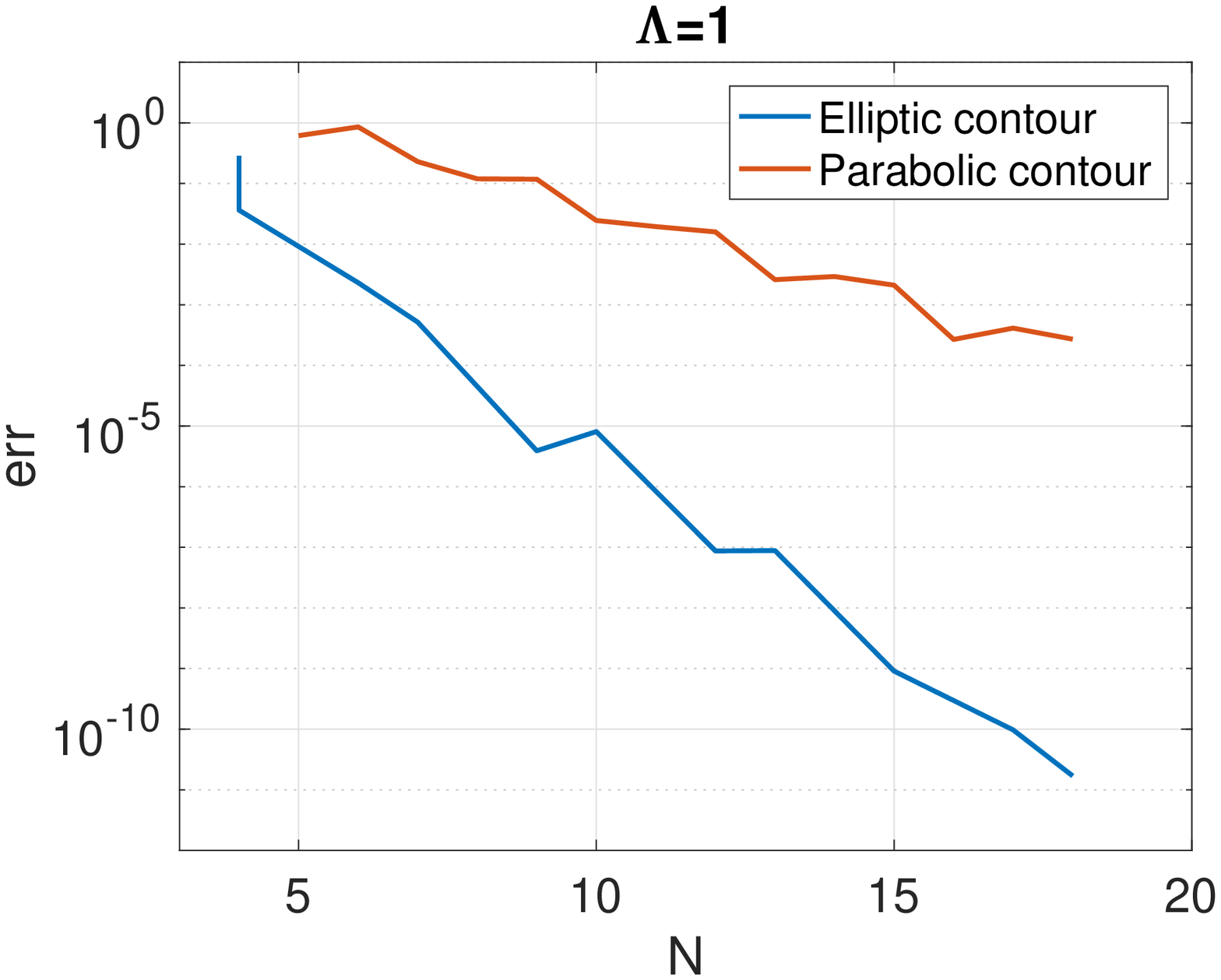}\includegraphics[scale=0.35]{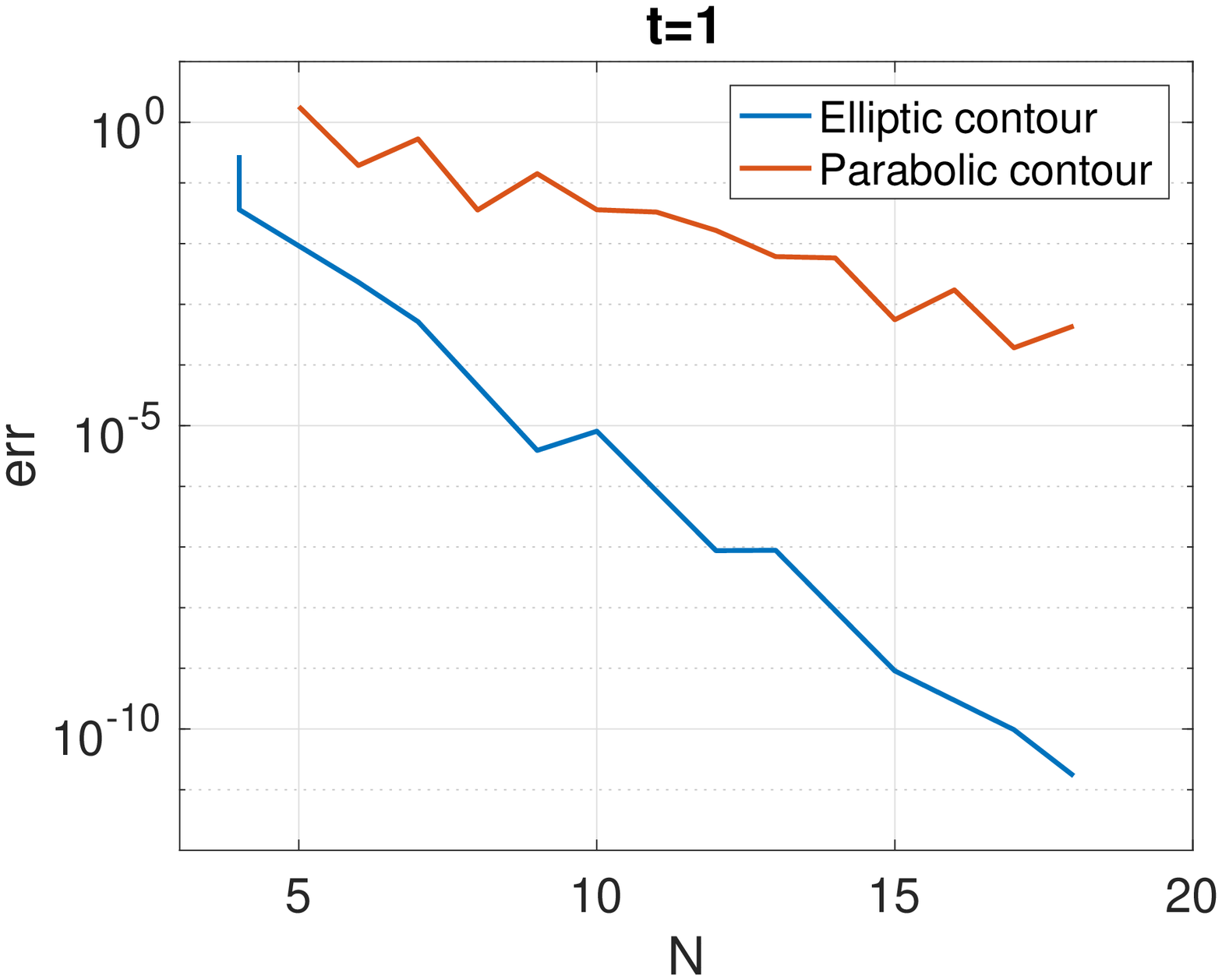}
\end{center}
\caption{Black-Scholes equation, comparison between \cite{LPS} and our method for $t=1$ $z_l=-40$, $z_r=0.05$. \textit{Left:} $\Lambda=1$. \textit{Right: }$\Lambda=1.5$.}\label{compBSLP1}
\end{figure}
\begin{figure}[h!]
\begin{center}
\includegraphics[scale=0.35]{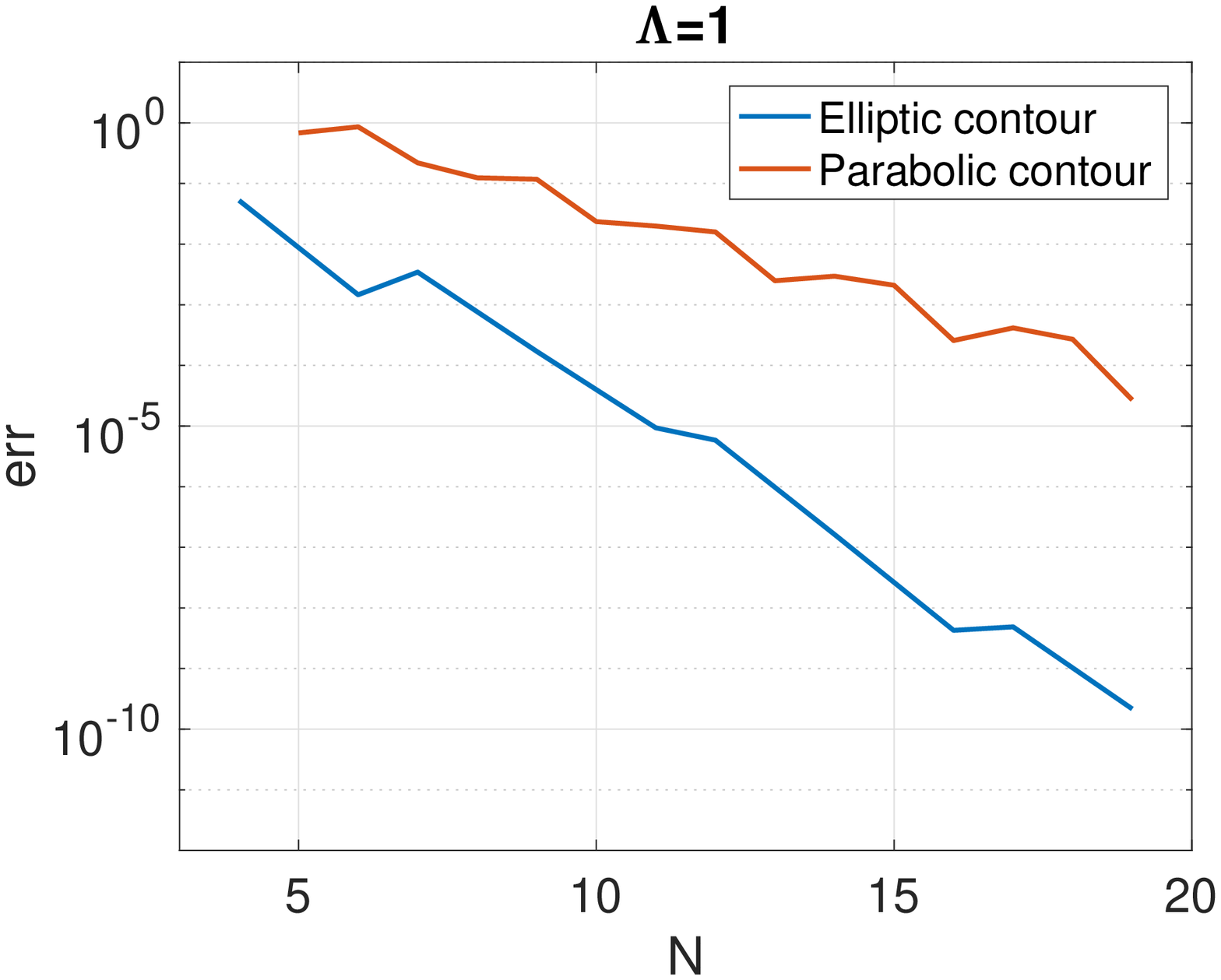}\includegraphics[scale=0.35]{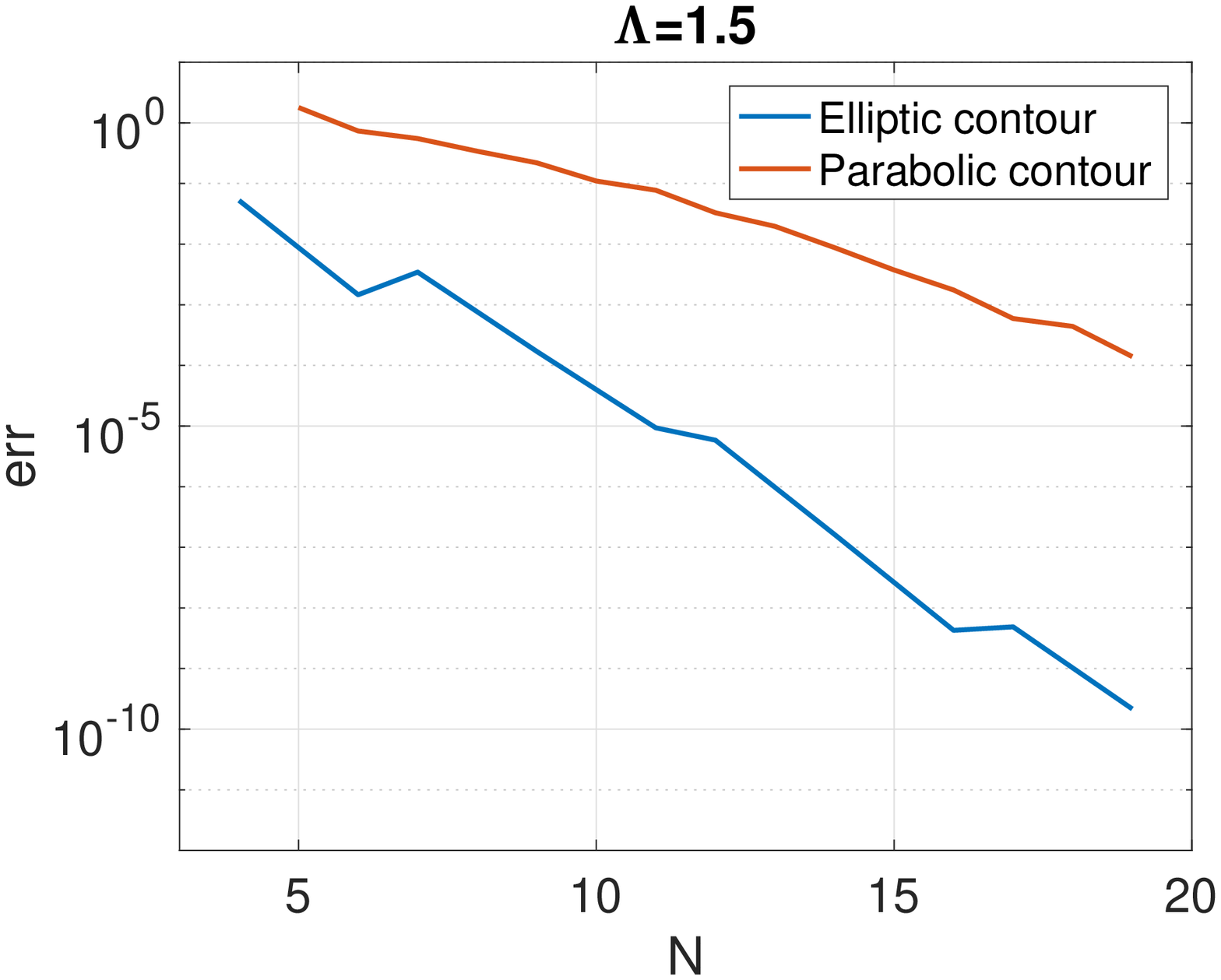}
\end{center}
\caption{Black-Scholes equation, comparison between \cite{LPS} and our method for $t=10$ $z_l=-4$, $z_r=0.01$. \textit{Left:} $\Lambda=1$. \textit{Right: }$\Lambda=1.5$.}\label{compBSLP10}
\end{figure}
\begin{figure}[h!]
\begin{center}
\includegraphics[scale=0.35]{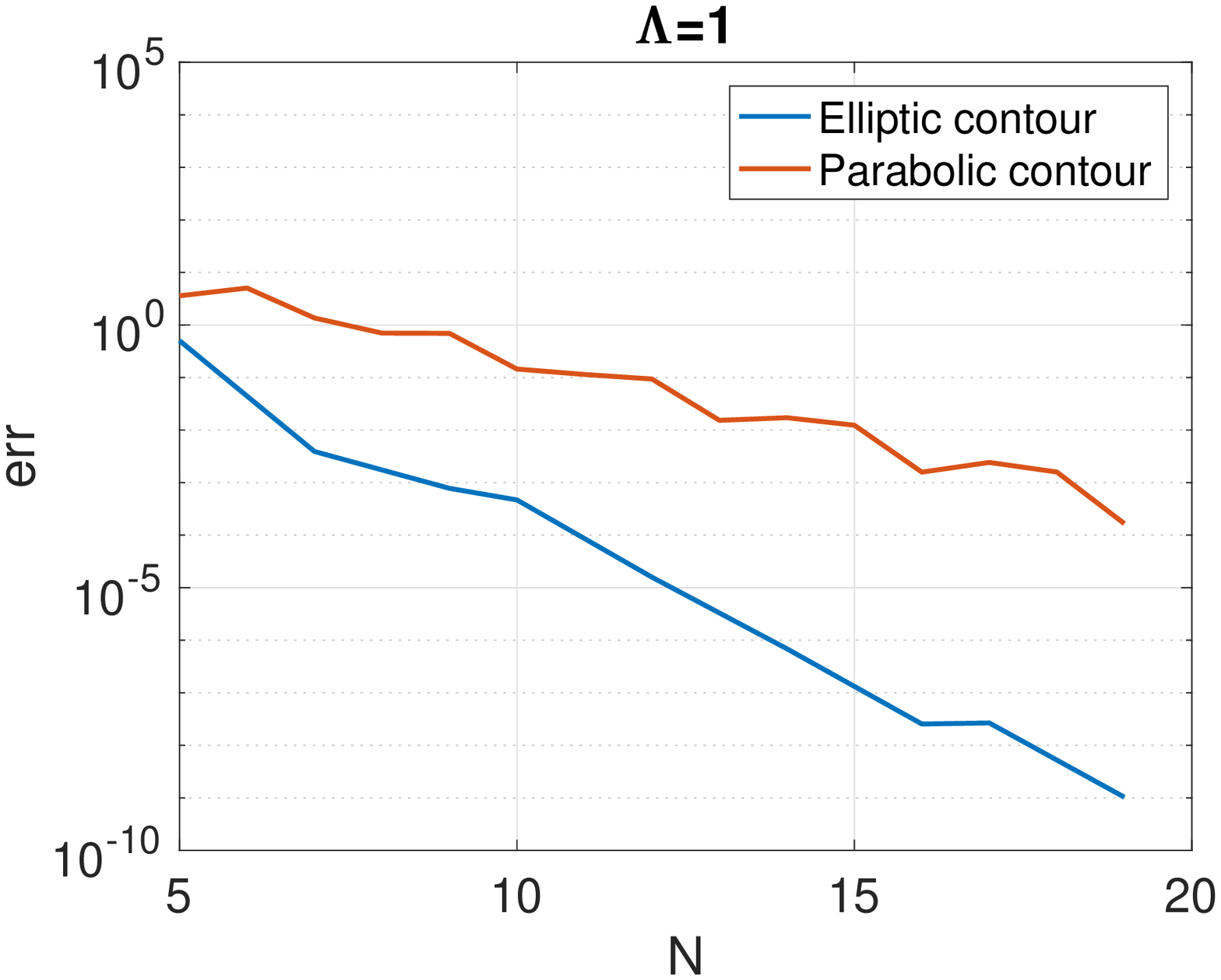}\includegraphics[scale=0.35]{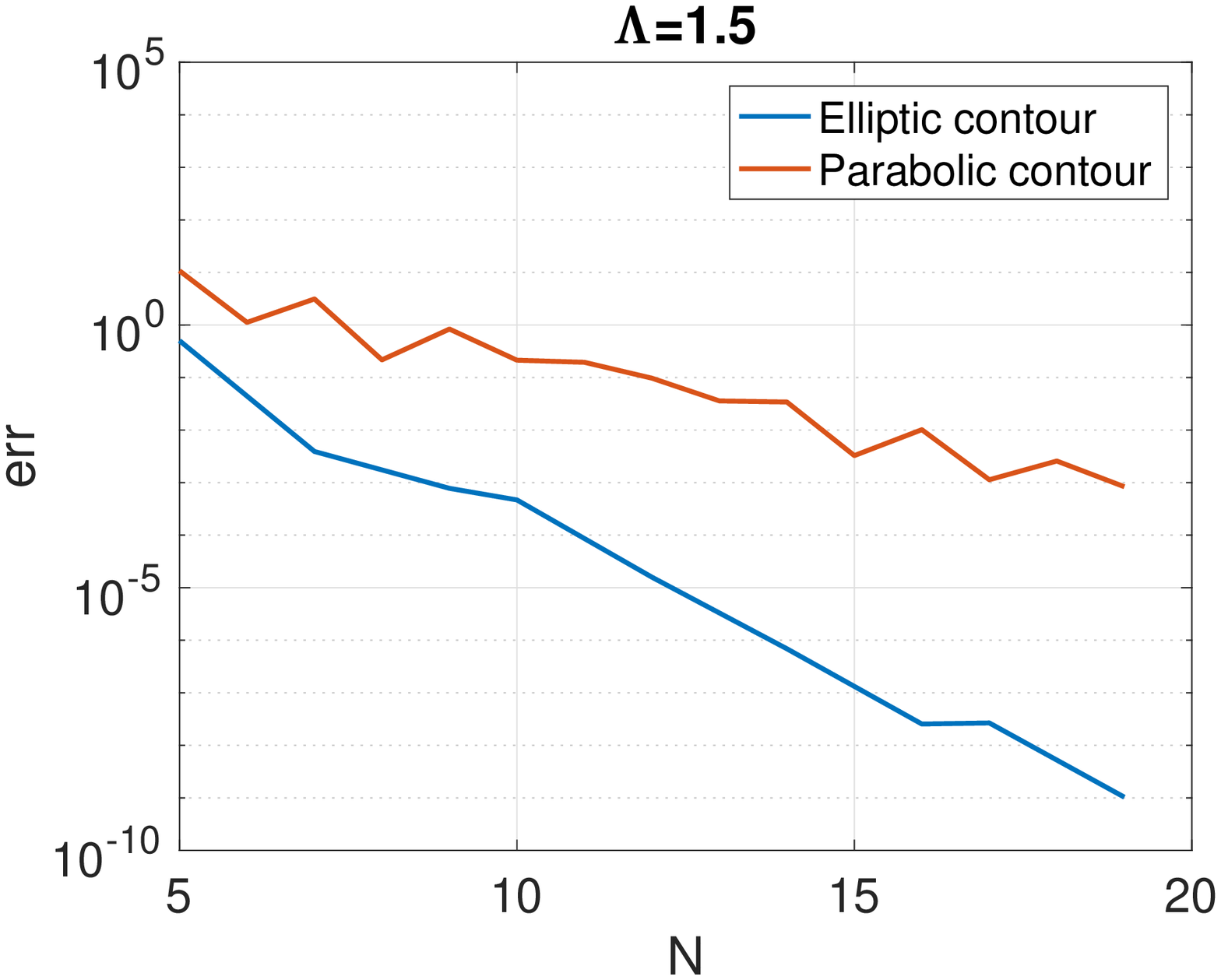}
\end{center}
\caption{Heston equation, comparison between \cite{LPS} and our method for $t=1$ $z_l=-40$, $z_r=0.09$. \textit{Left:} $\Lambda=1$. \textit{Right: }$\Lambda=1.5$.}\label{compHLP1}
\end{figure}
\begin{figure}[h!]\begin{center}
\includegraphics[scale=0.35]{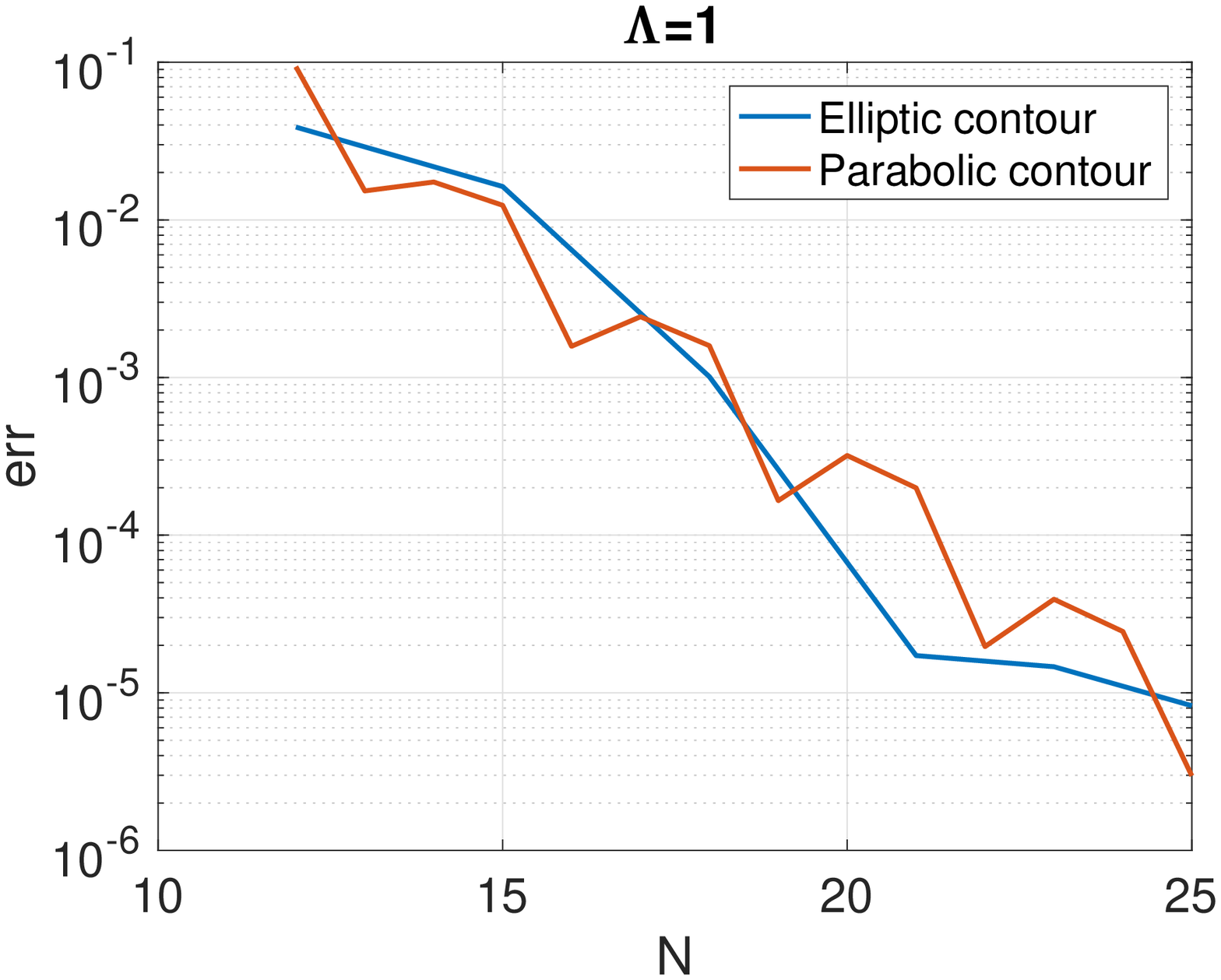}\includegraphics[scale=0.35]{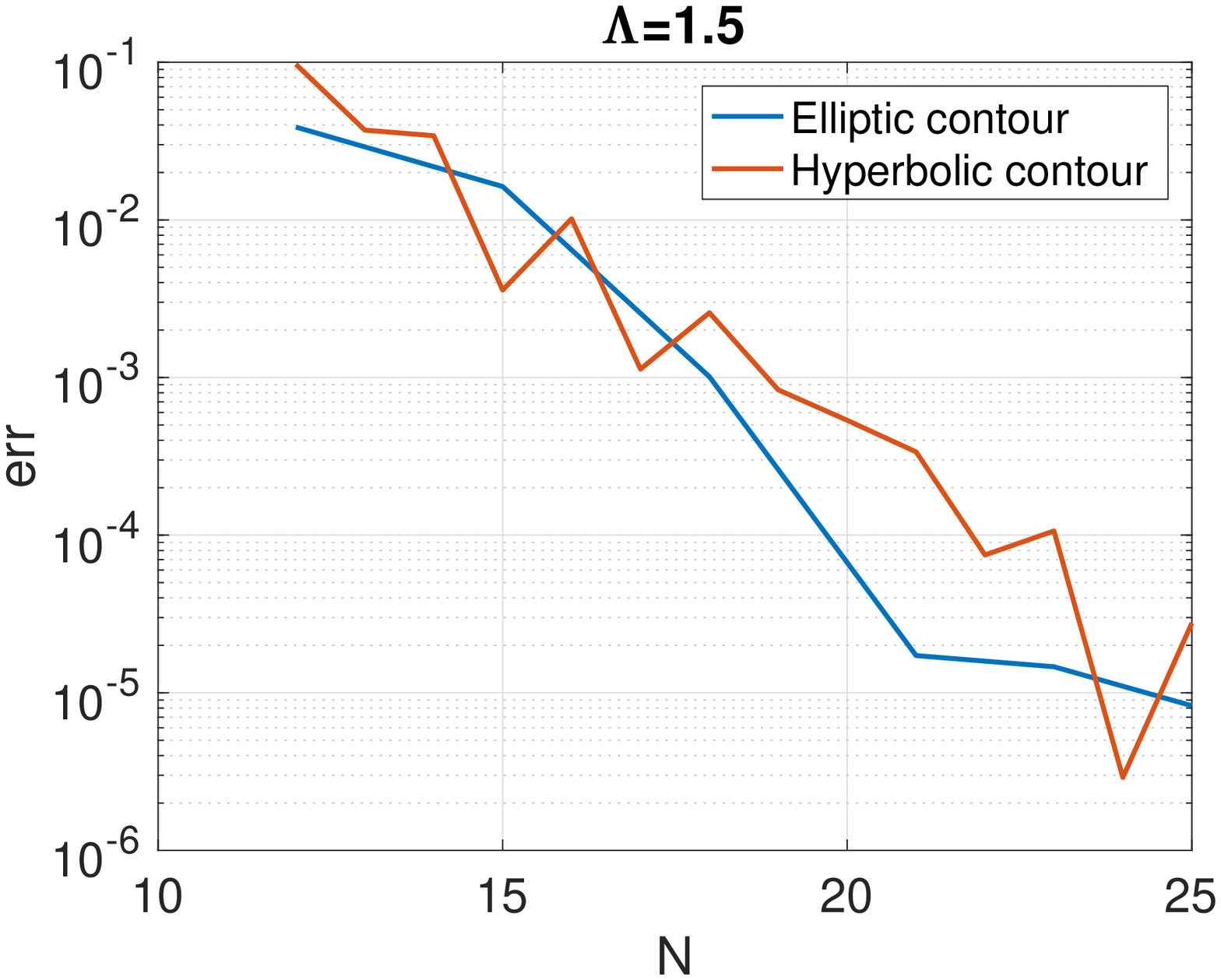}
\end{center}
\caption{Heston equation, comparison between \cite{LPS} and our method for $t=10$ $z_l=-4$, $z_r=0.06$. \textit{Left:} $\Lambda=1$. \textit{Right: }$\Lambda=1.5$.}\label{compHLP10}
\end{figure}

% \newpage

\section{Extension to the case of time intervals}\label{timInt}
We notice that the most expensive computation when evaluating \eqref{integrandFunction} is the inversion of the matrix $zI-A$ at the quadrature nodes. This inversion does not involve the time $t$ that appears only in the exponential part. For this reason, a great improvement to the efficiency of the method comes from the possibility of using a unique integration contour for a whole time interval $[t_0,t_1]$. In this Section we suggest a strategy for computing a unique profile of integration, uniquely defined by the parameter $a$ by \eqref{gammaActual}, \eqref{a1}, \eqref{a2}, \eqref{C}. By doing that, for a general time $t\in[t_0,t_1]$, we just need to compute the corresponding truncation parameter $c$ and the constant $K$ in \eqref{firstEquation}, \eqref{secondEquation}.
We would like to get an uniform error estimate like \eqref{estimateErr2} for the whole interval $[t_0,t_1]$. Recalling \eqref{errOpt} and the fact that $c\leq\frac{1}{2}$, we estimate
\begin{equation}\label{estIntervals}
\sup_{t\in[t_0,t_1]}\left\|I_N-I\right\|\lesssim \pi e^{D(a)t_1}e^{-2aN}\,.
\end{equation}
 Using estimate \eqref{estIntervals}, we recover the (theoretical) value of quadrature nodes sufficient to reach a prescribed precision $tol$. In particular, we get
\begin{equation}\label{nInt}
N=\frac{1}{2a}\left(D(a)t_1-\log\left(\frac{tol}{\pi }\right)\right)\,.
\end{equation}Fix $t=t_0$. We construct $\Gamma_+$ as explained in Section \ref{constrEll} for the time $t_0$ (lower time of the interval). The choice of the smaller time $t_0$ reflects in the setting of the center of the integration ellipse $z_l$ as explained in (i) of page \pageref{choicez}. In this way, we expect the contribution of the two half-lines in \eqref{Gamma} to be negligible for all the times $t\in[t_0,t_1]$. Application of the construction of Section \ref{constrEll} gives the parameters $z_l,z_r,d+\opi r$ defining uniquely $\Gamma_+$. At this point, we minimize the function
\begin{equation}\label{fToMIn2}
	f(a)=\frac{1}{2a}\left(D(a)t_1-\log\left(\frac{tol}{\pi}\right)\right)
\end{equation}where $D(a)$ is given by \eqref{D}. We end up with the optimal $a$ defining uniquely the profile of integration by \eqref{gammaActual}, \eqref{a1}, \eqref{a2}, \eqref{C}. We will use this profile for every time $t\in[t_0,t_1]$.\\
Even if the profile of integration is the same for all times, when $t$ changes we need to truncate it in a different point. For a general $t\in[t_0,t_1]$ we can compute the corresponding  values $c_t,K_t$ using Algorithm \ref{algK} of Section \ref{selConst}. In case we need to evaluate our integral for many times $t$, Algorithm \ref{algK} can be too expensive. Indeed, every iteration of this method requires the evaluation of the resolvent function. To save computational cost, we do as follows:
\begin{itemize}
\item we compute the pairs $(c_0,K_0)$ and $(c_1,K_1)$ corresponding to the times $t_0,t_1$;\\
\item calling $K_t$ the constant \eqref{secondEquation} for the general time $t\in(t_0,t_1)$ we make the assumption that $K_t$ is linear, i.e. we assume that
\begin{equation}\label{KInter}
K_t=K_0+(K_1-K_0)\frac{t-t_0}{t_1-t_0}\,.
\end{equation}The corresponding value of $c_t$ is given by \eqref{c}. This assumption turns out to be effective in the numerical experiments.
\end{itemize}

We show some numerical experiments for both Black-Scholes and Heston equation. Since the case of very large times is not really interesting (because the solution rapidly becomes stationary), we consider the case of intervals of the form $[t_0,\Lambda t_0]$ with $\Lambda=10$. In particular, we make the experiments on the intervals $[0.1,1]$, $[1,10]$. In this way, we approximate the solution on the whole time interval $[0.1,10]$ and the computation is competitive with respect to the classical PDEs integrators. In the plots \ref{intervalBS}, \ref{intervalH}, we show the numerical results for Black-Scholes and Heston equations. The target tolerance we choose is $tol=5\cdot 10^{-8}$ for Black--Scholes and $tol=5\cdot 10^{-4}$ for Heston. We also fix $z_r=0.01$ for Black--Scholes and $z_r=0.06$ for Heston.
\begin{figure}[h!]
\includegraphics[scale=0.4]{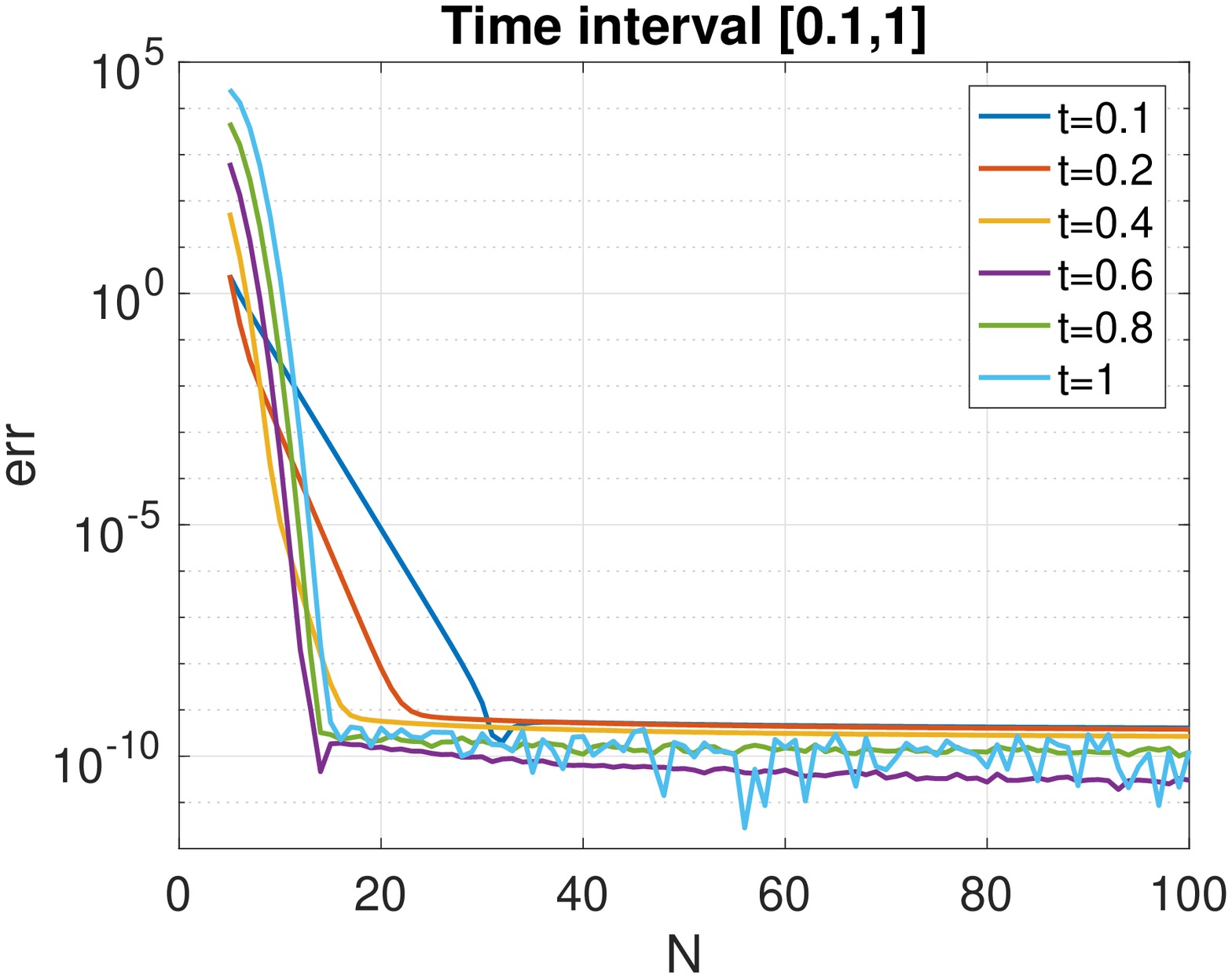}\includegraphics[scale=0.4]{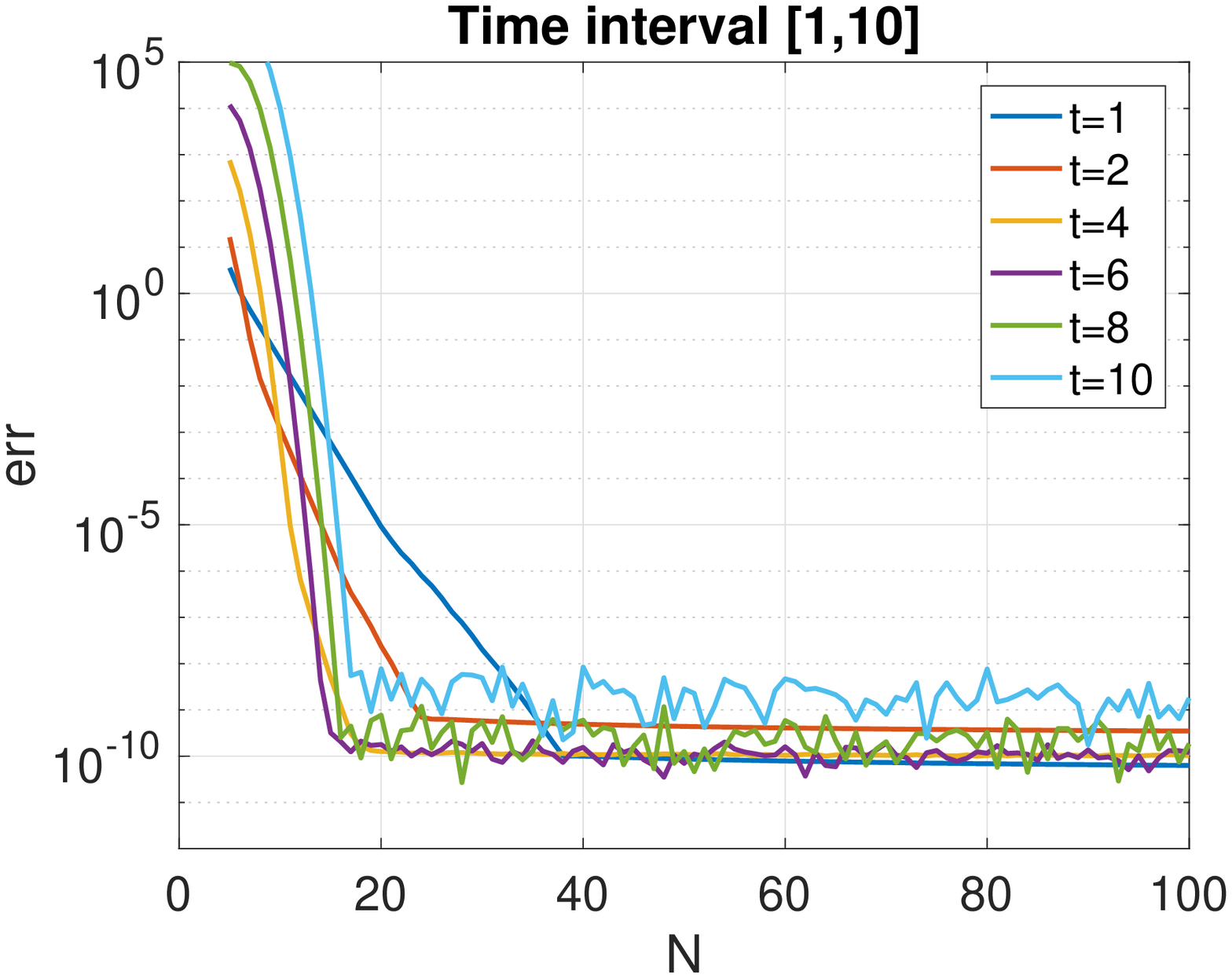}
\caption{Black-Scholes equation, $tol=10^{-8}$. \textit{Left:} time interval $[0.1,1]$, $z_l=-400$, $z_r=0.01$. \textit{Right: }time interval $[1,10]$, $z_l=-40$, $z_r=0.01$.}\label{intervalBS}
\end{figure}
% \newpage
\begin{figure}[h!]
\includegraphics[scale=0.4]{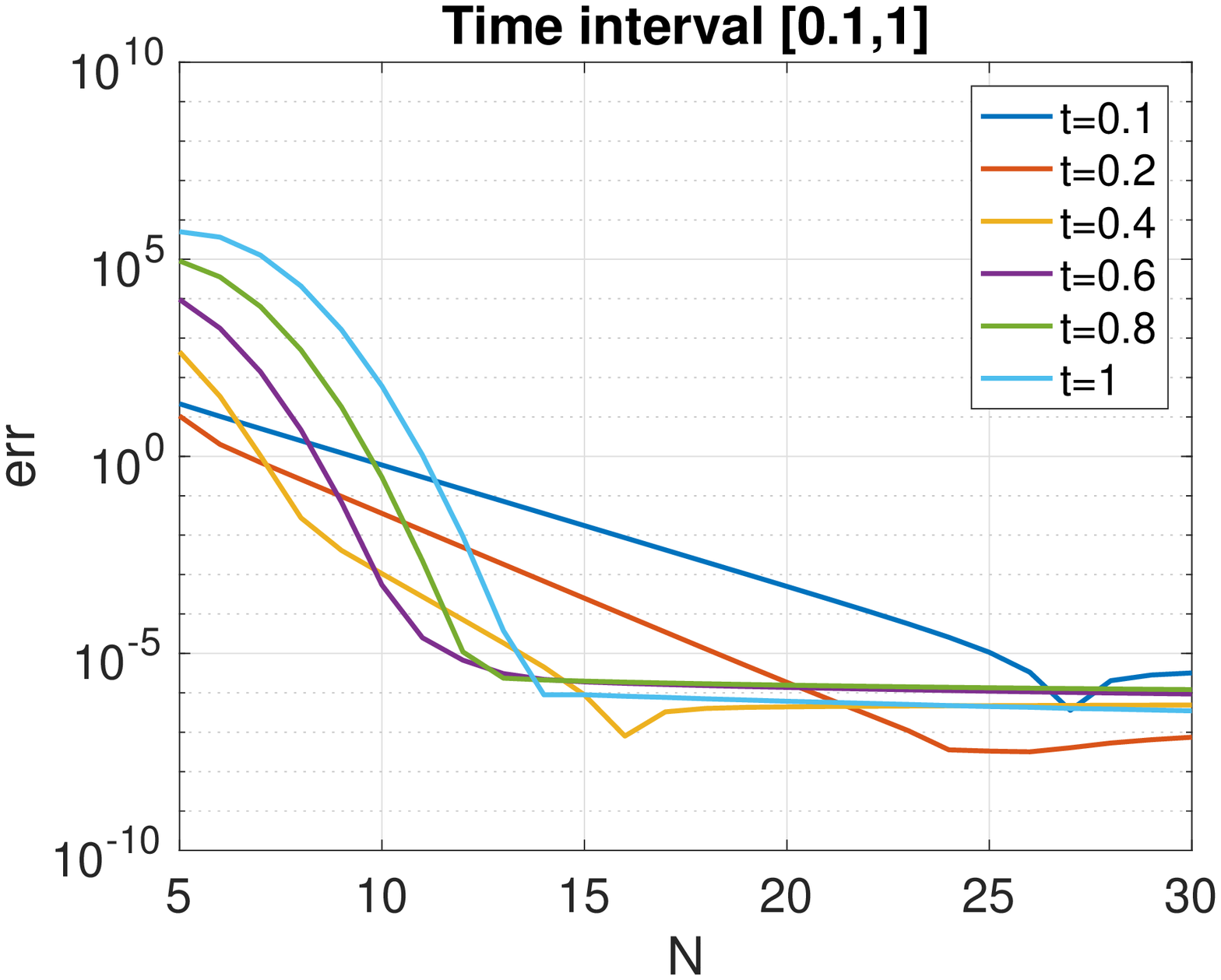}\includegraphics[scale=0.4]{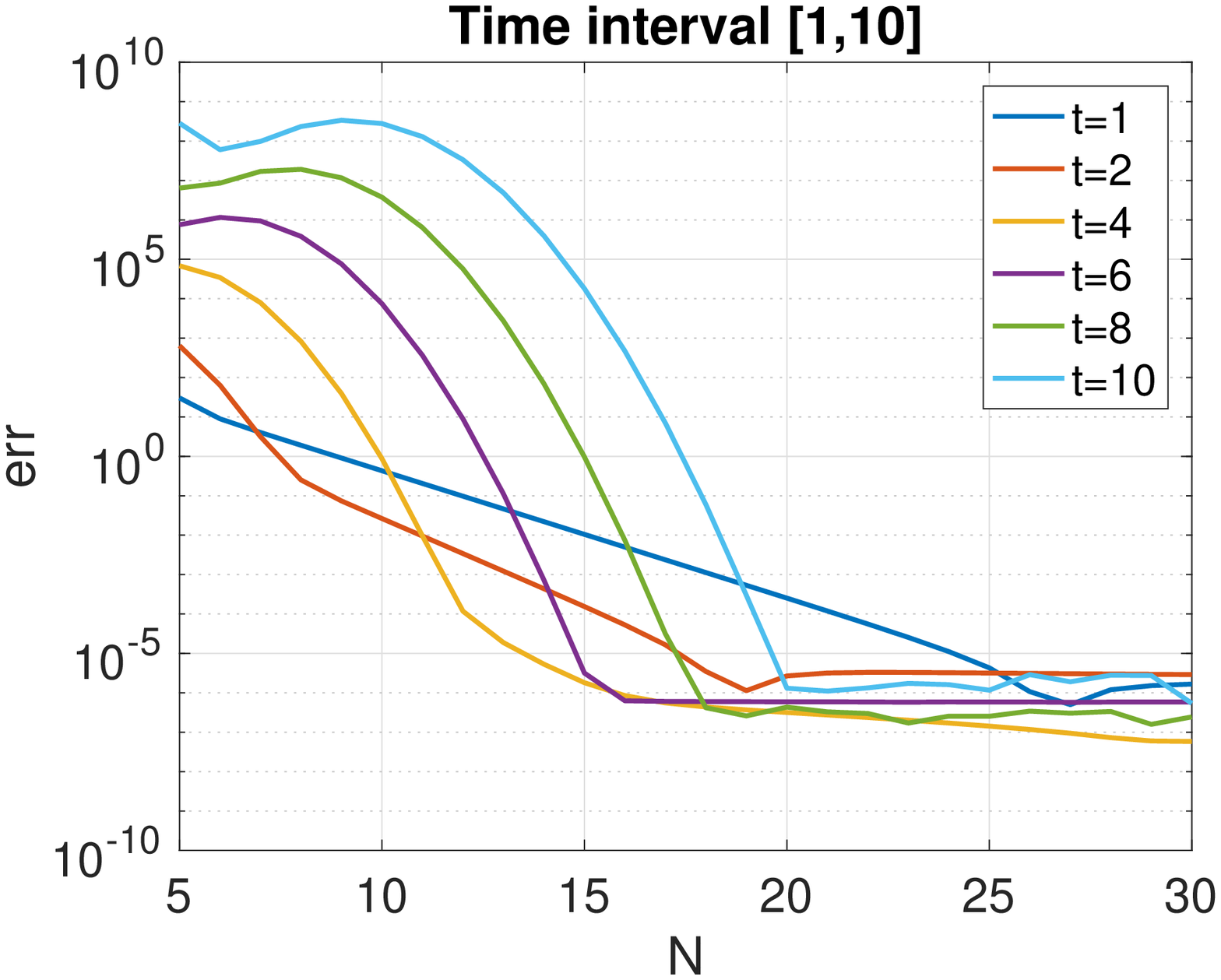}
\caption{Heston equation, $tol=10^{-4}$. \textit{Left:} time interval $[0.1,1]$, $z_l=-400$, $z_r=0.06$. \textit{Right: }time interval $[1,10]$, $z_l=-40$, $z_r=0.06$.}\label{intervalH}
\end{figure}
A slowdown of the convergence rate as $t$ decreases is observed. This is due to the fact that $c$ is decreasing w.r.t. time and the rate of convergence is ${\cal O}(e^{-\frac{a}{c}N})$. It is interesting to compare those performances with the one obtained by \cite{LPS} since this method is also conceived to work on time intervals. In Figures \ref{intervalBSLP}, \ref{intervalHLP} the results are plotted (in both cases we take $\alpha=0.4\,,d=0.4$)
\begin{figure}[h!]
\includegraphics[scale=0.27]{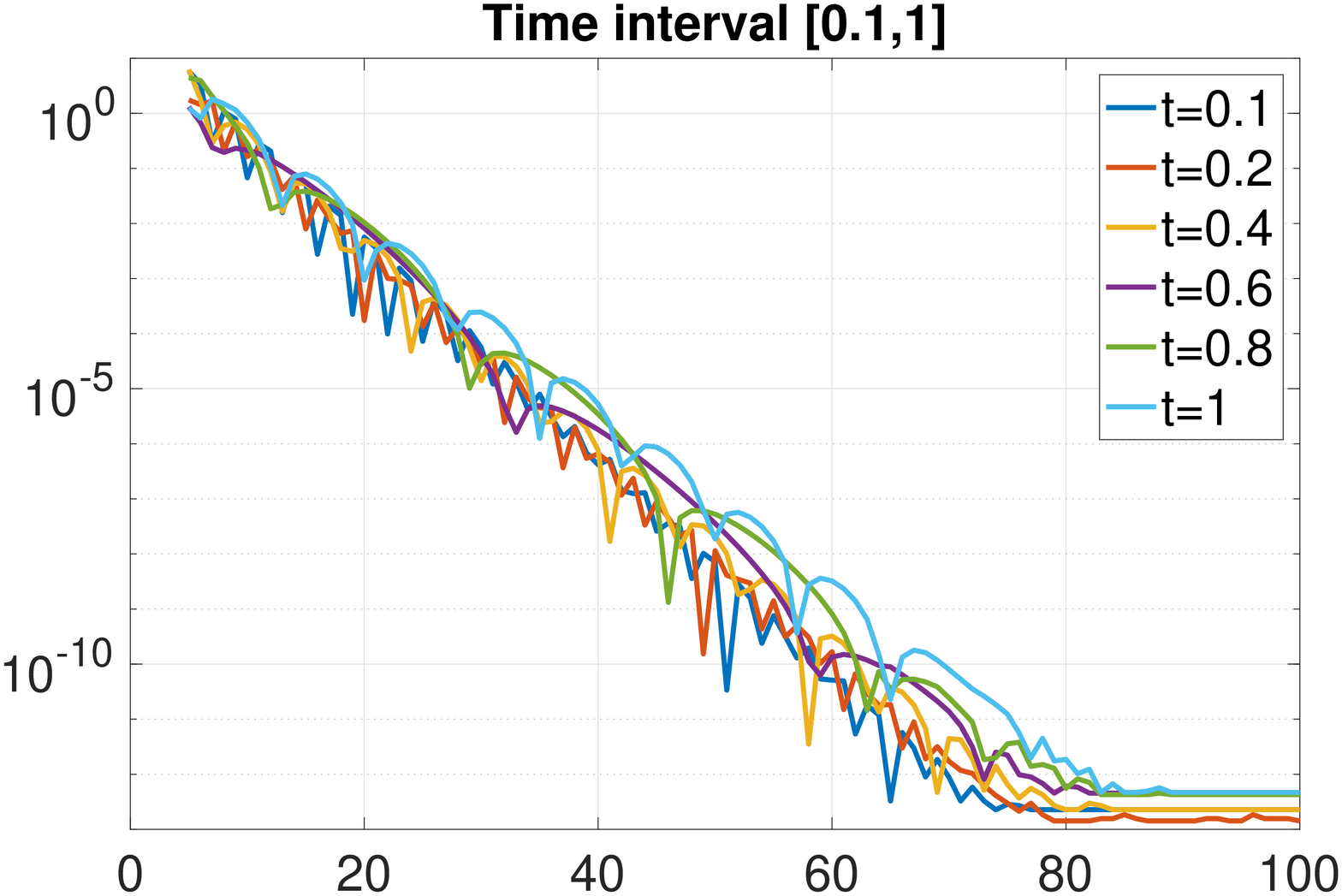}\includegraphics[scale=0.27]{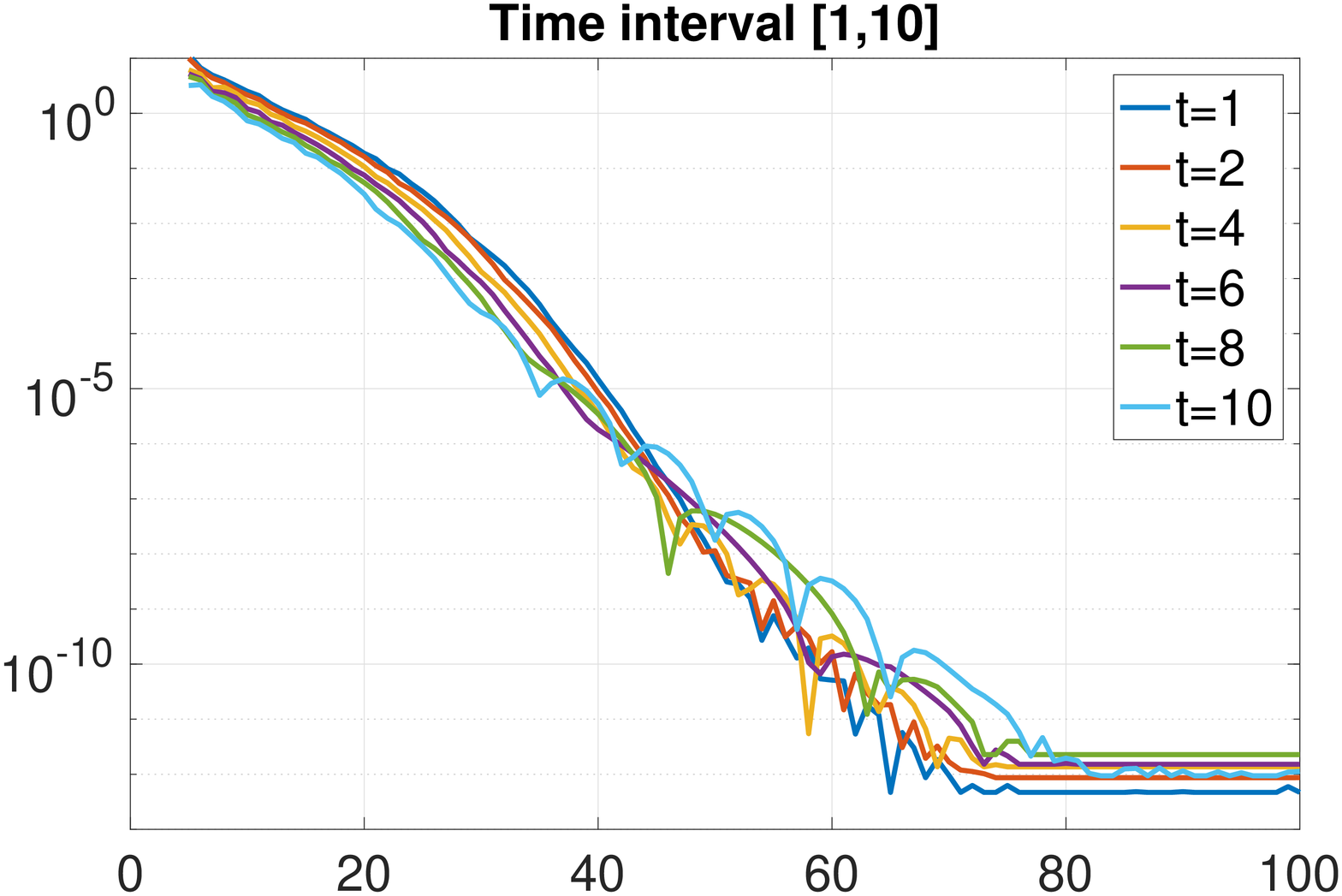}
\caption{Black-Scholes equation using \cite{LPS}. \textit{Left: }time interval $[0.1,1]$. \textit{Right: }time interval $[1,10]$.}\label{intervalBSLP}
\end{figure}
\begin{figure}[h!]
\includegraphics[scale=0.27]{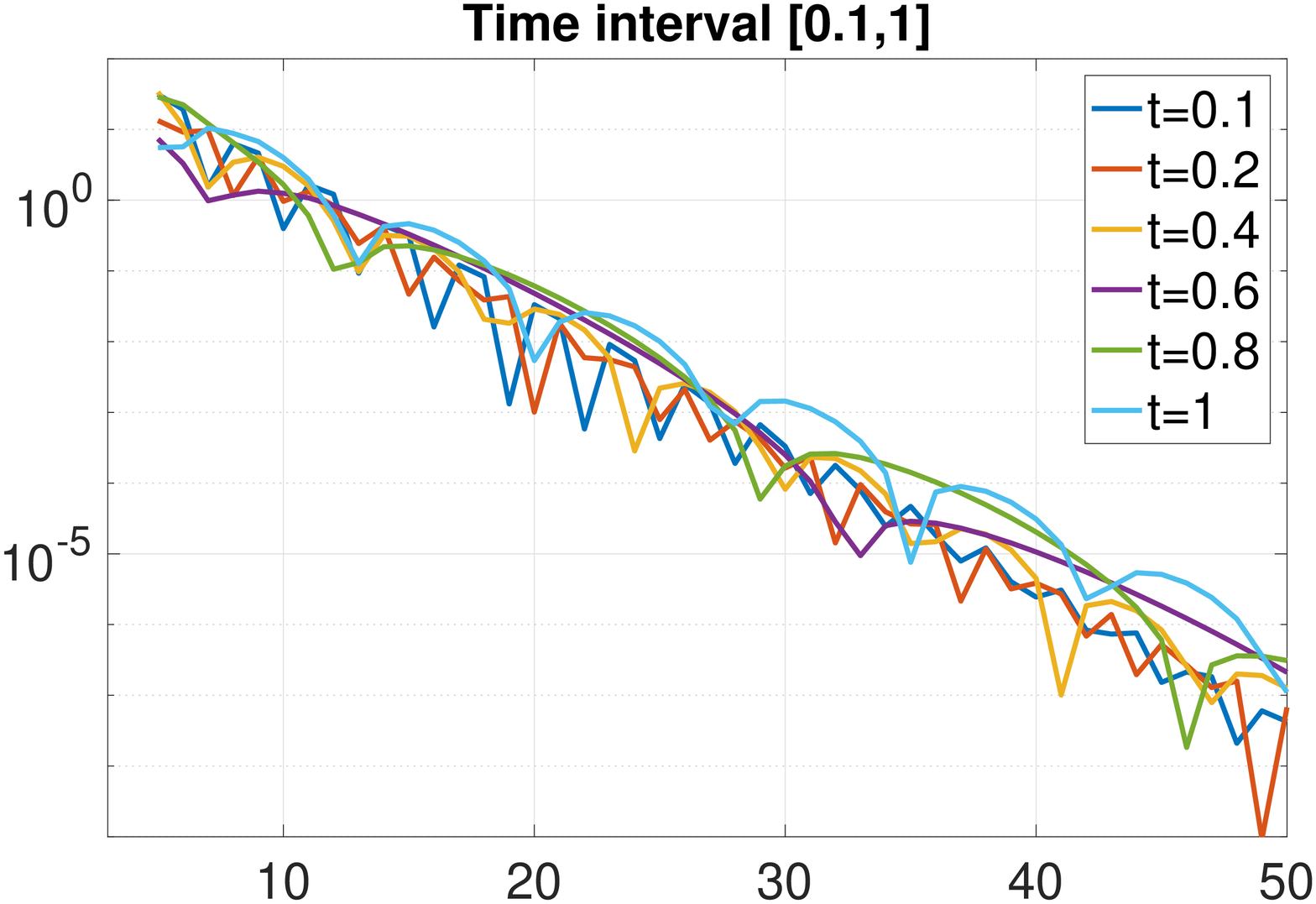}\includegraphics[scale=0.27]{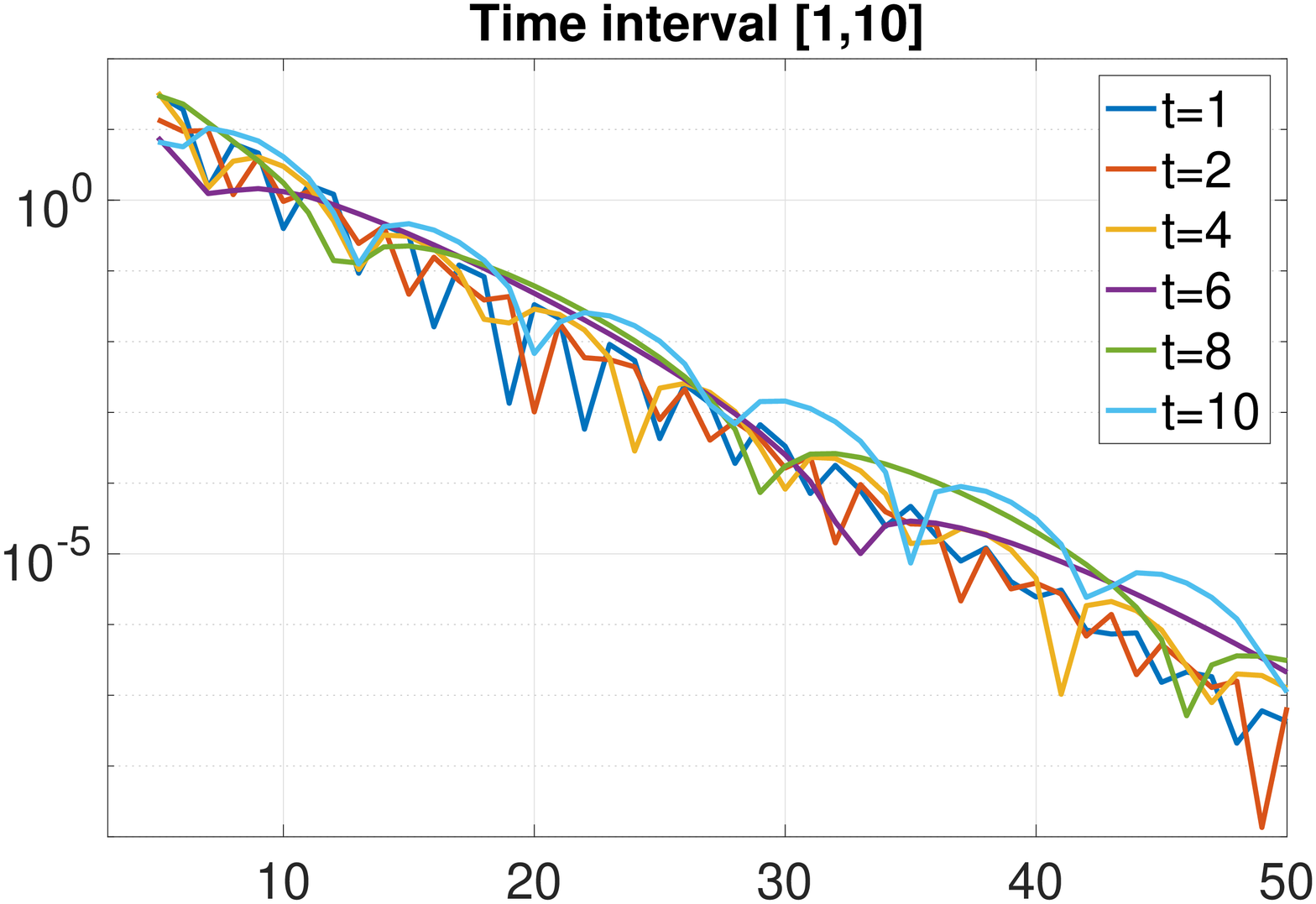}
\caption{Heston equation using \cite{LPS}. \textit{Left: }time interval $[0.1,1]$. \textit{Right: }time interval $[1,10]$.}\label{intervalHLP}
\end{figure}
% \newpage
\section{Conclusions}
In this paper we have proposed a new method for the numerical inversion of the Laplace transform of functions with specific properties arising in the space-time approximation of linear convection-diffusion equations.

Our method is based on a preliminary investigation of some pseudospectral level sets of $A$. In this way, the method can be directly applied to any linear system of ODEs with constant coefficient matrix and no other \textit{a priori} information about the discrete operator $A$ is needed. This first step is not considered sistematically in \cite{ITHWeid,LP,LPS,W}. In our applications, the computation of the pseudospectral level curves is performed by \texttt{eigtool}. The computational cost of the approximation made by using \texttt{eigtool} is reported in Subsection \ref{compCost}, where we also show that a low resolution in this approximation might be enough to construct a good integration contour.

We recap the main advantages of our method:
\begin{itemize}

%\item Our strategy is based on the approximation of the pseudospectral level curves associated to the discrete operator $A$ in \eqref{mainpb} by using {\tt eigtool}. This construction is quite general and does not require a priori knowledge about the spectrum of $A$. The computational cost of the approximation made by using \texttt{eigtool} is reported in Subsection \ref{compCost}, where we also show that a low resolution in this approximation might be enough to construct a good integration contour.

\item[(i) ] It is designed in order to achieve a prescribed precision as fast as possible.

\item[(ii) ] It is stable: adding quadrature nodes never deteriorates the quality of the approximation. As shown in \cite{LPS} and \cite{W}, this can be a delicate issue in the numerical inversion of the Laplace transform.

The stability constant of the method can be computed to carry out an a priori feasibility check to detect if the prescribed accuracy is too high.

\item[(iii) ] It is easily adapted to approximate the solution to \eqref{mainpb} on relatively large time intervals of the form $[t_0,\Lambda t_0]$, with $\Lambda >1$.

\item[(iv) ] Once the target accuracy $tol$ and the time $t$ are fixed, our algorithm selects the profile of integration independently of the number of quadrature nodes. Thus, the cost of adding quadrature nodes to reach the target accuracy is low in comparison to the algorithms in \cite{ITHWeid,LPS}, where the integration contour does depend on the number of quadrature nodes.

\end{itemize}

Future research will be devoted to reduce the dependance on {\tt eigtool}, which can be prohibitively expensive for large matrices arising from the spatial discretization of 2D and, specially, 3D convection-diffusion equations. The resolution of the linear systems with non-normal matrices $zI-A$ should also be more carefully studied.

\section*{Acknowledgments}

The authors thank K. J. in 't Hout for providing the codes implementing the method in \cite{ITHF}.
NG and MLF thank INdAM GNCS for financial support. MLF also acknowledges the support of the Spanish grant MTM2016-75465 and the Ram\'on y Cajal program of the Ministerio de Economia y Competitividad, Spain.


\begin{thebibliography}{1}
\bibitem{BLS} L. Banjai, M. L\'{o}pez-F\'{e}rnandez, A. Sch\"{a}dle, \textit{Fast and oblivious algorithms for dissipative and two-dimensional wave equations}, Siam J. Numer. Anal., Vol. 55, pp. 621-639, 2017.

\bibitem{BS} F. Black, M. Scholes, \textit{The pricing of options and corporate liabilities}, J. Polit. Econ., Vol. 81, pp. 637-654, 1973.

\bibitem{DW} B. Dingfelder, J. A. C. Weideman, \textit{An improved Talbot method for numerical Laplace transform inversion}, Numer. Algorithms, Vol. 68, pp. 167-183, 2015.

\bibitem{GaMa} I. P. Gavrilyuk, V. L. Makarov, {\em  Exponentially convergent parallel discretization methods
for the first order evolution equations}, Comput. Methods Appl. Math., Vol. 1, Number 4, pp.~333-355, 2001.

\bibitem{He} S. L. Heston, \textit{A closed-form solution for options with stochastic volatility with applications to bond and currency options}, Rev. Finan. Stud., Vol. 6, pp. 327-343, 1993.

\bibitem{Hu} J. C. Hull, \textit{Options, Futures and Other Derivatives}, 6th ed., Prentice Hall, New Jersey, 2006.

\bibitem{ITHF} K. J. in 'T Hout, S. Foulon, \textit{ADI diference schemes for option pricing in the Heston model with correlation}, Int. J. Numer. Anal. Model., Vol. 7, Number 2, pp. 303-320, 2010.

\bibitem{ITH} K. J. in 't Hout, \textit{ADI schemes in the numerical solution of the Heston PDE}, Numerical Analysis and Applied Mathematics, eds. T. E. Simos et.al., AIP Conf. Proc: 936, 2007.

\bibitem{ITHW} K. J. in 't Hout, B. D. Welfert, \textit{Stability of ADI schemes applied to convection-diffusion equations with mixed derivative terms}, Appl. Numer. Math. Vol. 57, pp. 19-35, 2007.

\bibitem{ITHW2} K. J. in 't Hout, B. D. Welfert, \textit{Unconditional stability of second -order ADI schemes applied to multi-dimensional diffusion equations with mixed derivative terms}, Appl. Numer. Math., Vol. 59, pp. 677-692, 2009.

\bibitem{ITHWeid} K. J. in 't Hout, J. A. C. Weideman, \textit{A contour integral method for the Black$\&$Scholes and Heston Equations}, SIAM J. Sci. Comput., Vol. 33, Num. 2, pp. 763-785, 2011.

\bibitem{JavT} M.~Javed, L.~N. Trefethen.
\newblock A trapezoidal rule error bound unifying the {E}uler-{M}aclaurin
  formula and geometric convergence for periodic functions.
\newblock {\em Proc. R. Soc. Lond. Ser. A Math. Phys. Eng. Sci.},
  470(2161):20130571, 9, 2014.

\bibitem{LP} M. L\'{o}pez-F\'{e}rnandez, C. Palencia, \textit{On the numerical inversion of the Laplace transform of certain holomorphic mappings}, Appl. Numer. Math., Vol. 51, pp. 289-303, 2004.

\bibitem{LPS} M. L\'{o}pez-F\'{e}rnandez, C.Palencia, A. Sch\"{a}dle, \textit{A spectral order method for inverting sectorial Laplace transform}, SIAM J. Numer. Anal., Vol. 44, Num. 3, pp. 1332–1350, 2006.

\bibitem{RT} S. C. Reddy, L. N. Trefethen, \textit{Pseudospectra of the convection-diffusion operator}, SIAM J. Appl. Math., Vol. 54, pp. 1634-1649, 1994.

\bibitem{SST} D. Sheen, I. Sloan, V. Thom\' ee, {\em A parallel method for time discretization of parabolic equations
based on Laplace transformation and quadrature}, IMA J. Numer. Anal., Vol. 23, pp. 269-299, 2003.

				
\bibitem{T} A. Talbot, \textit{The Accurate Numerical Inversion of Laplace Transforms}, IMA J. Appl. Math., Vol. 23, pp. 97-120, 1979.

\bibitem{TreE} L. N. Trefethen, M. Embree, Spectra and Pseudospectra: The Behaviour of Nonnormal Matrices and
Operators. Princeton, NJ: Princeton University Press, 2005.

\bibitem{TW} L. N. Trefethen, J. A. C. Weideman, \textit{The Exponentially Convergent Trapezoidal Rule}, SIAM Rev., Vol. 56, Number 3, pp. 385-458, 2014.
\bibitem{W} J. A. C. Weideman, \textit{Improved contour integral methods for parabolic PDEs}, IMA J. Numer. Anal., Vol. 30, pp. 334-350, 2010.

\bibitem{eigtool} Thomas G. Wright. EigTool. http://www.comlab.ox.ac.uk/pseudospectra/eigtool/, 2002.

\end{thebibliography}
\end{document}